\documentclass[12pt]{amsart}
\usepackage{amsfonts,latexsym,amsmath, amssymb}
\usepackage{mathrsfs,MnSymbol}
\usepackage{url,color}
\usepackage{upgreek}
\usepackage{fancyhdr}
\usepackage{hyperref}
\usepackage{times}
\usepackage{scalefnt}
\usepackage{url}
\usepackage{cite}

\newcommand{\bea}{\begin{eqnarray}}
\newcommand{\eea}{\end{eqnarray}}
\def\beaa{\begin{eqnarray*}}
\def\eeaa{\end{eqnarray*}}
\def\ba{\begin{array}}
\def\ea{\end{array}}
\def\be#1{\begin{equation} \label{#1}}
\def \eeq{\end{equation}}

\def\be{{\beta}}

\newtheorem{theorem}{Theorem}[section]
\newtheorem{lemma}[theorem]{Lemma}
\newtheorem{proposition}[theorem]{Proposition}
\newtheorem{corollary}[theorem]{Corollary}
\newtheorem{definition}[theorem]{Definition}
\newtheorem{remark}[theorem]{Remark}

\numberwithin{equation}{section}

\numberwithin{equation}{section}

\topmargin - .23 in

\begin{document}

\title{On scattering for the defocusing quintic nonlinear Schr\"odinger equation on the two-dimensional cylinder}
\author{Xing Cheng$^{*}$, Zihua Guo$^{**}$, and Zehua Zhao$^{***}$}

\thanks{$^*$ College of Science, Hohai University, Nanjing 210098,\ Jiangsu,\   China and School of Mathematical Sciences, Monash University, \ VIC 3800, \ Australia. \texttt{Xing.Cheng@monash.edu, chengx@hhu.edu.cn}}

\thanks{$^{**}$ School of Mathematical Sciences, Monash University, \ VIC 3800,\  Australia.     \texttt{Zihua.Guo@monash.edu}  }

\thanks{$^{***}$ Department of Mathematics, Johns Hopkins University, Baltimore, MD 21218.   \texttt{zzhao25@jhu.edu} }

\thanks{$^{*}$ Xing Cheng has been partially supported by the
NSF grant of China (No. 11526072).}	

\thanks{$^{**}$ Zihua Guo was supported by Monash University New Staff Grant and the ARC project (No. DP170101060).}

\begin{abstract}
	In this article, we prove the scattering for the quintic defocusing nonlinear Schr\"odinger equation on cylinder $\mathbb{R} \times \mathbb{T}$ in $H^1$. We establish an abstract linear profile decomposition in $L^2_x h^\alpha$, $0 <  \alpha \le 1$, motivated by the linear profile decomposition of the mass-critical Schr\"odinger equation in $L^2(\mathbb{R}^d )$, $d\ge 1$. Then by using the solution of the one-discrete-component quintic resonant nonlinear Schr\"odinger system, whose scattering can be proved by using the techniques in $1d$ mass critical NLS problem by B. Dodson, to approximate the nonlinear profile, we can prove scattering in $H^1$ by using the concentration-compactness/rigidity method. As a byproduct of our proof of the scattering of the one-discrete-component quintic resonant nonlinear Schr\"odinger system, we also prove the conjecture of the global well-posedness and scattering of the two-discrete-component quintic resonant nonlinear Schr\"odinger system made by Z. Hani and B. Pausader [Comm. Pure Appl. Math. {\bf67} (2014)].
\bigskip

\bigskip


\end{abstract}
\maketitle

\section{Introduction}
In this article, we consider the quintic defocusing nonlinear Schr\"odinger equation posed on the cylinder $\mathbb{R} \times \mathbb{T}$, which is also called waveguide:
\begin{equation}\label{eq1.1}
\begin{cases}
i\partial_t u + \Delta_{\mathbb{R} \times \mathbb{T}}  u = |u|^4 u,\\
u(0) = u_0\in H^1(\mathbb{R}\times \mathbb{T}),
\end{cases}
\end{equation}
where $\Delta_{\mathbb{R} \times \mathbb{T}} $ is the Laplace-Beltrami operator on $\mathbb{R} \times \mathbb{T}$, and $u : \mathbb{R} \times \mathbb{R} \times \mathbb{T}  \to \mathbb{C}$ is a complex-valued function. The equation has the symplectic form $\Im \int_{\mathbb{R}_x \times \mathbb{T}_y} u(x,y) \overline{v(x,y)} \,\mathrm{d}x \mathrm{d}y$ on the Hilbert space $L^2(\mathbb{R}\times \mathbb{T})$.

Equation \eqref{eq1.1} has the following conserved quantities:
\begin{align*}
\text{mass: }    &\quad
{M}(u(t))  = \int_{\mathbb{R}_x  \times \mathbb{T}_y} |u(t,x,y)|^2\,\mathrm{d}x\mathrm{d}y,\\
\text{   energy:  }     &  \quad
{E}(u(t))  = \int_{\mathbb{R}_x  \times \mathbb{T}_y} \frac12 |\nabla u(t,x,y)|^2  + \frac16 |u(t,x,y)|^6 \,\mathrm{d}x\mathrm{d}y,\\
\text{ momentum: } &  \quad
{P}(u(t)) = \Im \int_{\mathbb{R}_x \times \mathbb{T}_y} \overline{u}(x,y,t) \nabla u(x,y,t)\,\mathrm{d}x\mathrm{d}y.
\end{align*}
Equation \eqref{eq1.1} is a special case of the general nonlinear Schr\"odinger equation on the waveguide $\mathbb{R}^d \times \mathbb{T}^m$:
\begin{equation}\label{eq1.2n}
\begin{cases}
i\partial_t u + \Delta_{\mathbb{R}^d \times \mathbb{T}^m}  u = |u|^{p-1} u,\\
u(0) = u_0\in H^1(\mathbb{R}^d \times \mathbb{T}^m),
\end{cases}
\end{equation}
where $1 < p < \infty$, $m,d \in \mathbb{Z}$, and $m, d \ge 1$. This kind of equation in lower dimensions describes wave propagation in nonlinear and dispersive media. For instance, it can describe the nonlinear dynamics of superfluid films for which $u$ is the condensate wave function related to the film thickness and to the superfluid velocity. It also figures in the time-dependent Landau-Ginzburg model of phase transitions, in this case the wave function $u$ is a complex order parameter. Another phenomenon governed by the same equation is the propagation of slow varying electromagnetic wave envelopes in a plasma \cite{S,SL}. Other applications concern hydrodynamics and nonlinear optics.

We are interested in the range of $p$ for well-posedness and scattering of \eqref{eq1.2n} on $\mathbb{R}^d \times \mathbb{T}^m$.
On one hand, when considering the well-posedness of the Cauchy problem \eqref{eq1.2n}, intuitively, it is determined by the local geometry of the manifold $\mathbb{R}^d \times \mathbb{T}^m$. The manifold is locally $\mathbb{R}^d \times \mathbb{R}^m$. So we believe the well-posedness is the same as the Euclidean case, that is when $1<  p \le 1 + \frac4{m+d -2}$ the well-posedness is expected. Just as the Euclidean case, we say the equation is energy-subcritical when $ 1< p < 1 + \frac4{m+d -2}$, $m,d \ge 1$ and energy-critical when $p = 1 + \frac4{m+d -2}$, $m + d \ge 3$, $m,d \ge 1$.
On the other hand, when considering the scattering of \eqref{eq1.2n}, scattering is expected to be determined by the asymptotic volume growth of a ball with radius $r$ in the manifold $\mathbb{R}^d\times \mathbb{T}^m$ when $r\to \infty$. From the heuristic that linear solutions with frequency $\sim N$ initially localized around the origin will disperse at time t in the ball of radius $\sim N t$, scattering is expected to be partly determined by the asymptotic volume growth of balls with respect to their radius. Since $\inf_{ z \in \mathbb{R}^d \times \mathbb{T}^m } \text{Vol}_{\mathbb{R}^d \times \mathbb{T}^m} (B(z,r)) \sim r^d, \text{ as } r \to \infty$, the linear solution is expected to decay at a rate $\sim t^{-\frac{d}2}$ and based on the scattering theory on $\mathbb{R}^d$,
the solution of \eqref{eq1.2n} is expected to scatter for $ p \ge 1 + \frac4d$. Moreover,
scattering in the small data case is expected for $ 1 + \frac2d < p < 1 + \frac4d$ when $d \ge 1$. Modified scattering in the small data case
is expected for $p=3$ when $d=1$.

 Therefore, regarding heuristic on the well-posedness and scattering, we expect the solution of \eqref{eq1.2n} globally exists and scatters in the range $ 1 + \frac4d \le p \le 1 + \frac4{m+d-2}$. For $ 1 + \frac2d < p < 1 + \frac4d$ when $d \ge 1$, scattering is expected as in the Euclidean space case for small data. For $p = 3$ when $d =1$, modified scattering is expected as in the Euclidean space case for small data case. This heuristic was justified in \cite{CGYZ,TT,HTT1,IP1,HPTV,TV2,HP,Z1,Z2}.

 For other nonlinear dispersive equations on the waveguides, we refer to the work of L. Hari and N. Visciglia \cite{HV, HV1} on the small data scattering of the energy-subcritical and critical Klein-Gordon equations, and also the recent work of L. Forcella and L. Hari \cite{FH} on the large data scattering of the nonlinear Klein-Gordon eqations on $\mathbb{R}^d\times \mathbb{T}$, for $1\le d \le 4$ in the $H^1$-subcritical case, and the references therein.
We refer to \cite{MP} on the well-posedness of the Zakharov-Kuznetsov equations on $\mathbb{R}\times \mathbb{T}$.

Our main result addresses the scattering for \eqref{eq1.1} in $H^1(\mathbb{R}  \times \mathbb{T})$:
\begin{theorem}[Scattering in $H^1(\mathbb{R}  \times \mathbb{T})$] \label{th1.3}
For any initial data $u_0\in H^1(\mathbb{R}  \times \mathbb{T})$, there exists a global solution $u$ which scatters in the sense that there exist $u^\pm \in H^1(\mathbb{R} \times \mathbb{T})$ such that
\begin{equation*}
\left\|u(t)- e^{it\Delta_{\mathbb{R} \times \mathbb{T}}} u^\pm \right\|_{H^1(\mathbb{R} \times \mathbb{T})} \to 0,  \text{  as }  t\to \pm \infty.
\end{equation*}
\end{theorem}
\noindent Our argument can be applied to the more general Schr\"odinger type model
\begin{align*}
\begin{cases}
i\partial_t u + \Delta_x u - (-\Delta_y )^\alpha u  = |u|^4 u,\\
u(0,x,y ) = u_0(x,y) \in H^1(\mathbb{R}_x \times \mathbb{T}_y),
\end{cases}
\end{align*}
where $ \frac12 < \alpha \le 1$, this kind of model \cite{XH} is motivated by the half wave equation in the weak turbulence \cite{GG}.  By using the argument in the article, we can show scattering in $H_x^1 L_y^2 \cap L_x^2 H_y^\alpha  (\mathbb{R} \times \mathbb{T}) $.
\subsection{Main ideas.}
 The proof of Theorem \ref{th1.3} is based on the concentration compactness/rigidity method developed by C. E. Kenig and F. Merle \cite{KM,KM1}, and also the work of S. Ibrahim, N. Masmoudi, and K. Nakanishi \cite{IMN} on dealing with the nonlinear dispersive equations without scaling invariant.
We describe now some of the main ideas involved in the proof.
\subsubsection{A global Strichartz estimate and a weak scattering norm.}
For \eqref{eq1.1}, the dispersive effect of the $\mathbb{R} $-component is strong enough, to give a
 global Strichartz estimate \cite{TV,TV2}:
\begin{align}\label{eq1.321}
\left\|e^{it\Delta_{\mathbb{R} \times \mathbb{T}}} f\right\|_{L_{t,x}^6 H_y^1 \cap L_t^6 W_x^{1,6} L_y^2(\mathbb{R}\times \mathbb{R}  \times \mathbb{T})} \lesssim  \left\|f\right\|_{H_{x,y}^{1}},
\end{align}
this global Strichartz estimate is adequate for us to establish the well-posedness theory in $H^1(\mathbb{R}\times \mathbb{T})$. This kind of global Strichartz estimate is established in \cite{TV,TV2}. By the well-posedness and scattering theory, we observe to prove Theorem \ref{th1.3}, we only need to prove the solution satisfies a weaker space-time norm $L_{t,x}^6 H_y^{1-\epsilon_0}$, where $0 < \epsilon_0 < \frac12$ is some fixed number used hereafter, which is Theorem \ref{th2.946}. This fact is first observed in our previous work \cite{CGYZ}, which can be proved by using a bootstrap argument.
\subsubsection{An abstract profile decomposition.}
Since it is enough to prove the solution satisfies a weaker
space-time norm $L_{t,x}^6 H_y^{1-\epsilon_0}$, we need to establish a linear profile decomposition in $H^1(\mathbb{R}  \times \mathbb{T})$ with the remainder asymptotically small in $L_{t,x}^6 H_y^{1-\epsilon_0}$,
which is essentially equivalent to describe the defect of compactness of the embedding
\begin{equation}\label{eq}
e^{it\Delta_{\mathbb{R} \times \mathbb{T}}}: H^{ 1}_{x, y} (\mathbb{R} \times \mathbb{T})\hookrightarrow L_{t,x}^6 H_y^{1-\epsilon_0}(\mathbb{R}\times \mathbb{R}  \times\mathbb{T}),
\end{equation}
we can then establish a linear profile decomposition similar to \cite{HP}. However, our argument is mainly based on the argument to establish the linear
profile decomposition of the Sch\"odinger equation in $L^2(\mathbb{R} )$. In fact, in Subsection \ref{sse3.1}, we give an abstract linear profile decomposition (Theorem \ref{pr3.2v15}), which reveals the defect of compactness of the embedding
 \begin{align*}
 e^{it\Delta_{\mathbb{R}^d}}: L_x^2 h^\alpha(\mathbb{R}^d \times \mathbb{D}) \hookrightarrow L_{t,x}^\frac{2(d+2)}d h^{\alpha - \epsilon_0}(\mathbb{R} \times \mathbb{R}^d \times \mathbb{D}),
 \end{align*}
where $0 < \alpha \le 1$. This abstract linear profile decomposition is a infinite vector version of linear profile decomposition of the mass-critical Schr\"odinger equations, and the proof relies on the proof of the linear profile decomposition of the mass-critical Schr\"odinger equations, especially the bilinear Strichartz estimate of the Schr\"odinger equation on $\mathbb{R}^d$. This kind of linear profile decomposition is of independent interest, it can be used in the proof of the scattering of the vector-valued nonlinear Schr\"odinger system \cite{YZ}, and also can be used in the study of well-posedness and the long time behavior of the nonlinear dispersive equations such as
\begin{align*}
i\partial_t u + H u = F(u),
\end{align*}
where $H  = \Delta_{\mathbb{R}^d} + A$, with $A$ has only the point spectrum. We will not give further application of this kind linear profile decomposition here.

After establish the abstract linear profile decomposition, we can get the linear profile decomposition proposition \ref{pro3.9v23} after working on the Fourier coefficient of the partial Fourier transform of the functions in $H^1(\mathbb{R} \times \mathbb{T})$.
\subsubsection{A normal form type argument.}
The nonlinear profiles are defined to be the solution of the quintic nonlinear Schr\"odinger equation on $\mathbb{R}  \times \mathbb{T}$, with initial data is each profile in the linear profile decomposition. We need to give a good approximation of the large scale case, to attack this obstacle, just as in \cite{CGYZ,HP}, the nonlinear profile can be approximated by applying $e^{it\Delta_{\mathbb{T}}}$ to the rescaling of the solution of the one-discrete-component quintic resonant nonlinear Schr\"odinger system
\begin{equation}\label{eq1.5v54}
\begin{cases}
i\partial_t u_j + \Delta_{\mathbb{R} } u_j = \sum\limits_{(j_1,j_2,j_3,j_4,j_5) \in \mathcal{R}(j) } u_{j_1} \bar{u}_{j_2} u_{j_3}\bar{u}_{j_4}u_{j_5},\\
u_j(0) = u_{0,j},\ j\in \mathbb{Z},
\end{cases}
\end{equation}
 where $\mathcal{R}(j) =
 \left\{ j_1,j_2,j_3,j_4,j_5 \in \mathbb{Z}: j_1-j_2+j_3-j_4+j_5= j, \, |j_1|^2 - |j_2|^2 + |j_3|^2 - |j_4|^2 + |j_5|^2 = |j|^2 \right\}$. In theorem \ref{pr5.9}, we use the stability theory to prove the transformation of the solution of the one-discrete-component quintic resonant nonlinear Schr\"odinger system can approximate the nonlinear profile in the large scale case. To prove the error term is small in the Strichartz space $L_t^\infty L_x^2 H_y^1 \cap L_{t,x}^6 H_y^1$, we need to use a normal form argument.
 The normal form argument was developed in the ODE, see for example \cite{A}, it was developed by \cite{Sh} for the quadratic nonlinear wave and Klein-Gordon equations. See also \cite{GNT,GNT1,GNT2,GMS,GMS1} for further development.

\subsubsection{Scattering for the resonant system.}
 We prove the solution to the one-discrete-component quintic resonant nonlinear Schr\"odinger system \eqref{eq1.5v54} globally exists and scatters in
 $L^2_x h_j^1(\mathbb{R}\times \mathbb{Z})$, that is Theorem \ref{th1.2}. The main idea to solve the scattering problem for the quintic resonant system is to generalize the machinery built in the one dimensional mass-critical NLS problem (see \cite{D2}) from the equation case to the system case. The global skeleton of the proof would still be the classical Concentration compactness/Rigidity method established in \cite{KM,KM1} and the crucial techniques are long time Stricharz estimate and frequency-localized interaction Morawetz estimate (see \cite{D3,D2,D1}. Since the nonlinearity of the system involves resonant relations, the first step is to establish the nonlinear estimate to handle the nonlinearity which is fundamental. Based on this, following standard method and using the abstract linear profile decomposition developed by us in Subsection \ref{sse3.1}, the small data scattering result can be obtained and if the scattering statement does not hold, we can get an almost-periodic solution and it suffices to exclude the almost-periodic solution. Firstly, in theorem \ref{thm4.1}, we prove the long time Strichartz estimate in the $\tilde{X}_{k_0}$ space
defined in definition \ref{de4.24v13}. The long time Strichartz estimate is developed by B. Dodson \cite{D3,D1,D2}, which relies on the bilinear Strichartz estimate heavily. Secondly, we preclude the almost periodic solution in the two different scenarios regarding the scaling function $N(t)$, that is the rapid frequency cascade and quasi-soliton cases in Subsection \ref{subse4.4}. The rapid frequency cascade case can be proved to have higher regularity, then together with the conservation of energy, we can then exclude this case. While for the quasi-soliton case, we will use the frequency localized interaction Morawetz estimate in theorem \ref{th4.28v43}, together with the long time Strichartz estimate to control the error terms, to exclude this case. As a byproduct of the proof of our scattering theorem of the one-discrete-component quintic resonant nonlinear Schr\"odinger system, we find the argument in the proof can be applied to prove the global well-posedness and scattering of the two-discrete-component quintic resonant nonlinear Schr\"odinger system with little modification. Therefore, we have solved the conjecture in Z. Hani and B. Pausader's article (Conjecture 1.2 in Page 1470 of \cite{HP}). This together with \cite{HP} completely proves the global well-posedness and scattering of the quintic nonlinear Schr\"odinger equation on $\mathbb{R}\times \mathbb{T}^2$.

Compared to the cubic resonant system problem in \cite{YZ}, the quintic resonant relation is more complicated and we need to take care the nonlinearity more delicately. Remarkably, first we reduce the scattering norm from $L^6_{x,t}h^1$ to $L^6_{x,t}h^{\beta}$ ($\frac{3}{8}< \beta <1$) by establishing nonlinear estimate to make it possible to use profile decomposition to obtain the almost periodic solution. For the 2d cubic resonant system problem (see \cite{YZ}), since the resonance is less complicated, the authors can even reduce the the scattering norm from $L^4_{t,x}h^1$ to $L^4_{t,x}l^2$. Since our goal is to show the scattering norm of the solution is finite and a smaller quantity is of course more likely to be finite, it is easier for us to reduce the scattering norm to a smaller one. Another disadvantage thing for our problem is that our scattering norm involves regularity of the discrete direction, which destroys the symmetric structure. Thus, for this problem, we have to deal with the almost periodic solution more delicately. On one hand, we need to establish the long time Stricharz estimate for the quintic resonant system considering the regularity of the discrete direction. On the other hand, when we use interaction Morawetz identity, we consider the case without regularity for the purpose of using the symmetric properties to make the interaction Morawetz identity positive definite. One advantage for this problem is, similar to the mass-critical nonlinear Schr\"odinger problems, the quintic nonlinearity is considered `better' or `more friendly' than the cubic nonlinearity since there are `more' terms which gives more room for us to play techniques.

The paper is organized as follows. After introducing some notations and preliminaries, we give the well-posedness theory and small data scattering in Section \ref{se2}. We also give the stability theory in this section. In Section \ref{se4}, we derive the linear profile decomposition for data in $H^{ 1}(\mathbb{R}  \times \mathbb{T})$ and approximate the nonlinear profiles by using the solution of the one-discrete-component quintic resonant nonlinear Schr\"odinger system. In Section \ref{se6}, we prove the scattering of the one-discrete-component quintic resonant nonlinear Schr\"odinger system. Then, we reduce the non-scattering of the quintic nonlinear Schr\"odinger equation on $\mathbb{R}\times \mathbb{T}$ in $H^1$ to the existence of an critical element and show the extinction
of the critical element in Section \ref{se5}.
\subsection{Notation and Preliminaries}
We will use the notation $X\lesssim Y$ whenever there exists some constant $C>0$ so that $X \le C Y$. Similarly, we will use $X \sim Y$ if
$X\lesssim Y \lesssim X$. We will use the notation $\mathcal{O}(X)$ to denote a quantity the resembles $X$, that is, a finite linear combination of terms that look like those in $X$, possibly with some factors replaced by their complex conjugates.

We define the torus to be $\mathbb{T} = \mathbb{R}/(2\pi \mathbb{Z})$. In the following, we will use some space-time norm, for any time interval $I \subset \mathbb{R}$, $u(t,x,y): I \times \mathbb{R}  \times \mathbb{T} \to \mathbb{C}$, define the space-time norm
\begin{align*}
\left\|u\right\|_{L_t^p L_x^q L_y^2(I\times \mathbb{R} \times \mathbb{T})}  & := \left\|\left\|\Big( \int_{\mathbb{T}} |u(t,x,y)|^2 \,\mathrm{d}y \Big)^\frac12\right\|_{L_x^q(\mathbb{R} )}\right\|_{L_t^p(I)},
\end{align*}
for the vector function $\vec{f}(t,x)  = \{f_j(t,x) \}_{j\in \mathbb{Z} }$, we denote
\begin{align*}
\left\| \vec{f}\right\|_{L_t^p L_x^q h^s } := \bigg \|\Big(\sum_{j \in \mathbb{Z}}  \langle j \rangle^{2s} | f_j(t,x)|^2 \Big)^\frac12 \bigg\|_{L_t^p L_x^q}
\end{align*}
where $0 \le s \le 1$. When $s = 0$, we write $L_t^p L_x^q h^s$ to be $L_t^p L_x^q l^2$.

We now define the discrete nonisotropic Sobolev space. For $\vec{\phi} = \{\phi_k\}_{k\in \mathbb{Z}}$ a sequence of real-variable functions, we define
\begin{align*}
 H^{s_1}_x h^{s_2} := \left\{ \vec{\phi} = \left\{ \phi_k\right\}_{k\in \mathbb{Z}} : \left\|\vec{\phi}\right\|_{H_x^{s_1} h^{s_2}}  =   \bigg \| \bigg( \sum\limits_{k\in \mathbb{Z}} \langle k\rangle^{2s_2} |\phi_k (x) |^2  \bigg)^\frac12 \bigg\|_{H_x^{s_1}}  < \infty \right \},
\end{align*}
where $s_1,s_2\ge 0$. In particular, when $s_1 = 0$, we denote the space $H_x^{s_1} h^{s_2}$ to be $L_x^2 h^{s_2}$.

We will frequently use the partial Fourier transform and partial space-time Fourier transform: for $f(x,y): \mathbb{R}  \times \mathbb{T} \to \mathbb{C}$,
\begin{align*}
\mathcal{F}_x f(\xi,y) = \frac1{(2\pi)^\frac12} \int_{\mathbb{R} } e^{-ix\xi} f(x,y) \,\mathrm{d}x.
\end{align*}
Given $H: \mathbb{R} \times \mathbb{R}  \times \mathbb{T}\to \mathbb{C}$, we denote the partial space-time Fourier transform to be
\begin{align*}
\mathcal{F}_{t,x}{H}(\omega,\xi,y) = \frac1{2\pi} \int_{\mathbb{R}} \int_{\mathbb{R} } e^{i\omega t -i\xi x} H(t,x,y) \,\mathrm{d}x\mathrm{d}t.
\end{align*}
We also define the partial Littlewood-Paley projectors $P_{\le N}^x$ and $P_{\ge N}^x$ as follows:
fix a real-valued radially symmetric bump function $\varphi(\xi)$ satisfying
\begin{equation}\label{eq1.6v77}
\varphi(\xi) =
\begin{cases}
1, \ |\xi|\le 1,\\
0, \ |\xi|\ge 2,
\end{cases}
\end{equation}
for any dyadic number $N\in 2^{\mathbb{Z}}$, let
\begin{align*}
\mathcal{F}_x (P_{\le N}^x f)(\xi,y) = \varphi \left(\frac\xi N\right) (\mathcal{F}_x f)(\xi,y),\\
\mathcal{F}_x (P_{\le N}^x f)(\xi,y) = \left(1-\varphi \left(\frac\xi N\right) \right)(\mathcal{F}_x f)(\xi,y)
\end{align*}
In the article, $\epsilon_0 $ is some sufficiently small positive number.

\section{Well-posedness and small data scattering}\label{se2}
 In this section, we will review the well-posedness theory and small data scattering, that is Theorem \ref{th2.3} and Theorem \ref{th2.4}. These results have been established in \cite{TV,TV2}. We also give the stability theory which will be used in showing the existence of a critical element in Section \ref{se5}.
 The argument in this Section is similar to that in \cite{CGYZ}, where the key point is that $H^{\frac12+}(\mathbb{T})\hookrightarrow L^\infty(\mathbb{T})$.

We first recall the following Strichartz estimate, which is established in \cite{TV}.
\begin{proposition}[Strichartz estimate]
\begin{align}
\left\|e^{it\Delta_{\mathbb{R}  \times \mathbb{T}}} f\right\|_{L_t^p L_x^q L_y^2} \lesssim \left\|f\right\|_{L_{x,y}^2(\mathbb{R}_x \times \mathbb{T}_y)}, \label{eq2.1} \\
\left\|\int_0^t e^{i(t-s)\Delta_{\mathbb{R}  \times \mathbb{T}}} F(s,x,y)\,\mathrm{d}s\right\|_{L_t^p L_x^q L_y^2} \lesssim \left\|F\right\|_{L_t^{\tilde{p}'} L_x^{\tilde{q}'} L_y^2}, \label{eq2.1'}
\end{align}
where $(p,q),\, (\tilde{p},\tilde{q})$ satisfies $\frac2p + \frac1q = \frac12$, $\frac2{\tilde{p}} + \frac1{\tilde{q}} = \frac12$, and $4 \le p, \tilde{p} \le \infty$.
\end{proposition}
The following nonlinear estimate is useful in showing the local wellposedness.
\begin{proposition}[Nonlinear estimate] \label{pr3.3}
\begin{align}\label{eq2.8new}
\left\|\Pi_{k=1}^5  u_k \right\|_{L_t^\frac65 L_x^\frac65 H_y^{1-\epsilon_0}}   \lesssim  \Pi_{k=1}^5 \left\| u_k\right\|_{L_t^6 L_x^6 H_y^{1-\epsilon_0}}.
\end{align}
\end{proposition}
By the Strichartz estimate and the nonlinear estimate, we can give the local well-posedness $L_x^2 H_y^1(\mathbb{R} \times \mathbb{T}  )$ and $H_{x,y}^{1}(\mathbb{R} \times \mathbb{T})$ easily. Further, together with the conservation of mass and energy, we can extend the local solution to the global solution in $H^1$. We refer to \cite{C,Killip-Visan1,T2} for the proof.
\begin{theorem}[Well-posedness]\label{th2.3}
For any $E>0$, suppose that $\left\|u_0\right\|_{L_x^2 H_{y}^{1 }(\mathbb{R}  \times \mathbb{T})} \le E$, there exists $\delta_0 = \delta_0(E)>0$ such that if
\begin{equation*}
\left\|e^{it\Delta_{\mathbb{R} \times \mathbb{T}}} u_0 \right\|_{L_t^6 L_x^6 H_y^{1-\epsilon_0}  (I\times \mathbb{R} \times \mathbb{T})} \le   \delta_0,
\end{equation*}
on the time interval $I\subset \mathbb{R}$, then there exits a unique solution $u\in C_t^0  L_x^2 H_{ y}^{ 1 }(I\times \mathbb{R} \times \mathbb{T})$ of \eqref{eq1.1}
satisfying
\begin{align*}
\left\|u\right\|_{L_t^6 L_x^6 H_y^{1-\epsilon_0} }  \le 2 \left\|e^{it\Delta_{ \mathbb{R} \times \mathbb{T} }} u_0\right\|_{L_t^6 L_x^6 H_y^{1-\epsilon_0}   },\quad  %\\
\left\|u\right\|_{L_t^\infty L_x^2 H_{ y}^{ 1 }  }    \le C \left\|u_0\right\|_{ L_x^2 H_{ y}^{ 1 }  }.
\end{align*}
Moreover, if $u_0\in H_{x,y}^1(\mathbb{R}  \times \mathbb{T})$, then $u\in C_t^0 H_{x,y}^1(\mathbb{R}\times \mathbb{T})$.
\end{theorem}
The above theorem also implies the small data scattering in $H^1$.
\begin{theorem}[Small data scattering in $ H_{x,y}^1$]\label{th2.4}
There exists small positive constant $\delta >0$ such that if $u_0\in H_{x,y}^1$ and $\left\|u_0\right\|_{ H_{x, y}^{1}(\mathbb{R}_x  \times \mathbb{T}_y)} \le \delta$,
\eqref{eq1.1} has an unique global solution
$u(t,x,y) \in C_t^0H_{x,y}^{1} \cap L_{t,x}^6 H_y^1 \cap L_t^6 W_x^{1,6} L_y^2  $ and $u$ scatters in $ H_{x,y}^1$.
\end{theorem}
We now give the stability theory in $ L_x^2 H_y^{1-\epsilon_0 }(\mathbb{R}  \times \mathbb{T})$.
\begin{theorem}[Stability theory]\label{le2.6}
Let $I$ be a compact interval and let $\tilde{u}$ be an approximate solution to $i\partial_t u + \Delta_{\mathbb{R}  \times \mathbb{T}} u =|u|^4 u$ in the
sense that
$i\partial_t \tilde{u} + \Delta_{\mathbb{R} \times \mathbb{T}} \tilde{u} = |\tilde{u}|^4 \tilde{u} + e$
for some function $e$.

Assume that
\begin{align*}
\left\|\tilde{u}\right\|_{ L_t^\infty L_x^2 H_y^{1-\epsilon_0 } } \le M,%\\
\quad
\left\|\tilde{u}\right\|_{L_t^6 L_x^6 H_y^{1-\epsilon_0} } \le L,
\end{align*}
for some positive constants $M$ and $L$.

Let $t_0\in I$ and let $u(t_0)$ obey
\begin{equation}\label{eq3.32}
\left\|u(t_0)- \tilde{u}(t_0)\right\|_{ L_x^2 H_y^{1-\epsilon_0 } } \le M'
\end{equation}
for some $M'> 0$.

Moreover, assume the smallness conditions
\begin{align}
\left\|e^{i(t-t_0)\Delta_{\mathbb{R} \times \mathbb{T}}} (u(t_0)-\tilde{u}(t_0))\right\|_{L_t^6 L_x^6 H_y^{1-\epsilon_0}  } \le \epsilon, \label{eq3.33}\\
\left\|e\right\|_{L_t^\frac65 L_x^\frac65 H_y^{1-\epsilon_0}  } \le \epsilon, \label{eq3.34}
\end{align}
for some $0 < \epsilon \le \epsilon_1$, where $\epsilon_1 = \epsilon_1(M,M',L) > 0$ is a small constant.

Then, there exists a solution $u$ to $i\partial_t u + \Delta_{\mathbb{R}\times \mathbb{T}} u = |u|^4 u$ on $I\times \mathbb{R}  \times \mathbb{T}$ with
initial data $u(t_0)$ at time $t=t_0$ satisfying
\begin{align*}
\left\|u-\tilde{u}\right\|_{L_t^6 L_x^6 H_y^{1-\epsilon_0}}  \le C(M,M',L)\epsilon, &  \quad
\left\|u-\tilde{u}\right\|_{L_t^\infty L_x^2 H_{y}^{ 1-\epsilon_0}}    \le C(M,M',L)M',\\
\left\|u\right\|_{L_t^\infty L_x^2 H_{y}^{ 1-\epsilon_0 } \cap L_t^6 L_x^6 H_y^{1-\epsilon_0}}  & \le C(M,M',L).
\end{align*}
\end{theorem}
\begin{remark}[Persistence of regularity]
The results in the above theorems can be extended to $H^1(\mathbb{R} \times \mathbb{T})$.
\end{remark}
The following theorem reveals the finiteness of the solution in $L_{t,x}^6 H_y^{1-\epsilon_0}$ is enough to show the scattering of \eqref{eq1.1} in $H^1$, and we refer to \cite{CGYZ} for a proof by using the perturbation argument.
\begin{theorem}[Scattering norm]\label{th2.946}
Suppose that $u\in C_t^0 H_{x,y}^{ 1}(\mathbb{R}_t \times \mathbb{R}_x\times \mathbb{T}_y)$ is a global solution of \eqref{eq1.1} satisfying
$\left\|u\right\|_{L_t^6 L_x^6 H_y^{1-\epsilon_0}  (\mathbb{R}_t \times \mathbb{R}_x\times \mathbb{T}_y)} \le L$ and $\left\|u(0)\right\|_{H_{x,y}^1} \le M$ for some positive constants $M,\, L$,
then $u$ scatters in $H^1_{x,y}(\mathbb{R} \times \mathbb{T})$.
\end{theorem}

\section{Profile decomposition} \label{se4}
In this section, we will show the profile decomposition. First, we will first give an abstract linear profile decomposition in Subsection \ref{sse3.1}, which will heavily depend on the linear profile decomposition in $L^2(\mathbb{R}^d )$.
The linear profile decomposition in $L^2(\mathbb{R}^d )$ for the mass-critical nonlinear Schr\"odinger equation is established by R. Carles and S. Keraani \cite{CK} after the 2-dimensional work of F. Merle and L. Vega \cite{MV}. Later, P. B\'egout and A. Vargas \cite{BV} establish the linear profile decomposition of the mass-critical nonlinear Schr\"odinger equation
for general dimensions by the refined Strichartz inequality \cite{Bo1} and bilinear restriction estimate \cite{T1}.
We also refer to \cite{Killip-Visan1} for a version of the proof of the linear profile decomposition.
We then analyze the nonlinear profiles in Subsection \ref{sse3.2}, where we use a normal form type argument.
\subsection{Linear profile decomposition}\label{sse3.1}
In this subsection, we will establish an abstract linear profile decomposition, which is implied in \cite{CGYZ}. The abstract linear profile decomposition is mainly inspired by the mass-critical profile decomposition in $\mathbb{R}^d$, $d\ge 1$, which we refer to \cite{BV,Bo1,CK,Killip-Visan1,MV}. Similar to \cite{CGYZ}, we can obtain the following linear profile decomposition in an abstract version, which is of its own interest
%importance
for other nonlinear dispersive equations such as
\begin{align*}
i\partial_t u  + H u = f(u),
\end{align*}
where $u(t,x,y): \mathbb{R}\times  \mathbb{R}^d \times \mathbb{R}^m \to \mathbb{C}$, $d\ge 1$, $n\ge 1$, and the Hamiltonian operator $H$ is defined to be $\Delta_{\mathbb{R}^d_x}  + A_y$, with $A$ defined on $\mathbb{R}^m$ and the spectrum $\sigma(A) = \sigma_p(A) \subseteq \mathbb{D}$ with $\mathbb{D}$ be a subset of $\mathbb{Z}$. Then for any $f(x,y): \mathbb{R}^d\times \mathbb{R}^m \to \mathbb{C}$, $A f(x, y) = \sum\limits_{j\in \sigma(A)} f_j(x) \psi_j(y)$, where $\psi_j$ is the corresponding eigenfunction to the eigenvalue $j$ of $A$. For example, we can take $A$ to be $\Delta_{\mathbb{N}^n}$, where $\mathbb{N}$ is an n-dimensional manifold with Dirichlet or Neumann boundary condition. These kind equations appear in study of the transverse instability of the nonlinear dispersive equations.
\begin{definition}[Symmetry group $G$]\label{de6.640}
For position $x_0\in \mathbb{R}^d $, frequency $\xi_0 \in \mathbb{R}^d$, and scaling parameter $\lambda >0$, we define for $d\ge 1$, $\mathbb{D} \subseteq \mathbb{Z}$ the unitary
transformation
$g_{x_0,\xi_0,\lambda}: L_x^2 h^1(\mathbb{R}^d  \times \mathbb{D}) \to L_x^2 h^1(\mathbb{R}^d  \times \mathbb{D})$ by
\begin{align*}
g_{x_0,\xi_0,\lambda} \vec{f}(x) = \frac1{\lambda^\frac{d}2} e^{ix \xi_0} \vec{f}\left(\frac{x-x_0}\lambda\right).
\end{align*}
Let $G$ be the collection of such transformations.
\end{definition}
To describe the defect of compactness of the embedding $e^{it\Delta_{\mathbb{R}^d}}: L_x^2 h^\alpha (\mathbb{R}^d \times \mathbb{D}) \hookrightarrow L_{t,x}^\frac{2(d+2)}d h^{\alpha-\epsilon_0}(\mathbb{R}\times \mathbb{R}^d \times \mathbb{D})$, where $0 < \alpha \le 1$, we obtain the following theorem.
\begin{theorem}[Linear profile decomposition in $L_x^2 h^\alpha(\mathbb{R}^d \times \mathbb{D})$] \label{pr3.2v15}
Let $\{\vec{u}_{n}\}_{n\ge 1}$ be a bounded sequence in $L_x^2 h^\alpha(\mathbb{R}^d  \times \mathbb{D})$. Then (after passing to a subsequence if necessary) there
exists $K^*\in  \{0,1, \cdots \} \cup \{\infty\}$, functions $\{\vec{\phi}^{k}\}_{k=1}^{K^*} \subseteq L_x^2 h^\alpha$, group elements $\{g_n^k\}_{k=1}^{K^*} \subseteq G$,
and times $\{t_n^k\}_{k=1}^{K^*} \subseteq \mathbb{R}$, and $\vec{w}_{n}^K \in L_x^2 h^\alpha(\mathbb{R}^d \times \mathbb{M})$ such that
\begin{align}\label{eq3.1v16}
\vec{u}_n(x) = &  \sum_{k=1}^K g_n^k e^{it_n^k \Delta_{\mathbb{R}^d }} \vec{\phi}^k + \vec{w}_n^K(x)
         := \sum_{k=1}^K \frac1{(\lambda_n^k)^\frac{d}2 } e^{ix\xi_n^k} (e^{it_n^k \Delta_{\mathbb{R}^d } } \vec{\phi}^k)\left(\frac{x-x_n^k}{\lambda_n^k} \right) + \vec{w}_n^K(x ),
\end{align}
we have the following properties:
\begin{align*}
\limsup_{n\to \infty} \left\|e^{it\Delta_{\mathbb{R}^d } } \vec{w}_n^K \right\|_{L_{t,x}^{\frac{2(d+2)}d} h^{\alpha-\epsilon_0} (\mathbb{R} \times \mathbb{R}^d  \times \mathbb{D} )} \to 0, \ \text{ as } K\to \infty, \\
e^{-it_n^k \Delta_{\mathbb{R}^d }} (g_n^k)^{-1} \vec{w}_n^K \rightharpoonup 0  \text{ in } L_x^2 h^\alpha,  \text{ as  } n\to \infty,\text{ for each }  k\le K,\\
\sup_{K} \lim_{n\to \infty} \left( \left\|\vec{u}_n \right\|_{L_x^2 h^\alpha}^2 - \sum_{k=1}^K  \left\|\vec{\phi}^k \right\|_{L_x^2 h^\alpha}^2 - \left\|\vec{w}_n^K\right\|_{L_x^2 h^\alpha}^2\right) = 0,
\end{align*}
and for $k\ne k'$, and $n\to \infty$,
\begin{align*}
\frac{\lambda_n^k}{\lambda_n^{k'}} + \frac{\lambda_n^{k'}}{\lambda_n^k} + \lambda_n^k \lambda_n^{k'} |\xi_n^k - \xi_n^{k'}|^2
+ \frac{|x_n^k-x_n^{k'}- 2t_n^{k} (\lambda_n^k)^2 (\xi_n^k - \xi_n^{k'}) |^2 } {\lambda_n^k \lambda_n^{k'}}
 + \frac{|(\lambda_n^k)^2 t_n^k -(\lambda_n^{k'})^2 t_n^{k'}|}{\lambda_n^k \lambda_n^{k'}} \to \infty.
\end{align*}
Furthermore, if $\{\vec{u}_n\}_{n\ge 1}$ is bounded in $  L_{x}^2 h^\alpha\cap H_x^1 l^2(\mathbb{R}^d\times \mathbb{D})$, we need to modify the decomposition \eqref{eq3.1v16} to be
\begin{align}\label{eq3.2v16}
\vec{u}_n(x) = &  \sum_{k=1}^K g_n^k e^{it_n^k \Delta_{\mathbb{R}^d }} P_n^k \vec{\phi}^k + \vec{w}_n^K(x)
         := \sum_{k=1}^K \frac1{(\lambda_n^k)^\frac{d}2 } e^{ix\xi_n^k} (e^{it_n^k \Delta_{\mathbb{R}^d } } P_n^k \vec{\phi}^k)\left(\frac{x-x_n^k}{\lambda_n^k} \right) + \vec{w}_n^K(x ),
\end{align}
where the projector $P_n^k$ is defined by
\begin{align*}
P_n^k \vec{\phi}^k(x) =
\begin{cases}
\phi^k(x),                                                 & \text{ if } \lambda_n^k  \equiv 1,\\
P_{\le (\lambda_n^k)^\theta} \vec{\phi}^k, 0 < \theta < 1,  & \text{ if } \lambda_n^k \to \infty.
\end{cases}
\end{align*}
and
\begin{align*}
& \limsup_{n\to \infty} \left\|e^{it\Delta_{\mathbb{R}^d } } \vec{w}_n^K\right\|_{L_{t,x}^{\frac{2(d+2)}d} h^{\alpha-\epsilon_0} (\mathbb{R} \times \mathbb{R}^d  \times \mathbb{D} )} \to 0, \ \text{ as } K\to \infty, \\
& e^{-it_n^k \Delta_{\mathbb{R}^d }} (g_n^k)^{-1} \vec{w}_n^K \rightharpoonup 0  \text{ in } L_x^2 h^\alpha,  \text{ as  } n\to \infty,\text{ for each }  k\le K,\\
& \sup_{K} \lim_{n\to \infty} \left( \left\|\vec{u}_n \right\|_{L_x^2 h^\alpha \cap H_x^1 l^2 }^2 - \sum_{k=1}^K \bigg\|\frac1{(\lambda_n^k)^\frac{d}2} e^{ix\xi_n^k }(e^{it_n^k \Delta_{\mathbb{R}^d }} P_n^k \vec{\phi}^k )\left ( \frac{x-x_n^k}{\lambda_n^k} \right)  \bigg\|_{L_x^2 h^\alpha \cap H_x^1 l^2 }^2 - \left\|\vec{w}_n^K\right\|_{L_x^2 h^\alpha \cap H_x^1 l^2}^2\right) = 0,
\end{align*}
and for $k\ne k'$,
\begin{align*}
\frac{\lambda_n^k}{\lambda_n^{k'}} + \frac{\lambda_n^{k'}}{\lambda_n^k} + \lambda_n^k \lambda_n^{k'} |\xi_n^k - \xi_n^{k'}|^2
+ \frac{|x_n^{k'} -x_n^{k} - 2t_n^{k'} (\lambda_n^{k'})^2 (\xi_n^{k'} -  \xi_n^k) |^2 } {\lambda_n^k \lambda_n^{k'}}  + \frac{|(\lambda_n^k)^2 t_n^k -(\lambda_n^{k'})^2
t_n^{k'}|}{\lambda_n^k \lambda_n^{k'}} \to \infty, \text{ as $n\to \infty$. }
\end{align*}
Moreover, we can take $\lambda_n^k \equiv 1$ or $\lambda_n^k \to \infty$, as $n\to \infty$, $|\xi_n^k| \le C_k$, for every $1\le k\le K$, with
$\vec{\phi}^k \in  L_x^2 h^{\alpha} (\mathbb{R}^d  \times \mathbb{D}) $, and $\vec{w}_{n}^K \in  L_x^2 h^\alpha \cap  H_x^1 l^2  (\mathbb{R}^d \times \mathbb{D})$.
\end{theorem}
Before presenting Theorem \ref{pr3.2v15}, we will first establish the refined Strichartz estimate in Proposition \ref{pr6.741}. We will collect some basic facts appeared in \cite{Killip-Visan1}.
\begin{definition}\label{de4.2136}
Given $j\in \mathbb{Z}$, we write $\mathcal{D}_j$ for the set of all dyadic intervals of side-length $2^j$ in $\mathbb{R}^d$,
\begin{align*}
\mathcal{D}_j = \left\{ \Pi_{l=1}^d [2^j k_l , 2^j(k_l +1)) \subset \mathbb{R}^d  : k = (k_1, k_2, \cdots, k_d) \in \mathbb{Z}^d \right\}.
\end{align*}
We also write $\mathcal{D} = \bigcup_{j} \mathcal{D}_j$.
Given $Q\in \mathcal{D}$, we define $f_Q $ by $\mathcal{F}_x  ({f}_Q) = \chi_Q \,\mathcal{F}_x {f} $.
\end{definition}
By the bilinear Strichartz estimate on $\mathbb{R}^d $ \cite{T1} and interpolation, we have
\begin{corollary}\label{co4.2236}
Suppose $Q,Q'\in \mathcal{D}$ with $dist(Q,Q') \gtrsim diam(Q) = diam(Q')$, then for some $1<  p < 2$,
\begin{align*}
\left\|e^{it\Delta_{\mathbb{R}^d }} f_Q \cdot e^{it\Delta_{\mathbb{R}^d }} f_{Q'}\right\|_{L_{t,x}^\frac{d^2 + 3d + 1}{d(d+1)} }
\lesssim |Q|^{1- \frac2p - \frac1{d^2 + 3d +1}} \left\|\mathcal{F}_x {f}\right\|_{L_\xi^p(Q)} \left\|\mathcal{F}_x {f}\right\|_{L_\xi^p(Q')}.
\end{align*}
\end{corollary}
We also need the follow almost orthogonality property in \cite{TVV}.
\begin{lemma} \label{leA.9}
Let $\{R_k\}$ be a family of parallelepipeds in $\mathbb{R}^d$, where $d\ge 1$ obeying
$\sup\limits_\xi \sum\limits_{k} \chi_{\alpha R_k}(\xi) \lesssim 1$ for some $\alpha >1$.
Then for any $\{f_k\}\subseteq L^\frac{d+2}d(\mathbb{R}^d)$,
\begin{align}\label{eq3.1v15}
\left\|\sum_k P_{R_k} f_k\right\|_{L^\frac{d+2}d (\mathbb{R}^d)}  \lesssim \left(\sum_k \|f_k\|_{L^\frac{d+2}d(\mathbb{R}^d)}^\frac{d+2}d \right)^\frac{d}{d+2}, \text{ when $d\ge 2$,}
\end{align}
and
\begin{align}\label{eq3.2v15}
\left\|\sum_k P_{R_k} f_k\right\|_{L^3}  \lesssim \left(\sum_k \|f_k\|_{L^3}^\frac32\right)^\frac23,
\text{ when $d = 1$}.
\end{align}
\end{lemma}
\begin{proposition}[Refined Strichartz estimate] \label{pr6.741}
\begin{align*}
\left\|e^{it\Delta_{\mathbb{R}^d }} \vec{f} \right\|_{L_{t,x}^\frac{2(d+2)}d l^\frac{2(d+2)}d(\mathbb{R}  \times \mathbb{R}^d \times \mathbb{D})} \lesssim
\left\|\vec{f}\right\|_{L_{x}^2 l^2}^\frac{d+1}{d+2} \left(\sup_{Q\in \mathcal{D}} |Q|^{-\frac{d+1}{2(d^2 + 3 d +1)}} \left\|e^{it\Delta_{\mathbb{R}^d }} \vec{f}_Q \right\|_{L_{t,x }^{\frac{2(d^2 + 3d +1)}{d^2}} l^{\frac{2(d^2 + 3d +1)}{d^2}}}\right)^\frac1{d+2}.
\end{align*}
\end{proposition}
\begin{proof}
Given distinct $\xi, \xi' \in \mathbb{R}^d $, there is a unique maximal pair of dyadic cubes $Q\ni \xi$ and $Q'\ni \xi'$ obeying
$|Q| = |Q'|$ and $dist(Q,Q') \ge 4 diam(Q)$.

Let $\mathcal{W}$ denote the family of all such pairs as $\xi \ne \xi'$ vary over $\mathbb{R}^d $. Then
\begin{align}\label{eq4.4436*}
\sum_{(Q,Q')\in \mathcal{W}} \chi_Q(\xi) \chi_{Q'}(\xi') = 1, \text{ for a.e. } (\xi,\xi')\in \mathbb{R}^d  \times \mathbb{R}^d .
\end{align}
Note that since $Q$ and $Q'$ are maximal, $dist(Q,Q') \le 10 diam (Q)$. In addition, this shows that given $Q$ there are a bounded number of
$Q'$ so that $(Q,Q') \in \mathcal{W}$, that is,
\begin{align}\label{eq4.4536*}
\forall\, Q\in  \mathcal{W}, \ \sharp \left\{Q': (Q,Q') \in \mathcal{W}\right\} \lesssim 1.
\end{align}
In view of \eqref{eq4.4436*}, we can write
\begin{align*}
\left(e^{it\Delta_{\mathbb{R}^d }} \vec{f}\right)^2 = \sum_{(Q,Q')\in \mathcal{W}} e^{it\Delta_{\mathbb{R}^d }} \vec{f}_Q \cdot e^{it\Delta_{\mathbb{R}^d }} \vec{f}_{Q'}.
\end{align*}
We have the support of the space-time Fourier transform $e^{it\Delta_{\mathbb{R}^d}} \vec{f}_Q \cdot e^{it\Delta_{\mathbb{R}^d}} \vec{f}_{Q'}$ satisfies
\begin{align}\label{eq4.4637*}
supp\left(\mathcal{F}\left(e^{it\Delta_{\mathbb{R}^d }} \vec{f}_Q \cdot  e^{it\Delta_{\mathbb{R}^d }}  \vec{f}_{Q'}  \right)\right) \subseteq R(Q+ Q'),
\end{align}
where $Q + Q'$ is the Minkowski sum and
\begin{align*}
R(Q+ Q') = \left\{ (\omega, \eta) : \eta \in Q+ Q', 2 \le \frac{\omega- \frac12 |c(Q+Q')|^2 - c(Q+Q')(\eta- c(Q+ Q'))}{diam(Q+ Q')^2} \le 19 \right\},
\end{align*}
where $c(Q+ Q')$ denotes the center of the cube $Q+ Q'$.
We also note that $diam(Q+Q') = diam(Q) + diam(Q')= 2diam(Q)$.

By \cite{Killip-Visan1}, for any $\alpha \le 1.01$,
\begin{align}\label{eq4.4737*}
\sup_{\omega,\eta} \sum_{(Q,Q')\in \mathcal{F}}  \chi_{\alpha  {R}(Q+ Q')} (\omega,\eta) \lesssim 1,
\end{align}
where $\alpha \mathcal{R}$ denotes the $\alpha-$dilate of $R$ with the same center.

Similar to the argument in \cite{Killip-Visan1}, for $d =1$, we may apply \eqref{eq3.2v15}, H\"older's inequality, we have
\begin{align}\label{eq3.9v74}
\left\|e^{it\Delta_{\mathbb{R} }} \vec{f}\right\|_{L_{t,x}^6 l^6}^2
\lesssim & \bigg(\sum_{(Q,Q')\in \mathcal{W}} \big\|e^{it\Delta_{\mathbb{R}}} \vec{f}_Q \cdot e^{it\Delta_{\mathbb{R}}} \vec{f}_{Q'} \big\|_{L_{t,x}^3 l^3}^\frac32\bigg)^\frac23\\
\lesssim & \bigg(\sum_{(Q,Q')\in \mathcal{W}} \left\|e^{it\Delta_{\mathbb{R}}} \vec{f}_Q \right\|_{L_{t,x}^{10} l^{10}}^\frac12 \left\|e^{it\Delta_{\mathbb{R}}} \vec{f}_{Q'} \right\|_{L_{t,x}^{10} l^{10}}^\frac12
\left\|e^{it\Delta_{\mathbb{R}}} \vec{f}_Q \cdot e^{it\Delta_{\mathbb{R}}} \vec{f}_{Q'} \right\|_{L_{t,x}^\frac52 l^\frac52}\bigg)^\frac23.\notag
\end{align}
To estimate the term $\left\|e^{it\Delta_{\mathbb{R}}} \vec{f}_Q \cdot e^{it\Delta_{\mathbb{R}}} \vec{f}_{Q'} \right\|_{L_{t,x}^\frac52 l^\frac52}$, we use
Corollary \ref{co4.2236} and H\"older, we see
\begin{align}\label{eq3.10v74}
 \left\|e^{it\Delta_{\mathbb{R}}} \vec{f}_Q \cdot e^{it\Delta_{\mathbb{R}}} \vec{f}_{Q'} \right\|_{L_{t,x}^\frac52 l^\frac52}
\lesssim     \Big\| |Q|^{\frac45-\frac2p} \big\|\hat{\vec{f}}(\xi) \big\|_{L_\xi^p(Q)} \big\|\hat{\vec{f}}(\xi)\big\|_{L_\xi^p(Q')}\Big\|_{l^\frac52}
\lesssim     \big\||Q|^{-\frac15} \|\hat{\vec{f}}(\xi)\|_{L_\xi^2(Q)} \|\hat{\vec{f}}(\xi)\|_{L_\xi^2(Q')} \big\|_{l^\frac52} .
\end{align}
Then by \eqref{eq3.9v74}, \eqref{eq3.10v74}, together with \eqref{eq4.4536*}, we obtain
\begin{align*}
\left\|e^{it\Delta_{\mathbb{R} }} \vec{f}\right\|_{L_{t,x}^6 l^6}^2
\lesssim & \Big(\sup_{Q\in \mathcal{D}} |Q|^{-\frac15} \|e^{it\Delta_{\mathbb{R}}} f_Q\|_{L_{t,x}^{10} l^{10}}\Big)^\frac23
\Big(\sum_{(Q,Q')\in \mathcal{W}} \|\hat{\vec{f}}(\xi)\|_{L_\xi^2 l^5(Q\times \mathbb{Z})}
\|\hat{\vec{f}}(\xi)\|_{L_\xi^2 l^5(Q'\times \mathbb{Z})}\Big )^\frac23\\
\lesssim & \Big(\sup_{Q\in \mathcal{D}} |Q|^{-\frac15} \|e^{it\Delta_{\mathbb{R}}} f_Q\|_{L_{t,x}^{10} l^{10}}\Big)^\frac23
\bigg(\Big(\sum_Q \|\hat{\vec{f}}(\xi)\|_{L_\xi^2 l^5(Q\times \mathbb{Z})}^2\Big)^\frac12
\Big(\sum_Q \Big(\sum_{Q'\sim Q} \|\hat{\vec{f}}(\xi)\|_{L_\xi^2 l^5(Q'\times \mathbb{Z})}\Big)^2\Big)^\frac12\bigg)^\frac23\\
\lesssim & \Big(\sup_{Q\in \mathcal{D}} |Q|^{-\frac15} \|e^{it\Delta_{\mathbb{R}}} f_Q\|_{L_{t,x}^{10} l^{10}}\Big)^\frac23
\|\hat{\vec{f}}(\xi)\|_{L_\xi^2 l^5}^\frac23
\Big(\sum_{Q'} \Big(\sum_{Q\sim Q'} \|\hat{\vec{f}}(\xi)\|_{L_\xi^2 l^5(Q'\times \mathbb{Z})}^2\Big)^\frac12\Big)^\frac23\\
\lesssim & \Big(\sup_{Q\in \mathcal{D}} |Q|^{-\frac15} \|e^{it\Delta_{\mathbb{R}}} f_Q\|_{L_{t,x}^{10} l^{10}}\Big)^\frac23
\big\|\hat{\vec{f}}(\xi)\big\|_{L_\xi^2 l^5}^\frac43
\lesssim \Big(\sup_{Q\in \mathcal{D}} |Q|^{-\frac15} \|e^{it\Delta_{\mathbb{R}}} f_Q\|_{L_{t,x}^{10} l^{10}}\Big)^\frac23 \|\vec{f}\|_{L_x^2 l^2}^\frac43,
\text{ where } 1< p < 2.
\end{align*}
For $d\ge 2$, we will use \eqref{eq3.1v15} instead of \eqref{eq3.2v15} in the estimate \eqref{eq3.9v74}.
\begin{align*}
&  \left\|e^{it\Delta_{\mathbb{R}^d }} \vec{f} \right\|_{L_{t,x}^\frac{2(d+2)}d l^\frac{2(d+2)}d}^\frac{2(d+2)}d \\
& \lesssim \sum\limits_{(Q,Q')\in \mathcal{W}} \left\|e^{it\Delta} \vec{f}_Q \cdot e^{it\Delta} \vec{f}_{Q'} \right\|_{L_{t,x}^\frac{d+2}d l^\frac{d+2}d}^\frac{d+2}d \\
& \lesssim \sum\limits_{(Q,Q') \in \mathcal{W}} \left\|e^{it\Delta} \vec{f}_Q \right\|_{L_{t,x}^\frac{2(d^2 + 3d+1)}{d^2} l^\frac{2(d^2 +3 d + 1)}{d^2}}^\frac1d  \left\|e^{it\Delta} \vec{f}_{Q'} \right\|_{L_{t,x}^\frac{2(d^2 + 3d+1)}{d^2} l^\frac{2(d^2 + 3d+1)}{d^2}}^\frac1d \left\|e^{it\Delta} \vec{f}_Q \cdot e^{it\Delta} \vec{f}_{Q'} \right\|_{L_{t,x}^\frac{d^2 + 3 d +1}{d(d+1)} l^\frac{d^2 + 3d + 1}{d(d+1)}}^\frac{d+1}d.
\end{align*}
The estimate of $\left\|e^{it\Delta} \vec{f}_Q \cdot e^{it\Delta} \vec{f}_{Q'} \right\|_{L_{t,x}^\frac{d^2 + 3 d +1}{d(d+1)} l^\frac{d^2 + 3d + 1}{d(d+1)}} $ is similar to the $d=1$ case, and we can obtain
\begin{align*}
 \left\|e^{it\Delta_{\mathbb{R}^d }} \vec{f} \right\|_{L_{t,x}^\frac{2(d+2)}d l^\frac{2(d+2)}d}^\frac{2(d+2)}d
& \lesssim \bigg(\sup\limits_{Q\in \mathcal{D}} |Q|^{-\frac{d+1}{2(d^2 + 3d +1)}} \left\|e^{it\Delta} \vec{f}_Q \right\|_{L_{t,x}^\frac{2(d^2 + 3d+1)}{d^2} l^\frac{2(d^2 + 3d+1)}{d^2}}\bigg)^\frac2d \sum\limits_{Q\in \mathcal{D}} \bigg( |Q|^{-\frac{2-p}p} \left\|\hat{\vec{f}} \right\|_{L_\xi^p(Q) l^\frac{2(d^2 + 3d + 1)}{d(d+1)}}^2 \bigg)^\frac{d+1}d\\
& \lesssim \bigg(\sup\limits_{Q\in \mathcal{D}} |Q|^{-\frac{d+1}{2(d^2 + 3d +1)}} \left\|e^{it\Delta} \vec{f}_Q \right\|_{L_{t,x}^\frac{2(d^2 + 3d+1)}{d^2} l^\frac{2(d^2 + 3d+1)}{d^2}}\bigg)^\frac2d \left\|\vec{f} \right\|_{L_x^2 l^2}^\frac{2(d+1)} d,
\end{align*}
where in the last estimate we use the inequality
\begin{align}\label{eq3.8v16}
\sum\limits_{Q\in \mathcal{D}} \bigg( |Q|^{-\frac{2-p}p}\Big\|\hat{\vec{f}} \Big\|_{L_\xi^p(Q) l^\frac{2(d^2 + 3d + 1)}{d(d+1)}}^2 \bigg)^\frac{d+1}d \lesssim \left\|\vec{f} \right\|_{L_x^2 l^2}^\frac{2(d+1)} d
\end{align}
which can be proved as in \cite{Killip-Visan1}.
\end{proof}
To prove the inverse Strichartz estimate, we also need the following facts:
\begin{lemma}[Refined Fatou]\label{leA.5*}
Suppose $\{\vec{f}_n\}_{n\ge1} \subseteq L^\frac{2(d+2)}d l^\frac{2(d+2)}d (\mathbb{R}^{d+1} \times \mathbb{D})$ with $\limsup\limits_{n\to \infty} \|\vec{f}_n\|_{L^\frac{2(d+2)}d  l^\frac{2(d+2)}d } < \infty$. If $\vec{f}_n\to \vec{f}$ almost everywhere, then
\begin{align*}
\sum_{ j\in \mathbb{M}}   \int_{\mathbb{R}^{d+1}  } \left|| {f}_n(z,j)|^\frac{2(d+2)}d  - | {f}_n(z,j)- {f}(z,j)|^\frac{2(d+2)}d  - | {f}(z,j) |^\frac{2(d+2)}d \right| \,\mathrm{d}z \to 0, \text{ as } n\to \infty.
\end{align*}
In particular,
\begin{align*}
\|\vec{f}_n\|_{L^\frac{2(d+2)}d  l^\frac{2(d+2)}d }^\frac{2(d+2)}d  - \|\vec{f}_n-\vec{f}\|_{L^\frac{2(d+2)}d  l^\frac{2(d+2)}d  }^\frac{2(d+2)}d  \to \|\vec{f}\|_{L^\frac{2(d+2)}d  l^\frac{2(d+2)}d }^\frac{2(d+2)} d , \text{ as } n\to \infty.
\end{align*}
\end{lemma}
\begin{proposition}[Local smoothing estimate] \label{pr4.1436*}
Fix $\epsilon > 0$ and $\varphi \in C_c^\infty$, we have $\forall\, \vec{f}(x) = f(x,j) \in L_x^2 l^2 (\mathbb{R}^d \times \mathbb{D})$ and $R>0$,
\begin{align*}
\sum_{ j\in \mathbb{D}  }\int_{\mathbb{R}} \int_{\mathbb{R}^d} \left|\left(|\nabla_x |^\frac12 e^{it\Delta_{\mathbb{R}^d}}  {f}\right)(x,j)\right|^2 \varphi\left(\frac{x}R\right)  \,\mathrm{d}x  \mathrm{d}t \lesssim_\epsilon R  \|\vec{f}\|_{L_{x }^2 l^2(\mathbb{R}^d \times \mathbb{D})}^2.
\end{align*}
and, we also have
\begin{align*}
\sum_{ j\in \mathbb{D} }\int_{\mathbb{R}} \int_{\mathbb{R}^d } \left|\left(|  \nabla_x  |^\frac12 e^{it\Delta_{\mathbb{R}^d }}  {f} \right)(x,j)\right|^2 \langle x\rangle^{-1-\epsilon} \,\mathrm{d}x  \mathrm{d}t \lesssim_\epsilon \| \vec{f}\|_{L_{x }^2 l^2(\mathbb{R}^d  \times \mathbb{D})}^2.
\end{align*}
\end{proposition}
We can now prove the inverse Strichartz estimate.
\begin{proposition}[Inverse Strichartz estimate] \label{pr6.841}
Let $\vec{f}_n(x) = f_n(x,j) $ be a bounded sequence in $L_x^2 l^{2}(\mathbb{R}^d \times \mathbb{D}) $. Suppose that
\begin{align*}
\lim_{n\to \infty} \left\|\vec{f}_n\right\|_{L_x^2l^{2}(\mathbb{R}^d \times \mathbb{D}) } = A \text{ and } \lim_{n\to \infty} \left\|e^{it\Delta_{\mathbb{R}^d}} \vec{f}_n\right\|_{L_{t,x }^\frac{2(d+2)}d l^\frac{2(d+2)}d (\mathbb{R}\times \mathbb{R}^d \times \mathbb{D}) } = \epsilon,
\end{align*}
then there exist a subsequence in $n$, $\vec{\phi} = \phi(x,j)  \in L_x^2 l^{ 2}(\mathbb{R}^d \times \mathbb{D})  $, $\{\lambda_n\} \subseteq (0,\infty)$, $\xi_n\in \mathbb{R}^d $,
and $(t_n,x_n,j_n)\in \mathbb{R}\times \mathbb{R}^d  \times \mathbb{D}$ so that along the subsequence, we have the following:
\begin{align}
& \lambda_n^\frac{d}2  e^{-i\xi_n(\lambda_n x+ x_n) } (e^{it_n\Delta_{\mathbb{R}^d }}  {f}_n)(\lambda_n x + x_n,j+j_n) \rightharpoonup  {\phi}(x,j)\  \text{ in } L_x^2 l^{ 2}(\mathbb{R}^d \times \mathbb{D})  , \text{ as } n\to \infty, \label{3.10v17} \\
 & \lim_{n\to \infty} \left( \left\|\vec{f}_n\right\|_{L_x^2 l^{ 2}}^2 - \left\|\vec{f}_n-\vec{\phi}_n \right\|_{L_x^2 l^{ 2}}^2 \right) = \left\|\vec{\phi}\right\|_{L_x^2 l^{ 2}}^2
\gtrsim A^2 \left(\frac\epsilon A\right)^{2(d+1)(d+2)},\label{eq4.4936*} \\
&  \limsup_{n\to \infty} \left\|e^{it\Delta_{\mathbb{R}^d}} (\vec{f}_n-\vec{\phi}_n) \right\|_{L_{t,x }^\frac{2(d+2)}d  l^\frac{2(d+2)}d }^\frac{2(d+2)}d  \le \epsilon^\frac{2(d+2)}d  \left( 1- c \left(\frac\epsilon A\right)^\beta\right), \label{eq4.5036*}
\end{align}
where $c$ and $\beta$ are constants,
\begin{align*}
 {\phi}_n(x,j) = \frac1{\lambda_n^\frac{d} 2} e^{ix\xi_n} \left(e^{-i\frac{t_n}{\lambda_n^2} \Delta_{\mathbb{R}^d }}
{\phi}\right)\left(\frac{x-x_n}{\lambda_n},j-j_n\right).
\end{align*}
Moreover, if $\{\vec{f}_n\}$ is bounded in $L_{x}^2 h^\alpha(\mathbb{R}^d\times \mathbb{D})$, we can take $j_n = 0$, with
$\vec{\phi} \in L_x^2 h^{\alpha }(\mathbb{R}^d  \times \mathbb{D}) $, and
\begin{align*}
\lambda_n^\frac{d}2 e^{-i\xi_n(\lambda_n x+ x_n) } (e^{it_n\Delta_{\mathbb{R}^d }} \vec{f}_n)(\lambda_n x + x_n ) \rightharpoonup  \vec{\phi}(x)\  \text{ in } L_x^2 h^{\alpha  } , \text{ as } n\to \infty,\\
\lim_{n\to \infty} \left(\left\|\vec{f}_n(x)\right\|_{L_x^2 h^{\alpha }}^2 - \left\|\vec{ f}_n-\vec{\phi}_n\right\|_{L_x^2 h^{\alpha }}^2 \right) = \left\|\vec{\phi} \right\|
_{L_x^2 h^{\alpha }}^2
\gtrsim A^2 \left(\frac\epsilon A\right)^{2(d+1)(d+2)}.
\end{align*}
Furthermore, if $\left\{\vec{f}_n\right\}$ is bounded in $L_{x}^2 h^\alpha \cap H_x^1 l^2(\mathbb{R}^d\times \mathbb{D})$, we can not only take $j_n = 0$, $\lambda_n \ge 1$, with
$\vec{\phi} \in  L_x^2 h^{\alpha }   (\mathbb{R}^d  \times \mathbb{D}) $, and $|\xi_n|\lesssim 1$,
 and
\begin{align}
\lambda_n^\frac{d}2 e^{-i\xi_n(\lambda_n x+ x_n) } (e^{it_n\Delta_{\mathbb{R}^d }} \vec{f}_n)(\lambda_n x + x_n ) \rightharpoonup  \vec{\phi}(x)\  \text{ in } L_x^2 h^{\alpha } , \text{ as } n\to \infty, \label{3.15v17}\\
\lim_{n\to \infty} \left(\left\|\vec{f}_n(x)\right\|_{L_x^2 h^{\alpha} \cap H_x^1 l^2 }^2 - \left\|\vec{ f}_n-\vec{\phi}_n\right\|_{L_x^2 h^{\alpha} \cap H_x^1 l^2}^2 \right) = \left\|\vec{\phi}\right\|
_{L_x^2 h^{\alpha }\cap H_x^1 l^2 }^2
\gtrsim A^2 \left(\frac\epsilon A\right)^{2(d+1)(d+2)}, \label{3.16v17}\\
\limsup_{n\to \infty} \left\|e^{it\Delta_{\mathbb{R}^d}} (\vec{f}_n-\vec{\phi}_n) \right\|_{L_{t,x }^\frac{2(d+2)}d  l^\frac{2(d+2)}d }^\frac{2(d+2)}d  \le \epsilon^\frac{2(d+2)}d  \left( 1- c \left(\frac\epsilon A\right)^\beta\right), \label{eq3.17v17}
\end{align}
where
\begin{align*}
\vec{\phi}_n(x,j) = \frac1{\lambda_n^\frac{d} 2} e^{ix\xi_n} \left(e^{-i\frac{t_n}{\lambda_n^2} \Delta_{\mathbb{R}^d }}
 P_n \vec{\phi}\right)\left(\frac{x-x_n}{\lambda_n},j \right),
\end{align*}
$P_n$ is the projector defined by
\begin{align*}
P_n \vec{\phi}(x) =
\begin{cases}
\vec{\phi}(x),              & \text{ if } \limsup\limits_{n\to \infty} \lambda_n < \infty,\\
P_{\le \lambda_n^\theta} \vec{\phi}(x),  & \text{ if } \lambda_n \to \infty.
\end{cases}
\end{align*}
\end{proposition}
\begin{proof}
By Proposition \ref{pr6.741}, there exists $\{Q_n\} \subseteq \mathcal{D}$ so that
\begin{align}\label{eq4.5236*}
\epsilon^{d+2} A^{-(d+1)} \lesssim \liminf_{n\to \infty} |Q_n|^{-\frac{d+1}{2(d^2 + 3 d + 1)} } \left\|e^{it\Delta_{\mathbb{R}^d}} (\vec{f}_n)_{Q_n}\right\|_{L_{t,x }^{\frac{2(d^2 + 3 d +1)}{d^2}} l^{\frac{2(d^2 + 3 d +1)}{d^2}}}.
\end{align}
We choose $\lambda_n^{-1}$ to be the side-length of $Q_n$, which implies $|Q_n|= \lambda_n^{-d}$.
We also set $\xi_n = c(Q_n)$, this is the center of this cube. Next, we determine $x_n$, $j_n$ and $t_n$.
By H\"older's inequality,
\begin{align*}
 & \liminf_{n\to \infty} |Q_n|^{-\frac{d+1}{2(d^2 + 3 d + 1)} } \left\|e^{it\Delta_{\mathbb{R}^d }} (\vec{f}_n)_{Q_n} \right\|_{L_{t,x }^{\frac{2(d^2 + 3 d +1)}{d^2}} l^{\frac{2(d^2 + 3 d +1)}{d^2}}}\\
& \lesssim \liminf_{n\to \infty} |Q_n|^{-\frac{d+1}{2(d^2 + 3 d + 1)} } \left\|e^{it\Delta_{\mathbb{R}^d}} (\vec{f}_n)_{Q_n} \right\|_{L_{t,x }^\frac{2(d+2)} d  l^\frac{2(d+2)} d  }^\frac{d(d+2)}{d^2 + 3 d + 1} \left\|e^{it\Delta_{\mathbb{R}^d }} (\vec{f}_n)_{Q_n} \right\|_{L_{t,x }^\infty l^\infty }^\frac{d+1}{d^2 + 3d + 1} \nonumber \\
& \lesssim \liminf_{n\to \infty} \lambda_n^{\frac{d(d+1) }{2(d^2 + 3 d + 1)}} \epsilon^\frac{d(d+2)}{d^2 + 3 d+1} \left\|e^{it\Delta_{\mathbb{R}^d }} (\vec{f}_n)_{Q_n} \right\|_{L_{t,x }^\infty l^\infty }^\frac{d+1}{d^2 + 3d + 1}.
\end{align*}
Thus by \eqref{eq4.5236*}, there exists $(t_n,x_n,j_n) \in \mathbb{R} \times \mathbb{R}^d  \times \mathbb{D}$ so that
\begin{align}\label{eq4.5336*}
\liminf_{n\to \infty}  \lambda_n^\frac{d}2 \left|(e^{it_n \Delta_{\mathbb{R}^d }} ( {f}_n)_{Q_n} )(x_n,j_n)\right|\gtrsim \epsilon^{(d+1)(d+2)} A^{-(d^2 + 3d + 1)}.
\end{align}
Having selected our symmetry parameters, weak compactness of $L_x^2 l^{ 2}(\mathbb{R}^d \times \mathbb{D})$ guarantees that
\begin{align} \label{eq3.22v17}
\lambda_n^\frac{d}2 e^{-i\xi_n(\lambda_n x+x_n)} (e^{it_n \Delta_{\mathbb{R}^d }}  {f}_n)(\lambda_n x + x_n, j+ j_n) \rightharpoonup   {\phi}(x,j) \text{ in } L_x^2 l^{ 2}(\mathbb{R}^d  \times \mathbb{D})
\end{align}
for some subsequence in $n$.
Our next job is to show that $\vec{\phi}$ carries non-trivial norm. Define $h$ so that $\hat{h}$
is the characteristic function of the cube $[-\frac12,\frac12)^d$. Then $h(x) \delta_0(j)\in L_x^2 l^{ 2}(\mathbb{R}^d  \times \mathbb{D})$.
From \eqref{eq4.5336*}, we obtain
\begin{align}\label{eq4.5436*}
 \left|\langle h(x) \delta_0(j),  {\phi}(x,j)\rangle_{x,j} \right|
=  &  \lim_{n\to \infty} \left|\int_{\mathbb{R}^d }\sum_{j\in \mathbb{D}} \delta_0(j) \bar{h}(x)  \lambda_n^\frac{d}2 e^{-i\xi_n(\lambda_n x + x_n)} (e^{it_n \Delta_{\mathbb{R}^d }}  {f}_n)(\lambda_n x +x_n, j+ j_n) \,\mathrm{d}x   \right| \nonumber \\
=  & \lim_{n\to \infty} \lambda_n^\frac{d}2 \left|\sum_{j\in \mathbb{D}}  (e^{it_n \Delta_{\mathbb{R}^d }} ( {f}_n)_{Q_n})(x_n,j+j_n) \delta_0(j)   \right| \nonumber \\
=  & \lim_{n\to \infty}  \lambda_n^\frac{d}2 \left|(e^{it_n\Delta_{\mathbb{R}^d }} ( {f}_n)_{Q_n})(x_n,j_n) \right|
\gtrsim \epsilon^{(d+1)(d+2)} A^{-(d^2 + 3d +1)},
\end{align}
which quickly implies \eqref{eq4.4936*}. This leaves us to consider \eqref{eq4.5036*}. First, we claim that after passing to a subsequence,
\begin{align*}
e^{it\Delta_{\mathbb{R}^d }} \left( \lambda_n^\frac{d}2 e^{-i \xi_n(\lambda_n x + x_n) } (e^{it_n\Delta_{\mathbb{R}^d }}  {f}_n)(\lambda_n x + x_n, j+ j_n) \right) \to
\left( e^{it\Delta_{\mathbb{R}^d }} {\phi} \right)(x,j),\text{ a.e.  } (t, x,j) \in \mathbb{R} \times \mathbb{R}^d \times \mathbb{D}.
\end{align*}
Indeed, this follows from Proposition \ref{pr4.1436*} and the Rellich-Kondrashov theorem. Thus by applying Lemma \ref{leA.5*} and transferring the symmetries, we obtain
\begin{align*}
\left\|e^{it \Delta_{\mathbb{R}^d}} \vec{f}_n\right\|_{L_{t,x}^\frac{2(d+2)}d l^\frac{2(d+2)}d }^\frac{2(d+2)}d  - \left\|e^{it\Delta_{\mathbb{R}^d }} (\vec{f}_n-\vec{\phi}_n)\right\|_{L_{t,x }^\frac{2(d+2)}d  l^\frac{2(d+2)}d}^\frac{2(d+2)}d  - \left\|e^{it\Delta_{\mathbb{R}^d}} \vec{\phi}_n\right\|_{L_{t,x }^\frac{2(d+2)}d l^\frac{2(d+2)}d }^\frac{2(d+2)}d  \to 0, \text{ as } n\to \infty.
\end{align*}
The requisite lower bound on the right-hand side follows from \eqref{eq4.5436*}.

If $\left\{\vec{f}_n\right\}$ is bounded in $  L_{x}^2 h^\alpha(\mathbb{R}^d\times \mathbb{D})$, we can obtain \eqref{3.10v17} and \eqref{eq4.4936*} with the norm $L_x^2 l^2$ changed to the stronger norm $L^2 h^\alpha$ and take $j_n = 0$.

If $\left\{\vec{f}_n\right\}$ is bounded in $  L_{x}^2 h^\alpha\cap H_x^1 l^2(\mathbb{R}^d\times \mathbb{D})$, by
\begin{align*}
\limsup\limits_{n\to \infty} \left\|P_{\ge R}^x \vec{f}_n \right\|_{L_x^2 h^{\alpha-\frac{\epsilon_0}2} } \lesssim \limsup\limits_{n\to \infty} \langle R\rangle^{-\frac{\epsilon_0}2} \left\| \vec{f}_n \right\|_{H_x^1 l^2 \cap L_x^2 h^\alpha} \to 0, \text{ as } R\to \infty,
\end{align*}
we can replace $\vec{f}_n$ by $P_{\le R}^x  \vec{f}_n$ in the assumption of the proposition, for $R= R(A, \epsilon) >0$ large enough, then
we can take $\left\{Q_n\right\} \subset \mathcal{D}$, with $|Q_n|\lesssim R^d$ in the above argument, and let $\lambda_n$ be the inverse of the side-length of $Q_n$, therefore $\lambda_n \gtrsim R^{-1}$, and the center of the cube $\xi_n$ satisfies $\left|\xi_n\right| \lesssim R$.
We also have \eqref{eq3.22v17}, and furthermore, if $\limsup\limits_{n\to \infty} \lambda_n < \infty$, we have
\begin{align*}
\lambda_n^\frac{d}2 e^{-i\xi_n(\lambda_n x+x_n)} (e^{it_n \Delta_{\mathbb{R}^d }}  {f}_n)(\lambda_n x + x_n, j+ j_n) \rightharpoonup   {\phi}(x,j) \text{ in } L_x^2 h^{ \alpha } \cap H_x^1 l^2(\mathbb{R}^d  \times \mathbb{D}).
\end{align*}
To show \eqref{3.16v17}, we only need to treat the case $\lambda_n \to \infty$, since the other case is as in the previous cases. We can see
\begin{align*}
\lim\limits_{n\to \infty} \left\|\vec{\phi}_n\right\|_{H_x^1 l^2 \cap L_x^2 h^\alpha} \ge \lim\limits_{n\to \infty} \left\|P_{\le \lambda_n^\theta}^x \vec{\phi}\right\|_{L_x^2 h^\alpha} \gtrsim A^2 \left(\frac\epsilon A\right)^{2(d+1)(d+2)}.
\end{align*}
For the decoupling of the norm, it comes from the fact $P_{\lambda^\theta}^x \to \mathbb{I}$ in $L^2_x h^\alpha$ and \eqref{3.15v17}.
\eqref{eq3.17v17} follows as in the previous cases.
\end{proof}
By apply Proposition \ref{pr6.841}, we can obtain
\begin{proposition}[Linear profile decomposition in $L_x^2 h^\alpha(\mathbb{R}^d \times \mathbb{D})$]  \label{pr6.640}
Let $\left\{\vec{u}_{n}\right\}_{n\ge 1}$ be a bounded sequence in $L_x^2 h^\alpha(\mathbb{R}^d  \times \mathbb{D})$. Then (after passing to a subsequence if necessary) there
exists $K^*\in  \{0,1, \cdots \} \cup \{\infty\}$, functions $\left\{\vec{\phi}^{k}\right\}_{k=1}^{K^*} \subseteq L_x^2 h^\alpha$, group elements $\{g_n^k\}_{k=1}^{K^*} \subseteq G$,
and times $\left\{t_n^k\right\}_{k=1}^{K^*} \subseteq \mathbb{R}$ so that defining $\vec{w}_{n}^K$ by
\begin{align*}
\vec{u}_n(x) =    \sum_{k=1}^K g_n^k e^{it_n^k \Delta_{\mathbb{R}^d }} \vec{\phi}^k + \vec{w}_n^K(x)
         :=   \sum_{k=1}^K \frac1{(\lambda_n^k)^\frac{d}2 } e^{ix\xi_n^k} (e^{it_n^k \Delta_{\mathbb{R}^d } } \vec{\phi}^k)\left(\frac{x-x_n^k}{\lambda_n^k} \right) + \vec{w}_n^K(x ),
\end{align*}
we have the following properties:
\begin{align} \label{3.16v15}
\limsup_{n\to \infty} \left\|e^{it\Delta_{\mathbb{R}^d }} \vec{w}_n^K\right\|_{L_{t,x}^\frac{2(d+2)}d  l^\frac{2(d+2)}d (\mathbb{R} \times \mathbb{R}^d  \times \mathbb{D} )} \to 0, \ \text{ as } K\to \infty, \\
e^{-it_n^k \Delta_{\mathbb{R}^d }} (g_n^k)^{-1} \vec{w}_n^K \rightharpoonup 0  \text{ in } L_x^2 h^\alpha,  \text{ as  } n\to \infty,\text{ for each }  k\le K, \notag \\
\sup_{K} \lim_{n\to \infty} \left(\left\|\vec{u}_n\right\|_{L_x^2 h^\alpha}^2 - \sum_{k=1}^K \left\|\vec{\phi}^k\right\|_{L_x^2 h^\alpha}^2 - \left\|\vec{w}_n^K\right\|_{L_x^2 h^\alpha}^2\right) = 0, \notag
\end{align}
and lastly, for $k\ne k'$, and $n\to \infty$,
\begin{align*}
\frac{\lambda_n^k}{\lambda_n^{k'}} + \frac{\lambda_n^{k'}}{\lambda_n^k} + \lambda_n^k \lambda_n^{k'} |\xi_n^k - \xi_n^{k'}|^2 + \frac{|x_n^k-x_n^{k'}- 2t_n^{k} (\lambda_n^k)^2 (\xi_n^k - \xi_n^{k'}) |^2 } {\lambda_n^k \lambda_n^{k'}}  + \frac{|(\lambda_n^k)^2 t_n^k -(\lambda_n^{k'})^2 t_n^{k'}|}{\lambda_n^k \lambda_n^{k'}} \to \infty.
\end{align*}
Furthermore, if $\left\{\vec{u}_n\right\}_{n\ge 1}$ is bounded in $  L_{x}^2 h^\alpha\cap H_x^1 l^2(\mathbb{R}^d\times \mathbb{D})$, we can take $\lambda_n^k \ge 1$, with
$\vec{\phi}^k \in  L_x^2 h^{\alpha } (\mathbb{R}^d  \times \mathbb{D}) $, and $|\xi_n^k|\lesssim 1$,
 and
\begin{align*}
\vec{u}_n(x) =    \sum_{k=1}^K g_n^k e^{it_n^k \Delta_{\mathbb{R}^d }} P_n^k \vec{\phi}^k + \vec{w}_n^K(x)
         := \sum_{k=1}^K \frac1{(\lambda_n^k)^\frac{d}2 } e^{ix\xi_n^k} (e^{it_n^k \Delta_{\mathbb{R}^d } } P_n^k \vec{\phi}^k)\left(\frac{x-x_n^k}{\lambda_n^k} \right) + \vec{w}_n^K(x ),
\end{align*}
in addition,
\begin{align*}
\sup_{K} \lim_{n\to \infty} \left(\left\|\vec{u}_n\right\|_{L_x^2 h^\alpha \cap H_x^1 l^2}^2 - \sum_{k=1}^K \left\|g_n^k e^{it_n^k \Delta_{\mathbb{R}^d }} P_n^k \vec{\phi}^k \right\|_{L_x^2 h^\alpha \cap H_x^1 l^2}^2 - \left\|\vec{w}_n^K\right\|_{L_x^2 h^\alpha \cap H_x^1 l^2}^2\right) = 0.
\end{align*}
\end{proposition}
\begin{remark}\label{re3.11v15}
By using interpolation, the H\"older inequality and \eqref{3.16v15}, we have
\begin{align*}
  \limsup_{n\to \infty} \left\|e^{it\Delta_{\mathbb{R}^d }} \vec{w}_n^K\right\|_{L_{t,x}^\frac{2(d+2)}d h^{\alpha-\epsilon_0}}
\lesssim &  \limsup_{n\to \infty} \left\|e^{it\Delta_{\mathbb{R}^d }} \vec{w}_n^K\right\|_{L_{t,x}^\frac{2(d+2)}d  h^{\alpha }}^{\frac{\alpha-\epsilon_0}\alpha} \left\|e^{it\Delta_{\mathbb{R}^d }} \vec{w}_n^K\right\|_{L_{t,x}^\frac{2(d+2)}d  l^2}^{\frac{\epsilon_0}{\alpha}}  \\
\lesssim  & \limsup_{n\to \infty} \left\|\vec{w}_n^K\right\|_{L_x^2 h^\alpha}^\frac{\alpha-\epsilon_0}\alpha  \left\|e^{it\Delta_{\mathbb{R}^d}} \vec{w}_n^K\right\|_{L_{t,x}^\frac{2(d+2)}d  l^1} ^{\frac{2 \epsilon_0}{\alpha( d+4) }}  \left\|e^{it\Delta_{\mathbb{R}^d}} \vec{w}_n^K\right\|_{L_{t,x}^\frac{2(d+2)}d  l^\frac{2(d+2)}d }^{\frac{{(d+2)}\epsilon_0}{\alpha( d+4) }} \\
\lesssim & \limsup_{n\to \infty} \left\|\vec{w}_n^K\right\|_{L_x^2 h^\alpha}^\frac{\alpha-\epsilon_0}\alpha \left\|e^{it\Delta_{\mathbb{R}^d}} \vec{w}_n^K\right\|_{L_{t,x}^\frac{2(d+2)}d  h^\alpha} ^{\frac{ 2\epsilon_0}{ \alpha( d+4)  }}  \left\|e^{it\Delta_{\mathbb{R}^d}} \vec{w}_n^K\right\|_{L_{t,x}^\frac{2(d+2)}d  l^\frac{2(d+2)} d }^{\frac{(d+2) \epsilon_0}{\alpha( d+4) }}\\
\lesssim & \limsup_{n\to \infty} \left\|\vec{w}_n^K\right\|_{L_x^2 h^1}^{1-\frac{(d+2) \epsilon_0}{\alpha( d+4) }} \left\|e^{it\Delta_{\mathbb{R}^d}} \vec{w}_n^K\right\|_{L_{t,x}^\frac{2(d+2)} d l^\frac{2(d+2)} d }^{\frac{(d+2) \epsilon_0}{\alpha( d+4)  }}
  \to 0, \text{ as }  K\to \infty.
\end{align*}
\end{remark}
Therefore, Proposition \ref{pr6.640} and Remark \ref{re3.11v15} implies Theorem \ref{pr3.2v15}, which completes the proof.

As an consequence of Theorem \ref{pr3.2v15}, by working on the Fourier coefficients, we obtain
\begin{proposition}[Linear profile decomposition in $H^1( \mathbb{R}  \times \mathbb{T})$] \label{pro3.9v23}
Let $\left\{u_n\right\}_{n\ge 1}$ be a bounded sequence in $H_{x,y}^{ 1}(\mathbb{R}  \times \mathbb{T})$. Then after passing to a subsequence if necessary,
there exist $K^* \in \{0,1, \cdots\} \cup \{\infty\}$,
functions $\phi^{k}$ in $L_x^2 H_{y}^{ 1}(\mathbb{R} \times \mathbb{T})$ and mutually orthogonal frames $(\lambda_n^k, t_n^k, x_n^k,\xi_n^k )_{n\ge 1} \subseteq (0,\infty) \times \mathbb{R} \times \mathbb{R}  \times \mathbb{R} $,
which means
\begin{align}\label{eq3.141}
\frac{\lambda_n^{k'} }{\lambda_n^k} + \frac{\lambda_n^k}{\lambda_n^{k'}}
 + \lambda_n^{k'} \lambda_n^{k} |\xi_n^{k'} - \xi_n^{k}|^2
 + \frac{|x_n^k-x_n^{k'}- 2t_n^{k} (\lambda_n^k)^2 (\xi_n^k - \xi_n^{k'}) |^2 } {\lambda_n^k \lambda_n^{k'}}
+ \frac{|(\lambda_n^{k'} )^2 t_n^{k'} - (\lambda_n^k)^2 t_n^k |}{\lambda_n^{k'} \lambda_n^k} \to \infty, \text{ as } n \to \infty, \text{ for } k' \ne k,
\end{align}
with $\lambda_n^k\to 1$ or $\infty$, as $n\to \infty$, $|\xi_n^k| \le C_k$, for every $1\le k \le J$,
and for every $K \le K^*$ a sequence $r_n^K \in  H_{x,y}^{ 1}(\mathbb{R}  \times \mathbb{T})$, such that
\begin{align*}
u_n(x,y)  = \sum\limits_{k=1}^K \frac1{(\lambda_n^k )^\frac12} e^{ix\xi_n^k} \left(e^{it_n^k \Delta_{\mathbb{R}    }} P_n^k \phi^k\right)\left(\frac{x-x_n^k }{\lambda_n^k},y\right)  + r_n^K(x,y).
\end{align*}
Moreover,
\begin{align}
& \lim\limits_{n\to \infty} \left(\|u_n\|_{ L_{x,y}^{ 6}}^6 - \sum\limits_{k=1}^K \left\| \frac1{(\lambda_n^k)^\frac12} e^{ix\xi_n^k} \left(e^{it_n^k\Delta_{\mathbb{R}    }} P_n^k \phi^k\right)\left(\frac{x-x_n^k}{\lambda_n^k},y\right) \right\|_{  L_{x,y}^{ 6}}^6 - \|r_n^K\|_{ L_{x,y}^{ 6}}^6 \right) = 0,\\
& \lim\limits_{n\to \infty} \left(\|u_n\|_{ H_{x,y}^{ 1}}^2 - \sum\limits_{k=1}^K \left\| \frac1{(\lambda_n^k)^\frac12} e^{ix\xi_n^k} \left(e^{it_n^k\Delta_{\mathbb{R}    }} P_n^k \phi^k\right)\left(\frac{x-x_n^k}{\lambda_n^k},y\right) \right\|_{  H_{x,y}^{ 1}}^2 - \|r_n^K\|_{  H_{x,y}^{ 1}}^2 \right) = 0, \, \forall\, K \le K^*, \label{eq5.1} \\
&(\lambda_n^k)^\frac12   e^{-it_n^k \Delta_{\mathbb{R}   }}\left(e^{-i(\lambda_n^k x + x_n^k) \xi_n^k}  r_n^J\left(\lambda_n^k x+ x_n^k,y\right)\right) \rightharpoonup 0 \text{ in } L_x^2 H^{ 1}_{y},  \text{ as  } n\to \infty,   \text{ for each }  k\le K,  \label{eq5.2} \\
&  \limsup\limits_{n\to \infty} \left\|e^{it \Delta_{\mathbb{R} \times \mathbb{T}   }} r_n^{K} \right\|_{L_t^6 L_x^6 H_y^{1-\epsilon_0} (\mathbb{R}\times \mathbb{R} \times \mathbb{T})}   \to 0, \text{ as } K\to K^*. \label{eq5.3}
\end{align}
\end{proposition}

\subsection{Approximation of profiles}\label{sse3.2}
In this subsection, by using the solution of the one-discrete-component quintic resonant nonlinear Schr\"odinger system to approximate the nonlinear Schr\"odinger equation with initial data be the bubble in the linear profile decomposition, we can show the nonlinear profile has a bounded space-time norm.
\begin{theorem}[Large-scale profiles] \label{pr5.9}
For any $\phi\in L_x^2 H_y^{ 1}(\mathbb{R}  \times \mathbb{T})$, $0 < \theta < 1$, $(\lambda_n, t_n, x_n,\xi_n )_{n\ge 1} \subseteq (0,\infty) \times \mathbb{R}^3 $, $\lambda_n\to \infty$, as $n\to \infty$, $|\xi_n| \lesssim 1$, there is a global solution $u_n \in C_t^0 L_x^2 H_{y}^{ 1}$ of
\begin{align}\label{eq3.833}
\begin{cases}
i\partial_t u_n + \Delta_{\mathbb{R}  \times \mathbb{T}} u_n = |u_n|^4 u_n,\\
u_n(0,x,y) = \frac1{\lambda_n^\frac12}  e^{ix \xi_n} (e^{it_n \Delta_{\mathbb{R}  }} P_{\le \lambda_n^\theta} \phi)\left(\frac{x-x_n}{\lambda_n},y\right),
\end{cases}
\end{align}
for $n$ large enough, satisfying
$\|u_n\|_{L_t^\infty L_x^2 H_{y}^{ 1}\cap L_{t,x}^6 H_y^1(\mathbb{R}\times \mathbb{R} \times \mathbb{T})} \lesssim_{\|\phi\|_{L_x^2 H_y^1}} 1$.
Furthermore, assume $\epsilon_1 $ is a sufficiently small positive constant depending only on $\|\phi\|_{L_x^2 H_y^1}, \vec{v}_0 \in  H^2_x h^1$, and
\begin{align}\label{eq3.935}
\left\|\vec{\phi}- \vec{v}_0\right\|_{L^2_x h^1 } \le \epsilon_1.
\end{align}
There exists a solution $\vec{v} \in  C_t^0 H_x^2  h^1(\mathbb{R} \times \mathbb{R} \times \mathbb{Z})$ of the one-discrete-component quintic resonant nonlinear Schr\"odinger system
\begin{align}\label{eq3.4new}
\begin{cases}
i\partial_t v_j + \Delta_{\mathbb{R} } v_j =  \sum\limits_{(j_1,j_2,j_3,j_4,j_5 )\in \mathcal{R}(j)} v_{j_1} \bar{v}_{j_2} v_{j_3} \bar{v}_{j_4} v_{j_5},\\
v_j(0) = v_{0,j}, j\in \mathbb{Z},
\end{cases}
\end{align}
with
\begin{align*}
 v_j(0,x)  & = v_{0,j}(x),   \ \text{ if }  t_n = 0,\\
 \left\|\vec{v} (t) - e^{it\Delta_{\mathbb{R} }}
 \vec{v}_0 \right\|_{L_x^2 h^1} &  \to 0, \text{ as }  t\to \pm \infty, \  \text{ if } t_n\to \pm \infty.
\end{align*}
such that for any $\epsilon > 0$,
it holds that
\begin{align*}
\|u_n(t) - w_n(t)\|_{L_t^\infty L_{x }^2 H_y^1 \cap L_{t,x}^6 H_y^1(\mathbb{R}\times \mathbb{R}  \times \mathbb{T}) } \lesssim_{\|\phi\|_{L_x^2 H_y^1}} \epsilon_1,\\
\|u_n\|_{L_t^\infty L_{x }^2 H_y^1 \cap L_{t,x}^6 H_y^1(\mathbb{R}\times \mathbb{R}  \times \mathbb{T}) } \lesssim 1,
\end{align*}
for $n$ large enough, where
\begin{align*}
w_n(t, x,y) =  e^{-i(t-t_n)|\xi_n|^2} e^{ix\xi_n} V_{\lambda_n}(t, x, y),
\end{align*}
and $V_{\lambda_n}$ is defined to be
\begin{align*}
V_{\lambda_n}(t,x,y) = \sum_{j\in \mathbb{Z}} \frac1{\lambda_n^\frac12} e^{-it|j|^2} e^{iyj} v_j\left(\frac{t}{\lambda_n^2} +t_n, \frac{x-x_n-2\xi_n(t-t_n)}{\lambda_n}\right), \  (t,x,y) \in  \mathbb{R} \times \mathbb{R}  \times \mathbb{T}.
\end{align*}
\end{theorem}
Before giving the proof, we present a simple lemma which is useful in the proof of the theorem.
Similar to the argument in \cite{CGYZ}, we have
\begin{lemma}[Elementary estimate]\label{le3.1061}
\begin{align}\label{eq3.31v83}
\sup_{j\in \mathbb{Z}} \langle j\rangle^2 \sum_{\substack{j_1,j_2,j_3,j_4,j_5\in \mathbb{Z}, \\j_1-j_2+ j_3- j_4+ j_5 = j}} \langle j_1\rangle^{-2} \langle j_2\rangle^{-2} \langle j_3\rangle^{-2}  \langle j_4\rangle^{-2} \langle j_5\rangle^{-2} \lesssim 1.
\end{align}
\end{lemma}
\begin{proof}
We see
\begin{align*}
 \langle j\rangle^2 \sum_{\substack{j_1,j_2,j_3,j_4,j_5\in \mathbb{Z},\\ j_1-j_2+ j_3 -j_4 + j_5 = j}} \langle j_1\rangle^{-2} \langle j_2\rangle^{-2} \langle j_3\rangle^{-2} \langle j_4\rangle^{-2 } \langle j_5\rangle^{-2}%\\
=  \langle j\rangle^2 \sum_{j_2, j_3,j_4,j_5 \in \mathbb{Z}} \langle j_2 \rangle^{-2} \langle j_3\rangle^{-2} \langle j_4\rangle^{-2} \langle j_5\rangle^{-2}  \langle j+j_2-j_3+ j_4 - j_5 \rangle^{-2} .
\end{align*}
By an easy calculus, we have
\begin{align*}
 & \sum_{j_5 \in \mathbb{Z}} \langle j+j_2-j_3 + j_4 - j_5 \rangle^{-2} \langle j_5\rangle^{-2}\\
\le &  \int_{\mathbb{R}} \langle j+j_2 - j_3 + j_4 -y\rangle^{-2} \langle y\rangle^{-2} \, \mathrm{d}y \\
\le  & \int_{|j+j_2-j_3 + j_4 -y|\ge \frac{|j+j_2-j_3 + j_4 |}2}  \langle j+j_2-j_3 + j_4 -y\rangle^{-2} \langle y\rangle^{-2} \, \mathrm{d}y
 + \int_{|j+j_2-j_3 + j_4 -y|\le \frac{|j+j_2-j_3 + j_4 |}2}  \langle j+j_2-j_3 + j_4 -y\rangle^{-2} \langle y\rangle^{-2} \, \mathrm{d}y\\
\lesssim &  \langle j+j_2 - j_3 + j_4  \rangle^{-2} \int_{\mathbb{R}} \langle y\rangle^{-2} \, \mathrm{d}y
 +  \langle j+j_2 -j_3 + j_4 \rangle^{-2}  \int_{\mathbb{R}} \langle j+j_2-j_3 + j_4 - y\rangle^{-2} \, \mathrm{d}y \lesssim  \langle j+j_2-j_3 + j_4 \rangle^{-2},
\end{align*}
arguing in the same way, we can obtain
\begin{align*}
& \langle j\rangle^2 \sum_{j_2, j_3,j_4,j_5 \in \mathbb{Z}} \langle j_2 \rangle^{-2} \langle j_3\rangle^{-2} \langle j_4\rangle^{-2} \langle j_5\rangle^{-2}  \langle j+j_2-j_3+ j_4 - j_5 \rangle^{-2}\\
\lesssim &  \langle j\rangle^2 \sum_{j_2 , j_3,j_4 \in \mathbb{Z}} \langle j_2 \rangle^{-2} \langle j_3 \rangle^{-2} \langle j_4\rangle^{-2} \langle j+j_2 -j_3 + j_4 \rangle^{-2}\\
 \lesssim &  \langle j\rangle^2 \sum_{j_2,j_3} \langle j_2\rangle^{-2} \langle j_3 \rangle^{-2} \langle j+ j_2 - j_3\rangle^{-2}
 \lesssim \langle j\rangle^2 \sum_{j_2} \langle j_2\rangle^{-2} \langle j+j_2\rangle^{-2} \lesssim
 \langle j\rangle^2 \cdot \langle j\rangle^{-2} \lesssim 1.
\end{align*}
\end{proof}
\begin{remark}
By scattering theorem of the one-discrete-component quintic resonant nonlinear Schr\"odinger system, we have the solution of the one-discrete-component quintic resonant nonlinear Schr\"odinger system satisfies the following result as a consequence of the triangle inequality, Strichartz estimate, Minkowski inequality, and Lemma \ref{le3.1061}.
\begin{align*}
\left\|\vec{v}\right\|_{h^1 L_{t,x}^6}
& \lesssim \left\|\vec{v}(t)- e^{it\Delta_{\mathbb{R} }} \vec{v}^+ \right\|_{h^1 L_{t,x}^6} + \left\|e^{it\Delta_{\mathbb{R}}} \vec{v}^+ \right\|_{h^1L_{t,x}^6}\\
& \lesssim \Big\|\sum_{(j_1,j_2,j_3,j_4,j_5)\in \mathcal{R}(j)} v_{j_1} \bar{v}_{j_2} v_{j_3} \bar{v}_{j_4} v_{j_5} \Big \|_{h_j^1 L_{t,x}^\frac65}
+ \left\|\vec{v}^+\right\|_{h_j^1L_x^2}\\
& \lesssim \Big\|\sum_{(j_1,j_2,j_3,j_4,j_5)\in \mathcal{R}(j)} v_{j_1} \bar{v}_{j_2} v_{j_3} \bar{v}_{j_4} v_{j_5} \Big\|_{ L_{t,x}^\frac65 h_j^1} + \left\|\vec{v}^+\right\|_{L_x^2 h_j^1}\\
& \lesssim \Big\|\Big(\sum_{j\in \mathbb{Z}} \langle j\rangle^2 | v_j(t,x)|^2\Big)^\frac52\Big\|_{L_{t,x}^\frac65} +\|\vec{v}^+\|_{L_x^2 h^1} \sim \|\vec{v}\|_{L_{t,x}^6 h^1}^5 + \|\vec{v}^+\|_{L_x^2 h^1}.
\end{align*}
\end{remark}
{\it Proof of Theorem \ref{pr5.9}.}
Without loss of generality, we may assume that $x_n = 0$.
Using the Galilean transform and the fact that $\xi_n$ is bounded, we may assume that $\xi_n = 0$ for
all $n$. We see
\begin{align*}
w_n(t,x,y) = \sum_{j\in \mathbb{Z}} \frac1{\lambda_n^\frac12} e^{-it|j|^2} e^{iyj} v_j\left(\frac{t}{\lambda^2_n} + t_n, \frac{x}{\lambda_n}\right).
\end{align*}
When $t_n = 0$, we will show $V_{\lambda_n}$ is an approximate solution to the quintic nonlinear Schr\"odinger equation on $\mathbb{R}  \times \mathbb{T}$ in the sense of Theorem \ref{le2.6}. By noting $v_j$ satisfies \eqref{eq3.4new} and an easy computation, we have
\begin{align}\label{eq3.933}
e_{\lambda_n} & = (i\partial_t + \Delta_{\mathbb{R} \times \mathbb{T}}) V_{\lambda_n} - |V_{\lambda_n}|^4 V_{\lambda_n}  \\
     & = - \sum\limits_{j\in \mathbb{Z}} e^{-it|j|^2} e^{i  y j } \sum\limits_{ (j_1,j_2,j_3,j_4,j_5) \in \mathcal{NR}(j) } e^{-it(|j_1|^2 - |j_2|^2 + |j_3|^2 -|j_4|^2 + |j_5|^2 -  |j|^2)} (v_{j_1,{\lambda_n}} \bar{v}_{j_2,{\lambda_n}} v_{j_3,{\lambda_n}} \bar{v}_{j_4,{\lambda_n}} v_{j_5,\lambda_n}) (t,x),\notag
\end{align}
where
\begin{equation*}
 \mathcal{NR}(j) = \left\{ (j_1,j_2,j_3,j_4,j_5) \in \mathbb{Z}^5: j_1-j_2+j_3 - j_4 + j_5 -j=0,|j_1|^2-|j_2|^2+|j_3|^2 -|j_4|^2 + |j_5|^2 -|j|^2 \ne 0 \right\}.
\end{equation*}
We first decompose
$e_{\lambda_n} = P_{\ge 2^{-10}}^x e_{\lambda_n} + P_{\le 2^{-10}}^x e_{\lambda_n}$.
By Bernstein's inequality, the Plancherel theorem, \eqref{eq3.933}, Leibnitz's rule, H\"older's inequality, and Lemma \ref{le3.1061}, we have
\begin{align*}
  \left\|P_{\ge 2^{-10}}^x e_{\lambda_n} \right\|_{L_{t }^\frac65  L_x^{\frac65} H_y^1(\mathbb{R} \times \mathbb{R} \times \mathbb{T})}
\lesssim & \left\|P_{\ge 2^{-10}}^x \nabla_x e_{\lambda_n} \right\|_{L_{t,x}^\frac65  H_y^1(\mathbb{R} \times \mathbb{R}  \times \mathbb{T})} \notag\\
\lesssim  & \frac1\lambda_n   \Bigg\| \bigg(\sum_{j} \langle j\rangle^2 |\nabla_x v_{j}(t,x)|^2\bigg)^\frac12    \Bigg\|_{L_{t,x}^6}
\Bigg\| \bigg(\sum_{j} \langle j\rangle^2 |  v_{j}(t,x)|^2\bigg)^\frac12    \Bigg\|_{L_{t,x}^6}^4
 \lesssim_{\|\phi\|_{L_x^2 H_y^1}}     \frac1\lambda_n \|\vec{v}_0 \|_{ H^1_x  h^1}^3,\notag
\end{align*}
where in the second inequality, we use the following estimate
\begin{align*}
& \Bigg\| \bigg(\sum_{j\in \mathbb{Z}} \langle j\rangle^2 \bigg|\sum_{ {(j_1,j_2,j_3,j_4,j_5) \in \mathcal{NR}(j)}  }
e^{-i{\lambda}^2_n t(|j_1|^2 - |j_2|^2 + |j_3|^2 - |j_4|^2 + |j_5|^2  - |j|^2)}   (\nabla_x v_{j_1 } \cdot \overline{v}_{j_2 } \cdot v_{j_3 } \overline{v}_{j_4} v_{j_5})(t,x) \bigg|^2\bigg)^\frac12 \Bigg\|_{L_{t,x}^\frac65}\\
\lesssim  &     \Bigg\| \bigg(\sum_{j} \langle j\rangle^2 |\nabla_x v_{j}(t,x)|^2\bigg)^\frac12    \Bigg\|_{L_{t,x}^6}
\Bigg\| \bigg(\sum_{j} \langle j\rangle^2 |  v_{j}(t,x)|^2\bigg)^\frac12    \Bigg\|_{L_{t,x}^6}^4,
\end{align*}
this is a consequence of H\"older's inequality, and Lemma \ref{le3.1061}.
Thus $P_{\ge 2^{-10}}^x e_{\lambda_n}$ is acceptable when $\lambda_n$ is large enough depending on $\|\phi\|_{L_x^2 H_y^1}$ and $\epsilon_1$. We turn to the estimate of $P_{\le 2^{-10}}^x e_{\lambda_n}$. By integrating by parts, we have
\begin{align*}
- &   \int_0^t e^{i(t-\tau)\Delta_{\mathbb{R}  \times \mathbb{T}}} P_{\le 2^{-10}}^x e_{\lambda_n}(\tau) \,\mathrm{d}\tau\\
= &    \sum_{\substack{j\in \mathbb{Z},\\(j_1,j_2,j_3,j_4,j_5) \in \mathcal{NR}(j) }}
\int_0^t e^{i(t -\tau) \Phi(j_1,j_2,j_3,j_4,j_5, j) -it (|j_1|^2 - |j_2|^2 + |j_3|^2 - |j_4|^2 + |j_5|^2 ) + i  y j }
P_{\le 2^{-10}}^x (v_{j_1,{\lambda_n}} \bar{v}_{j_2,\lambda_n} v_{j_3,\lambda_n} \bar{v}_{j_4 ,\lambda_n} v_{j_5,\lambda_n} )(\tau, x)\,\mathrm{d}\tau
\\
= &    \sum_{\substack{j\in \mathbb{Z},\\(j_1,j_2,j_3,j_4,j_5 ) \in \mathcal{NR}(j)}}
i  e^{ i t \Delta_{\mathbb{R}\times \mathbb{T}}}   e^{i  y j } \frac{P_{\le 2^{-10}}^x}{\Phi(j_1,j_2,j_3,j_4,j_5, j)}
 (v_{j_1,\lambda_n} \bar{v}_{j_2,\lambda_n} v_{j_3,\lambda_n}  \bar{v}_{j_4,\lambda_n}v_{j_5,\lambda_n}  )(0,x)
 \\
&   \sum_{\substack{j\in \mathbb{Z},\\(j_1,j_2,j_3,j_4,j_5) \in \mathcal{NR}(j) }}i  e^{-i t (|j_1|^2 - |j_2|^2 + |j_3|^2 - |j_4|^2 + |j_5|^2)}e^{i y j } \frac{P_{\le 2^{-10}}^x }{ \Phi(j_1,j_2,j_3,j_4,j_5, j)}
    (v_{j_1,\lambda_n} \bar{v}_{j_2,\lambda_n} v_{j_3,\lambda_n} \bar{v}_{j_4,\lambda_n} v_{j_5,\lambda_n} )( t,x )
\\ & -\sum_{\substack{j\in \mathbb{Z},\\(j_1,j_2,j_3,j_4,j_5) \in \mathcal{NR}(j)}}i e^{ i t\Delta_{\mathbb{R} }} e^{i  y j }  \int_0^t  \frac{ie^{-i\tau\Phi(j_1,j_2,j_3,j_4,j_5, j)}}{\Phi(j_1,j_2,j_3,j_4,j_5, j)} \partial_\tau  P_{\le 2^{-10}}^x (v_{j_1,\lambda_n} \bar{v}_{j_2,\lambda_n} v_{j_3,\lambda_n} \bar{v}_{j_4,\lambda_n} v_{j_5,\lambda_n} )(\tau,x) \, \mathrm{d}\tau
 \\
 =: &  A_1 + A_2 + A_3,
\end{align*}
where $\Phi(j_1,j_2,j_3,j_4,j_5, j) = \Delta_{\mathbb{R} } + |j_1|^2 - |j_2|^2 + |j_3|^2 - |j_4|^2 + |j_5|^2 - |j|^2$.

For $A_1$, by Strichartz, Plancherel's theorem, Minkowski, the boundedness of the operator $\frac{P_{\le 2^{-10}}^x}{\Phi(j_1,j_2,j_3,j_4,j_5, j)}$ on $L_x^r(\mathbb{R} )$, $1< r<\infty$, when $(j_1,j_2,j_3,j_4,j_5)\in \mathcal{NR}(j)$, H\"older's inequality, Lemma \ref{le3.1061} and Sobolev, we have
\begin{align*}
 &  \Bigg\|  \sum_{\substack{j\in \mathbb{Z},\\(j_1,j_2,j_3,j_4,j_5 ) \in \mathcal{NR}(j)} }e^{ i t \Delta_{\mathbb{R} \times \mathbb{T}}}  e^{i  y j } \frac{i}{ \Phi(j_1,j_2,j_3,j_4,j_5, j)} P_{\le 2^{-10}}^x (v_{j_1,\lambda_n} \bar{v}_{j_2,\lambda_n} v_{j_3,\lambda_n} \bar{v}_{j_4,\lambda_n}  v_{j_5,\lambda_n})(0,x )\Bigg\|_{L_t^\infty L_x^2 H_{y}^{1}
 \cap L_{t,x}^6 H_y^1}\\
\lesssim &    \Bigg(\sum_{j\in \mathbb{Z}} \langle j\rangle^2 \Bigg \| \sum_{ (j_1,j_2,j_3,j_4,j_5 ) \in \mathcal{NR}(j) }    \frac{ P_{\le 2^{-10}}^x (v_{j_1,\lambda_n} \bar{v}_{j_2,\lambda_n} v_{j_3,\lambda_n} \bar{v}_{j_4,\lambda_n}  v_{j_5,\lambda_n} )(0,x )}{ \Phi(j_1,j_2,j_3,j_4,j_5, j) }\Bigg \|_{L_x^2
   }^2\Bigg)^\frac12  \\
\lesssim &    \Bigg(\sum_{j\in \mathbb{Z}} \langle j\rangle^2   \Bigg( \sum_{ (j_1,j_2,j_3,j_4,j_5 ) \in \mathcal{NR}(j) }    \Bigg \| \frac{ P_{\le 2^{-10}}^x (v_{j_1,\lambda_n} \bar{v}_{j_2,\lambda_n} v_{j_3,\lambda_n} \bar{v}_{j_4,\lambda_n} v_{j_5,\lambda_n} )(0,x )}{ \Phi(j_1,j_2,j_3,j_4,j_5, j) }\Bigg \|_{  L_x^2
 }\Bigg)^2\Bigg)^\frac12  \\
 \lesssim &   \frac1{\lambda_n^2} \Bigg(\sum_{j\in \mathbb{Z}} \langle j\rangle^2 \Big(  \sum_{ (j_1,j_2,j_3,j_4,j_5 ) \in \mathcal{NR}(j) }    \big \| (v_{j_1} \bar{v}_{j_2} v_{j_3} \bar{v}_{j_4}  v_{j_5}  )(0,x )\big \|_{  L_x^2
 }\Big)^2\Bigg)^\frac12  \\
 \lesssim &   \frac1{\lambda_n^2} \Bigg(\sum_{j\in \mathbb{Z}}  \sum_{ (j_1,j_2,j_3 ,j_4,j_5) \in \mathcal{NR}(j) }  \Pi_{  i = 1}^5 \langle j_i\rangle^2 \| v_{j_i}(0,x)\|_{L_x^{10}}^2\Bigg)^\frac12 \\
 & \quad \cdot \bigg(\sup_{j\in \mathbb{Z}} \, \langle j\rangle^2 \sum_{ (j_1,j_2,j_3,j_4,j_5 ) \in \mathcal{NR}(j) }   \langle j_1\rangle^{-2}\langle j_2\rangle^{-2}\langle j_3\rangle^{-2} \langle j_4\rangle^{-2} \langle j_5\rangle^{-2}
\bigg )^\frac12
\lesssim  \frac1{\lambda_n^2}  \Big(\sum_{j\in \mathbb{Z}} \langle j\rangle^2 \|  v_{j,\lambda}(0,x)\|_{L_x^{10} }^2\Big)^\frac52
\lesssim
\frac1{\lambda_n^2} \|\vec{v}_0\|_{ H^1_x h^1}^5.
\end{align*}
We can estimate $A_2$ similarly to $A_1$ by using Plancherel, Minkowski, Strichartz estimate, the boundedness of the operator $\frac{P_{\le 2^{-10}}^x}{\Phi(j_1,j_2,j_3,j_4,j_5, j)}$ on $L_x^r(\mathbb{R} )$, $1< r<\infty$, when $(j_1,j_2,j_3,j_4,j_5)\in \mathcal{NR}(j)$,
\begin{align*}
&\Bigg\|  \sum_{\substack{j\in \mathbb{Z},\\ (j_1,j_2,j_3,j_4,j_5 ) \in \mathcal{NR}(j) }}
\frac{ie^{-it (|j_1|^2 - |j_2|^2 + |j_3|^2 - |j_4|^2 + |j_5|^2 )}  e^{i  y j  } }{\Phi(j_1,j_2,j_3,j_4,j_5,j)} P_{\le 2^{-10}}^x (v_{j_1,\lambda_n}\bar{v}_{j_2,\lambda_n} v_{j_3,\lambda_n} \bar{v}_{j_4,\lambda} v_{j_5,\lambda_n} )
(t ,x)\Bigg\|_{L_t^\infty L_x^2 H_{y}^{1}\cap L_{t,x}^6 H_y^1}\\
 &  \lesssim \Bigg(\sum_{j\in \mathbb{Z}} \Big\|\sum_{(j_1,j_2,j_3,j_4,j_5 )\in \mathcal{NR}(j)} \frac{e^{-it(|j_1|^2-|j_2|^2+|j_3|^2 - |j_4|^2 + |j_5|^2 )} P_{\le 2^{-10}}^x }{\Phi(j_1,j_2,j_3,j_4,j_5, j)}
(v_{j_1,\lambda_n}\bar{v}_{j_2,\lambda_n} v_{j_3,\lambda_n} \bar{v}_{j_4,\lambda_n} v_{j_5,\lambda_n} )(t,x) \Big\|_{L_t^\infty L_x^2 \cap L_{t,x}^6 }^2 \langle j\rangle^2\Bigg)^\frac12\\
&  \lesssim   \Bigg(\sum_{j\in \mathbb{Z}} \Big(\sum_{(j_1,j_2,j_3, j_4,j_5 )\in \mathcal{NR}(j)} \Big\|\frac{P_{\le 2^{-10}}^x (v_{j_1,\lambda_n} \bar{v}_{j_2,\lambda_n} v_{j_3,\lambda_n } \bar{v}_{j_4,\lambda_n} v_{j_5,\lambda_n} )(t,x)}{\Phi(j_1,j_2,j_3,j_4,j_5, j)}\Big\|_{L_t^\infty L_x^2 \cap L_{t,x}^6 }\Big)^2 \langle j\rangle^2 \Bigg)^\frac12 \\
&  \lesssim \lambda_n^{-2} \bigg(\sum_{j\in \mathbb{Z}} \Big(\sum_{(j_1,j_2,j_3,j_4,j_5 )\in \mathcal{NR}(j)} \|v_{j_1}(0,x) v_{j_2}(0,x) v_{j_3}(0,x) v_{j_4}(0,x) v_{j_5}(0,x) \|_{L_x^2}\Big)^2 \langle j\rangle^2\bigg)^\frac12\\
& + \lambda_n^{-2}\Big(\sum_{j\in \mathbb{Z}} \Big(\sum_{(j_1,j_2,j_3,j_4,j_5)\in \mathcal{NR}(j)} \|(i\partial_t + \Delta_{\mathbb{R} })(v_{j_1} \bar{v}_{j_2} v_{j_3} \bar{v}_{j_4} v_{j_5} )(t,x)\|_{L_t^1
L_x^2}\Big)^2 \langle j\rangle^2\Big)^\frac12\\
&  :=    \lambda_n^{-2} (A_{21} + A_{22}),
\end{align*}
we see by H\"older inequality, Lemma \ref{le3.1061}, Sobolev inequality,
\begin{align*}
A_{21} \lesssim & \bigg(\sum_{j\in \mathbb{Z}}\Big(\sum_{(j_1,j_2,j_3,j_4,j_5)\in \mathcal{NR}(j)} \|v_{j_1}(0,x)\|_{L_x^2} \|v_{j_2}(0,x)\|_{L_x^\infty } \|v_{j_3}(0,x)\|_{L_x^\infty}   \|v_{j_4}(0,x)\|_{L_x^\infty }    \|v_{j_5}(0,x)\|_{L_x^\infty }
\Big)^2 \langle j\rangle^2\bigg)^\frac12\\
\lesssim & \bigg(\sum_{j\in \mathbb{Z}} \Big(\sum_{(j_1,j_2,j_3,j_4,j_5)\in \mathcal{NR}(j)} \langle j_1\rangle \|v_{j_1}(0,x)\|_{L_x^2} \Pi_{k= 2}^5
 \left( \langle j_k\rangle \|v_{j_k}(0,x)\|_{H_x^1}\right)
\langle j_1\rangle^{-1} \langle j_2\rangle^{-1} \langle j_3\rangle^{-1} \langle j_4\rangle^{-1} \langle j_5\rangle^{-1}
\Big)^2 \langle j\rangle^2\bigg)^\frac12\\
\lesssim & \left(\sum_{j\in \mathbb{Z}} \langle j\rangle^2 \|v_j(0,x)\|_{H_x^1}^2\right)^\frac52 \sim \|\vec{v}_0\|_{H_x^1 h^1}^5.
\end{align*}
By \eqref{eq3.4new}, we see $A_{22}$ can be controlled by the following terms
\begin{align}\label{eq3.33v86}
\bigg(\sum_{j\in \mathbb{Z}} \Big(\sum_{(j_1,j_2,j_3,j_4,j_5)\in \mathcal{NR}(j)} \|\Delta_{\mathbb{R} } v_{j_1} \bar{v}_{j_2} v_{j_3} \bar{v}_{j_4} v_{j_5} \|_{L_t^1 L_x^2} + \|\nabla v_{j_1} \cdot \nabla v_{j_2} \cdot v_{j_3} v_{j_4} v_{j_5} \|_{L_t^1 L_x^2} \\
+ \Big\|v_{j_1} v_{j_2} \sum_{(p_1,p_2,p_3,p_4,p_5)\in \mathcal{R}(j_3)} v_{p_1} \bar{v}_{p_2} v_{p_3} \bar{v}_{p_4} v_{p_5} v_{j_4} v_{j_5} \Big\|_{L_t^1L_x^2}\Big)^2 \langle j\rangle^2\bigg)^\frac12 \notag 
\end{align}
together with other similar terms when the derivatives are distributed to different components.

Therefore, we only need to estimate \eqref{eq3.33v86}, by H\"older, we have
\begin{align*}
\eqref{eq3.33v86} \lesssim &  \bigg(\sum_{j\in \mathbb{Z}} \Big(\sum_{(j_1,j_2,j_3,j_4,j_5)\in \mathcal{NR}(j)} \|\Delta_{\mathbb{R} } v_{j_1} \bar{v}_{j_2} v_{j_3} \bar{v}_{j_4} v_{j_5} \|_{L_t^1 L_x^2} + \|\nabla v_{j_1} \cdot \nabla v_{j_2} \cdot v_{j_3} v_{j_4} v_{j_5} \|_{L_t^1 L_x^2} \Big)^2
\langle j\rangle^2\bigg)^\frac12\\
&+  \bigg(\sum_{j\in \mathbb{Z}} \Big(\sum_{(j_1,j_2,j_3,j_4,j_5)\in \mathcal{NR}(j)} \Big\|v_{j_1} v_{j_2} \sum_{(p_1,p_2,p_3,p_4,p_5)\in \mathcal{R}(j_3)} v_{p_1} \bar{v}_{p_2} v_{p_3} \bar{v}_{p_4} v_{p_5} v_{j_4} v_{j_5} \Big\|_{L_t^1L_x^2}\Big)^2 \langle j\rangle^2\bigg)^\frac12. \notag %\\
\end{align*}
We now consider the first term, by the H\"older inequality, we see
\begin{align*}
& \bigg(\sum_{j\in \mathbb{Z}} \Big(\sum_{(j_1,j_2,j_3,j_4,j_5)\in \mathcal{NR}(j)} \|\Delta_{\mathbb{R} } v_{j_1} \bar{v}_{j_2} v_{j_3} \bar{v}_{j_4} v_{j_5} \|_{L_t^1 L_x^2} + \|\nabla v_{j_1} \cdot \nabla v_{j_2} \cdot v_{j_3} v_{j_4} v_{j_5} \|_{L_t^1 L_x^2} \Big)^2
\langle j\rangle^2\bigg)^\frac12 \\
& \lesssim \Big(\sum_{j\in \mathbb{Z}} \sum_{(j_1,j_2,j_3,j_4,j_5)\in \mathcal{NR}(j)} \langle j_1\rangle^2 \|\Delta_{\mathbb{R} } v_{j_1} \|_{L_t^5 L_x^{10}}^2
\langle j_2\rangle^2 \|v_{j_2} \|_{L_t^5 L_x^{10}}^2 \langle j_3\rangle^2 \|v_{j_3} \|_{L_t^5 L_x^{10}}^2
\langle j_4\rangle^2 \|v_{j_4} \|_{L_t^5 L_x^{10}}^2 \langle j_5\rangle^2 \|v_{j_5} \|_{L_t^5 L_x^{10}}^2
\Big)^\frac12\\
& + \Big(\sum_{j\in \mathbb{Z}} \sum_{(j_1,j_2,j_3,j_4,j_5)\in \mathcal{NR}(j)} \langle j_1\rangle^2 \|\nabla v_{j_1}\|_{L_t^5 L_x^{10} }^2 \langle j_2 \rangle^2 \|\nabla v_{j_2}\|_{L_t^ 5 L_x^{10} }^2 \langle j_3\rangle^2 \|v_{j_3}\|_{L_t^5  L_x^{10} }^2  \langle j_4\rangle^2 \|v_{j_4}\|_{L_t^5  L_x^{10} }^2   \langle j_5\rangle^2 \|v_{j_5}\|_{L_t^5  L_x^{10} }^2\Big)^\frac12\\
& \lesssim  \Big(\sum_{j \in \mathbb{Z}} \langle j\rangle^2 \|v_j\|_{L_t^5 W_x^{2,10}}^2\Big)^\frac52 + \Big(\sum_{j \in \mathbb{Z}}  \langle j\rangle^2 \|v_j\|_{L_t^5 W_x^{1,10}}^2\Big)^\frac52.
\end{align*}
We now turn to the second term, as in the first term, we only need to use the H\"older inequality, and also note the resonant relation in our term,
\begin{align*}
&  \bigg(\sum_{j\in \mathbb{Z}} \Big(\sum_{(j_1,j_2,j_3,j_4,j_5)\in \mathcal{NR}(j)} \Big\|v_{j_1} v_{j_2} \sum_{(p_1,p_2,p_3,p_4,p_5)\in \mathcal{R}(j_3)} v_{p_1} \bar{v}_{p_2} v_{p_3} \bar{v}_{p_4} v_{p_5} v_{j_4} v_{j_5} \Big\|_{L_t^1L_x^2}\Big)^2 \langle j\rangle^2\bigg)^\frac12\\
\lesssim &  \bigg(\sum_{j\in \mathbb{Z}} \langle j\rangle^2 \Big(\sum_{(j_1,j_2,j_3,j_4,j_5)\in \mathcal{NR}(j)} \|v_{j_1} \|_{L_t^9 L_x^{18}} \|v_{j_2}\|_{L_t^9 L_x^{18}}
\sum_{(p_1,p_2,p_3,p_4,p_5)\in \mathcal{R}(j_3)} \Pi_{k=1}^5 \|v_{p_k}\|_{L_t^9 L_x^{18}}
 \|v_{j_4}\|_{L_t^9 L_x^{18}}  \|v_{j_5}\|_{L_t^9 L_x^{18}}
\Big)^2\bigg)^\frac12\\
\lesssim &  \Bigg(\sum_{j\in \mathbb{Z}} \langle j\rangle^2\bigg(\sum_{j_1-j_2+p_1-p_2+ p_3 -p_4+ p_5 -j_4 +j_5 = j,\atop |j_1|^2 - |j_2|^2 + |p_1|^2 - |p_2|^2 + |p_3|^2  - |p_4|^2 + |p_5|^2 - |j_4|^2 + |j_5|^2 \ne |j|^2}
\|v_{j_1}\|_{L_t^9 L_x^{18}} \|v_{j_2}\|_{L_t^9 L_x^{18}} \Pi_{k=1}^5 \|v_{p_k}\|_{L_t^9 L_x^{18}}
\|v_{j_4}\|_{L_t^9 L_x^{18}} \|v_{j_5}\|_{L_t^9 L_x^{18}}
\bigg)^2\Bigg)^\frac12\\
\lesssim &  \Big(\sum_{j\in \mathbb{Z}} \langle j\rangle^2 \big\||\nabla_x|^\frac29 v_j\big\|_{L_t^9 L_x^\frac{18}5}^2\Big)^\frac92.
\end{align*}
This finish the estimate of $A_2$.
We now turn to the estimate of $A_3$. The estimate of $A_3$ is similar to the estimate of $A_1$ and $A_2$, although the estimate is a little complex. By the Strichartz estimate, we have
\begin{align}
& \Bigg \| \sum_{\substack{j\in \mathbb{Z},\\ (j_1,j_2,j_3,j_4,j_5 ) \in \mathcal{NR}(j) }}  e^{-i t (|j_1|^2 - |j_2|^2 + |j_3|^2 -|j_4|^2 + |j_5|^2 )}   e^{i  y j }
\int_0^t \frac{  e^{i(t -\tau) \Phi(j_1,j_2,j_3,j_4,j_5, j)}}{\Phi(j_1,j_2,j_3,j_4,j_5, j)} \notag \\
&\hspace{4cm} \cdot \partial_\tau P_{\le 2^{-10}}^x (v_{j_1,\lambda_n} \bar{v}_{j_2,\lambda_n} v_{j_3,\lambda_n} \bar{v}_{j_4,\lambda_n} v_{j_5,\lambda_n}
)(\tau, x)  \,\mathrm{d}\tau \Bigg\|_{L_t^\infty L_x^2 H_{y}^{1} \cap L_{t,x}^6 H_{y}^{1}}  \notag \\
\lesssim & \Big\|(i\partial_t + \Delta_{\mathbb{R}  \times \mathbb{T}}) \sum_{j\in \mathbb{Z}} \sum_{(j_1,j_2,j_3,j_4,j_5)\in \mathcal{NR}(j)} e^{-it(|j_1|^2 - |j_2|^2 + |j_3|^2 - |j_4|^2 + |j_5|^2)} e^{iyj}
\notag
 \\
 & \quad \cdot \int_0^t \frac{e^{i(t-\tau)\Phi(j_1,j_2,j_3,j_4,j_5, j)}}{\Phi(j_1,j_2,j_3,j_4,j_5, j)} \partial_\tau P_{\le 2^{-10}}^x
 (v_{j_1,\lambda_n} \bar{v}_{j_2,\lambda_n} v_{j_3,\lambda_n} \bar{v}_{j_4,\lambda_n}v_{j_5,\lambda_n}
 )(\tau,x)\,\mathrm{d}\tau \Big\|_{L_t^1 L_x^2 H_y^1}.\label{eq3.34v86}
\end{align}
By Plancherel, Minkowski, the boundedness of the operator $\frac{P_{\le 2^{-10}}^x}{\Phi(j_1,j_2,j_3,j_4,j_5,j)}$
on $L_x^r(\mathbb{R})$, $1< r<\infty$, when $(j_1,j_2,j_3,j_4,j_5)\in \mathcal{NR}(j)$, \eqref{eq3.4new}, we see \eqref{eq3.34v86} can be controlled by
\begin{align}\label{eq3.1133}
& \Bigg\|\bigg(\sum_{j\in \mathbb{Z}} \langle j\rangle^2 \bigg\|\sum_{(j_1,j_2,j_3,j_4,j_5)\in \mathcal{NR}(j)} e^{-it(|j_1|^2 - |j_2|^2 + |j_3|^2 - |j_4|^2 + |j_5|^2)} \frac{\partial_t P_{\le 2^{-10}}^x
(v_{j_1,\lambda_n} \bar{v}_{j_2,\lambda_n} v_{j_3,\lambda_n} \bar{v}_{j_4,\lambda_n} v_{j_5, \lambda_n} )(t,x)}{\Phi(j_1,j_2,j_3,j_4,j_5, j)}\bigg\|_{L_x^2}^2\bigg)^\frac12\Bigg\|_{L_t^1} \\
\lesssim & \Bigg\|\bigg(\sum_{j\in \mathbb{Z}} \langle j\rangle^2\bigg(\sum_{(j_1,j_2,j_3,j_4,j_5)\in \mathcal{NR}(j)} \bigg\|\frac{\partial_t P_{\le 2^{-10}}^x (v_{j_1,\lambda_n} \bar{v}_{j_2,\lambda_n} v_{j_3,\lambda_n} \bar{v}_{j_4, \lambda_n} v_{j_5, \lambda_n} )(t,x)}{\Phi(j_1,j_2,j_3,j_4,j_5, j)}\bigg\|_{L_x^2}\bigg)^2\bigg)^\frac12\Bigg\|_{L_t^1}\notag\\
\lesssim & \Bigg\|\bigg(\sum_{j\in \mathbb{Z}} \langle j\rangle^2 \bigg(\sum_{(j_1,j_2,j_3,j_4, j_5)\in \mathcal{NR}(j)} \left\|\partial_t(v_{j_1,\lambda_n}
\bar{v}_{j_2,\lambda_n} v_{j_3,\lambda_n} \bar{v}_{j_4, \lambda_n} v_{j_5, \lambda_n}
)\right\|_{L_x^2}\bigg)^2\bigg)^\frac12\Bigg\|_{L_t^1}\notag\\
\lesssim & \lambda_n^{-2} \Bigg\|\bigg(\sum_{j\in \mathbb{Z}} \langle j\rangle^2 \bigg(\sum_{(j_1,j_2,j_3, j_4, j_5 )\in \mathcal{NR}(j)} \Big\|\bar{v}_{j_2} v_{j_3}  \bar{v}_{j_4 } v_{j_5} \sum_{(p_1,p_2,p_3, p_4, p_5 )\in R(j_1)}v_{p_1}\bar{v}_{p_2} v_{p_3} \bar{v}_{p_4} v_{p_5} \Big\|_{L_x^2}\bigg)^2\bigg)^\frac12\Bigg\|_{L_t^1}\notag\\
& + \lambda_n^{-2} \bigg\|\bigg(\sum_{j\in \mathbb{Z}} \langle j\rangle^2\bigg(\sum_{(j_1,j_2,j_3,j_4, j_5 )\in \mathcal{NR}(j)} \left\|\bar{v}_{j_2} v_{j_3} \bar{v}_{j_4} v_{j_5}  \Delta_{\mathbb{R} } v_{j_1}\right\|_{L_x^2}\bigg)^2\bigg)^\frac12\bigg\|_{L_t^1}:=\lambda_n^{-2} (  A_{31} + A_{32}). \notag
\end{align}
We can see by the H\"older inequality and Lemma \ref{le3.1061},
\begin{align}\label{eq3.2465}
A_{32}
\lesssim \bigg\|\Big(\sum_{j\in \mathbb{Z}}
\langle j\rangle^2 \|\Delta_{\mathbb{R} } v_j\|_{L_x^{10}}^2\Big)^\frac12 \Big(\sum_{j\in \mathbb{Z}} \langle j\rangle^2 \|v_j\|_{L_x^{10}}^2\Big)^2\bigg\|_{L_t^1}
\lesssim   \Big(\sum_{j\in \mathbb{Z}} \langle j\rangle^2 \|v_j\|_{L_t^5 W_x^{2,10}}^2\Big)^\frac52,
\end{align}
similarly, by the H\"older inequality and Lemma \ref{le3.1061}, we also have
\begin{align}\label{eq3.2565}
A_{31} \lesssim & \Bigg\|\Big( \Pi_{k=2}^5 \sum_{j_k} \langle j_k\rangle^2 \|v_{j_k}\|_{L_x^{10}}^2\Big)^\frac12
\bigg(\sum_{j_1} \langle j_1\rangle^2 \Big\|\sum_{(p_1,p_2,p_3,p_4,p_5 )\in \mathcal{R}(j_1)} v_{p_1} \bar{v}_{p_2} v_{p_3} \bar{v}_{p_4} v_{p_5} \Big\|_{L_x^{10} }^2\bigg)^\frac12\Bigg\|_{L_t^1} \\
\lesssim & \bigg\|\Big(\sum_j \langle j\rangle^2 \|v_j\|_{L_x^{10} }^2\Big)^\frac12 \bigg\|_{L_t^5}^4
\Bigg\| \bigg(\sum_{j_1} \langle j_1\rangle^2 \Big\|\sum_{(p_1,p_2,p_3,p_4,p_5)\in \mathcal{R}(j_1)} v_{p_1} \bar{v}_{p_2} v_{p_3} \bar{v}_{p_4} v_{p_5} \Big\|_{L_x^{10} }^2\bigg)^\frac12\Bigg\|_{L_t^5}\notag\\
\lesssim & \Big(\sum_j \langle j\rangle^2 \|v_j(t,x)\|_{L_t^5 L_x^{10} }^2\Big)^2  \cdot \bigg\|\Big(\sum_{p\in \mathbb{Z}} \langle p\rangle^2 \|v_p(t,x)\|_{L_x^{50}}^2\Big)^\frac12\bigg\|_{L_t^{25} }^5\notag\\
\lesssim & \Big(\sum_j \langle j\rangle^2 \|v_j(t,x)\|_{L_t^5 L_x^{10} }^2\Big) \cdot \bigg(\sum_{p\in \mathbb{Z}} \langle p\rangle^2 \| v_p(t,x)\|_{L_t^{25} W_x^{\frac25, \frac{50}{21}}}^2\bigg)^\frac52.\notag
\end{align}
Therefore, by \eqref{eq3.34v86}, \eqref{eq3.1133}, \eqref{eq3.2465} and \eqref{eq3.2565}, we have
\begin{align*}
&\  \Bigg \| \sum_{\substack{j\in \mathbb{Z},\\  ( j_1,j_2,j_3,j_4,j_5 ) \in \mathcal{NR}(j) }} e^{-i t (|j_1|^2 - |j_2|^2 + |j_3|^2 - |j_4|^2 + |j_5|^2)} e^{i y j } \int_0^t  \frac{e^{i(t -\tau)\Phi(j_1,j_2,j_3,j_4,j_5 j)}}{\Phi(j_1,j_2,j_3,j_4,j_5, j)}\\
& \hspace{3cm} \cdot \partial_\tau P_{\le 2^{-10}}^x (v_{j_1,\lambda_n} \bar{v}_{j_2,\lambda_n} v_{j_3,\lambda_n} \bar{v}_{j_4,\lambda_n} v_{j_5,\lambda_n}
)(\tau, x)   \,\mathrm{d}\tau \Bigg\|_{L_t^\infty L_x^2 H_{ y}^{ 1}  \cap L_{t,x}^6 H_y^1}\\
&  \lesssim_{\|\phi\|_{L_x^2 H_y^1} }   \lambda_n^{-2}\left(\|\vec{v}_0\|_{ H^2_x h^1} + \|\vec{v}_0\|_{H^2_x h^1 }^3\right).
\end{align*}
So $P_{\le 2^{-10}}^x e_{\lambda_n}$ is acceptable when $\lambda_n$ is large enough depending on $\|\phi\|_{L_x^2 H_y^1}$ and $\epsilon_1$.

\noindent Therefore, $\int_0^t e^{i(t-\tau)\Delta_{\mathbb{R}
\times \mathbb{T}}} e_{\lambda_n}(\tau) \, \mathrm{d}\tau$ is small enough in $L_{t,x}^6 H_y^1 \cap L_t^\infty L_x^2 H_y^1$
when $\lambda_n$ is large enough. We still need to verify the easier assumptions of Theorem \ref{le2.6}.

By Plancherel, \eqref{eq3.4new}, Lemma \ref{le3.1061}, H\"older, and the scattering theorem of the one-discrete-component quintic resonant nonlinear Schr\"odinger system, we have
\begin{align*}
 \left\|V_{\lambda_n} \right\|_{ L_t^\infty L_x^2  H_{ y}^{ 1}(\mathbb{R}\times  \mathbb{R} \times \mathbb{T} )}
\lesssim & \left( \sum_{j\in \mathbb{Z}} \langle j \rangle^2 \bigg(\|v_j(0,x)\|_{L_x^2}^2
+ \bigg\|\sum_{(j_1, j_2, j_3,j_4,j_5 )\in \mathcal{R}(j) } v_{j_1} \overline{v}_{j_2} v_{j_3} \overline{v}_{j_4} v_{j_5} \bigg\|_{L_{t,x}^\frac65 }^2\bigg) \right )^\frac12\\
\lesssim  & \left\|\vec{v}_0\right\|_{ L^2_x h^1} +  \left(\sum_j \langle j \rangle^2 \left\|v_j\right\|_{L_{t,x}^6 }^2 \right)^\frac52
\lesssim   \left\|\vec{v}_0\right\|_{ L^2_x h^1} + \left\|\vec{v}_0\right\|_{ L^2_x h^1}^5,
\end{align*}
and
\begin{align*}
 \left\|V_{\lambda_n}\right\|_{L_t^6 L_x^6 H_y^1(\mathbb{R} \times \mathbb{R}  \times \mathbb{T})}
\lesssim & \left\|V_{\lambda_n}(0,x,y)\right\|_{L_x^2H_y^1} + \left\|(i\partial_t + \Delta_{\mathbb{R} \times \mathbb{T}})V_{\lambda_n} \right\|_{L_{t,x}^\frac65  H_y^1}\\
\sim & \left(\sum\limits_{j\in \mathbb{Z}} \langle j\rangle^2 \left\|v_j(0,x)\right\|_{L_x^2}^2\right)^\frac12
 +  \left\| \left(\sum_{j\in \mathbb{Z}} \langle j\rangle^2 |(i\partial_t + \Delta_{\mathbb{R} })v_j(t,x)|^2 \right)^\frac12\right\|_{L_{t,x}^\frac65}\\
\lesssim &
\left\|\vec{v}_0\right\|_{ L^2_x h^1} +
\left\| \Big(\sum_{j} \langle j \rangle^2 |v_{j}(t,x)|^2\Big)^\frac12 \right\|_{L_{t,x}^6}^5
\lesssim \left\|\vec{v}_0\right\|_{ L^2_x h^1} + \left\|\vec{v}_0\right\|_{ L^2_x h^1}^5.
\end{align*}
Moreover, by Plancherel and \eqref{eq3.935}, we have
\begin{align*}
 \left\|u_{n}(0) - V_{\lambda_n}(0)\right\|_{L_x^2 H_{y}^1}
= \left\|P_{\le \lambda_n^\theta } \vec{ \phi} - \vec{v}_{0}\right\|_{L_x^2 h^1} \lesssim \epsilon_1,
\end{align*}
when $n$ large enough. Applying Theorem \ref{le2.6}, we conclude that for $\lambda_n$ (depending on $\vec{v}_0$) large enough, the solution $u_{n}$ of \eqref{eq3.833} exists globally and
\begin{align*}
\left\|u_{n} - V_{\lambda_n} \right\|_{L_t^\infty L_x^2 H_{y}^1 \cap L_{t,x}^6 H_y^1(\mathbb{R} \times \mathbb{R}  \times \mathbb{T})} \lesssim \epsilon_1,
\end{align*}
which ends the proof in the case $t_n = 0$.

When $t_n\to \pm \infty$, $v_j$ is the solution of the one-discrete-component quintic resonant nonlinear Schr\"odinger system with
\begin{align*}
\left\|v_j(t,x) - e^{it\Delta_{\mathbb{R} }} v_{0,j}(x)\right\|_{L_x^2 h_j^1} \to 0, \text{ as }  t\to \pm \infty.
\end{align*}
By the argument in the previous case, we can also obtain the result.

\section{Global well-posedness and scattering for the quintic resonant nonlinear Schr\"odinger system}\label{se6}
In this section, we prove the global well-posedness and scattering for the one-discrete-component quintic resonant nonlinear Schr\"odinger system, which is used in the proof of Theorem \ref{pr5.9}. In the meantime, we also prove the conjecture of the global and scattering of the two-component quintic resonant nonlinear Schr\"odinger system made in \cite{HP}. We will show the existence of the almost-periodic solution by the energy induction argument, and the almost-periodic solution is precluded by using the arguments of B. Dodson\cite{D1,D2,D3}.

For the one-discrete-component quintic resonant nonlinear Schr\"odinger system
\begin{equation}\label{eq1.733}
\begin{cases}
i\partial_t u_j + \Delta_{\mathbb{R} } u_j = \sum\limits_{(j_1,j_2,j_3,j_4,j_5) \in \mathcal{R}(j) } u_{j_1} \bar{u}_{j_2} u_{j_3}\bar{u}_{j_4}u_{j_5},\\
u_j(0) = u_{0,j},\ j\in \mathbb{Z},
\end{cases}
\end{equation}
 where
 \begin{align*}
 \mathcal{R}(j) =
 \left\{ j_1,j_2,j_3,j_4,j_5 \in \mathbb{Z}: j_1-j_2+j_3-j_4+j_5= j, \, |j_1|^2 - |j_2|^2 + |j_3|^2 - |j_4|^2 + |j_5|^2 = |j|^2 \right\}.
\end{align*}
For $\vec{u} = \left\{u_j\right\}_{j \in \mathbb{Z}}$, and we will frequently denote the nonlinearity in \eqref{eq1.733} by $\vec{F}(\vec{u})$, that is,
\begin{align*}
\vec{F}(\vec{u})  = \left\{ \vec{F}_j(\vec{u})\right\}_{j\in \mathbb{Z}},
\end{align*}
where $\vec{F}_j(\vec{u}) = \sum\limits_{(j_1,j_2,j_3,j_4,j_5) \in \mathcal{R}(j) } u_{j_1} \bar{u}_{j_2} u_{j_3}\bar{u}_{j_4}u_{j_5}$.
we have the following scattering result.
\begin{theorem}[Scattering for the one-discrete-component quintic resonant nonlinear Schr\"odinger system]\label{th1.2}
Let $E>0$, for any initial data $\vec{u}_0$ satisfying
\begin{equation*}
\|\vec{u}_0\|_{L_x^2 h^1} :=   \left\|\Big(\sum\limits_{j\in \mathbb{Z}} \langle j\rangle^2 |u_{0,j}(x)|^2\Big)^\frac12 \right\|_{L^2(\mathbb{R} )} \le E,
\end{equation*}
there exists a global solution $\vec{u} = \left\{u_j\right\}_{j\in \mathbb{Z}}$ to \eqref{eq1.733} satisfying
\begin{equation}\label{eq5.19new}
\left\|\vec{u}\right\|_{L_{t,x}^6 h^1(\mathbb{R}\times \mathbb{R}  \times \mathbb{Z})}   \le C,
\end{equation}
for some constant $C $ depends only on $\left\|\vec{u}_0\right\|_{L_x^2 h^1}$.
In addition, the solution scatters in $L_x^2 h^1$ in the sense that there exists $\left\{u_j^{\pm }\right\}_{j \in \mathbb{Z}}  \in  L^2_x h^1$
such that
\begin{equation*}
  \bigg \| \Big( \sum\limits_{j\in \mathbb{Z}} \langle j\rangle^2 |  u_j(t) - e^{it\Delta_{\mathbb{R} }} u_j^{\pm  }|^2 \Big)^\frac12  \bigg\|_{L^2(\mathbb{R} )} \to 0, \text{ as } t\to \pm \infty.
\end{equation*}
\end{theorem}
The one-discrete-component quintic resonant nonlinear Schr\"odinger system has the following conserved quantities:
\begin{align*}
& \text{mass: }    \quad   \mathcal{M}_{a,b,c}(\vec{u}(t))   = \int_{\mathbb{R}   } \sum_{j\in \mathbb{Z}} (a+ bj +  c |j|^2) |u_j(t,x )|^2\,\mathrm{d}x, \text{ where } a,b,c \in \mathbb{R},\\
&  \text{   energy:  }       \quad \mathcal{E}(\vec{u}(t))    =    \int_{\mathbb{R}   } \sum_{j\in \mathbb{Z}}\frac12 |\nabla u_j(t,x)|^2  + \frac16  \sum_{\substack{j\in \mathbb{Z}, \\n\in \mathbb{N}}} \Big|\sum_{ \substack{
j_1-j_2 + j_3  =j, \\
 |j_1|^2 - |j_2|^2 + |j_3|^2 = n} } (  u_{j_1} \bar{u}_{j_2} u_{j_3})(t,x) \Big|^2 \,\mathrm{d}x\\
   & \hspace{3cm}  =    \int_{\mathbb{R}   } \sum_{j\in \mathbb{Z}}\frac12 |\nabla u_j(t,x)|^2  + \frac16
 \sum_{j\in \mathbb{Z}} \sum_{ \substack{j_1,j_2,j_3,j_4,j_5  \in \mathbb{Z},\\
j_1-j_2 + j_3 - j_4 + j_5 = j, \\
 |j_1|^2 - |j_2|^2 + |j_3|^2 - |j_4|^2 + |j_5|^2 = |j|^2} }   ( u_{j_1} \bar{u}_{j_2} u_{j_3} \bar{u}_{j_4} u_{j_5} \bar{u}_{j } )(t,x)  \,\mathrm{d}x.
\end{align*}
 The following proposition proves the local wellposedness and small-data scattering for the one-discrete-component quintic resonant nonlinear Schr\"odinger system \eqref{eq1.733}.
\begin{proposition}[Local well-posedness and small data scattering]\label{lea1}
Let $\vec{u}_0 \in  L_x^2 h^1$ satisfies $\left\|\vec{u}_0\right\|_{ L^2_x h^1} \le E$, then

(1) There exists an open interval $ 0 \in I$ and a unique solution $\vec{u}(t)$ of \eqref{eq1.733} in $C_t^0 L_x^2 h^1(I\times \mathbb{R} \times \mathbb{Z} ) \cap L_{t,x}^6 h^1(I \times \mathbb{R} \times \mathbb{Z})$.

(2) There exists a small positive constant $E_0$ such that if $\mathcal{E}(\vec{u}_0)  \le E_0$, $\vec{u}(t)$ is global and scatters in positive and negative infinite time.

(3) If $\vec{u}_0 \in  H_x^k h^\sigma$ for some $\sigma \ge 1$ and $k\ge 0$, then $\vec{u}(t) \in C_t^0 H_x^k h^\sigma(I\times \mathbb{R} \times \mathbb{Z}  )$.

(4) If $\|\vec{u}\|_{L_{t,x}^6 h^1(I_{max} \times \mathbb{R} \times \mathbb{Z})} < \infty$, where $I_{max}$ is the maximal lifespan of the solution, we have $I_{max} = \mathbb{R}$, and $\vec{u}$ scatters in $L_x^2 h^1$.
\end{proposition}
\begin{proof}
The proof follows from a simple fixed-point theorem, once we have established the fundamental nonlinear estimate. By the Strichartz estimate, we see
\begin{align*}
\| \vec{u} \|_{L_{t,x}^6 h^1(I \times \mathbb{R})} \lesssim
\|\vec{u}_0\|_{L_x^2 h^1}+  \bigg\| \Big(\sum_{j\in \mathbb{Z}} \Big| \sum_{(j_1,j_2,j_3,j_4,j_5) \in \mathcal{R}(j)} u_{j_1} \bar{u}_{j_2} u_{j_3} \bar{u}_{j_4} u_{j_5} \Big|^2\Big)^\frac12  \bigg\|_{L_{t,x}^\frac65(I \times \mathbb{R}  )}.
\end{align*}
The first term on the right-hand side is the $L^2_x h^1 $-norm.
For the second, we compute using Lemma \ref{le3.1061} that
\begin{align*}
 \Bigg\| \bigg(\sum_{j\in \mathbb{Z}} \langle j\rangle^2 \Big | \sum_{(j_1,j_2,j_3,j_4,j_5) \in \mathcal{R}(j)} u_{j_1} \bar{u}_{j_2} u_{p_3} \bar{u}_{j_4} u_{j_5} \Big
|^2\bigg)^\frac12 \Bigg\|_{L_{t,x}^\frac65( I\times \mathbb{R} )}
\lesssim \|\vec{u}\|_{L_{t,x}^6 h^1(I \times \mathbb{R} \times \mathbb{Z})}^5.
\end{align*}
Consequently, we obtain
\begin{align*}
\|\vec{u}\|_{L_{t,x}^6 h^1(I \times \mathbb{R} \times \mathbb{Z})} \lesssim \|e^{it\Delta_{\mathbb{R} }} \vec{u}_0\|_{L_{t,x}^6 h^1(I \times \mathbb{R} \times \mathbb{Z})} + \|\vec{u}\|_{L_{t,x}^6 h^1(I \times \mathbb{R} \times \mathbb{Z})}^5.
\end{align*}
This and the Strichartz estimate
\begin{align*}
\|e^{it\Delta_{\mathbb{R}  }} \vec{u}_0\|_{L_{t,x}^6 h^1(I \times \mathbb{R} \times \mathbb{Z})} \lesssim \|\vec{u}_0\|_{L^2_x h^1 } \lesssim E
\end{align*}
allows one to run a classical fixed-point argument in $L_{t,x}^6 h^1(I \times \mathbb{R} \times \mathbb{Z}) \cap C_t^0  L^2_x h^1(I\times \mathbb{R}  \times \mathbb{Z})$ provided $I$ or $E$ is small enough.
The rest of the Proposition follows from standard arguments.
\end{proof}
\begin{remark}
The norm $L_{t,x}^6 h^1$ can be regarded as the ``scattering norm" which plays the same role as $L^{10}_{t,x}(\mathbb{R}\times \mathbb{R}^3)$ in \cite{Iteam1} and $L^6_{t,x}(\mathbb{R}\times \mathbb{R})$ in \cite{D2}.
\end{remark}
\begin{remark}\label{re6.433}
By the Strichartz estimate and Theorem \ref{th1.2}, we have
\begin{align}\label{eq5.64eq}
\left\| \vec{u} \right\|_{L_t^q L_x^r h^1(\mathbb{R}\times \mathbb{R} )} \lesssim \left\|\vec{\phi}\right\|_{L_x^2 h^1}, \text{  where } \frac2q + \frac1r = \frac12, \ 4 \le q \le \infty.
\end{align}
\end{remark}
We also have the persistence of regularity:
\begin{corollary}[Persistence of regularity]\label{co6.533}
Suppose $\vec{u}_0 \in L_x^2 h^1$ and $\vec{u} \in C_t^0 L_x^2 h^1(\mathbb{R} \times \mathbb{R} \times \mathbb{Z})$ is the solution of \eqref{eq1.733}.
Suppose also that $\vec{v}_0 \in H^4_x h^5$ satisfies
\begin{align*}
\left\|\vec{u}_0 - \vec{v}_0 \right\|_{L_x^2 h^1} \lesssim \epsilon,
\end{align*}
and that $\vec{v}$ is the solution to \eqref{eq1.733} with initial data $%\vec{v}(0) =
 \vec{v}_0$ at time $t  = 0$. Then it holds that
\begin{align*}
& \|\vec{u} - \vec{v} \|_{L_t^\infty L_x^2 h^1 \cap \vec{W}(\mathbb{R})}  \lesssim_{\|\vec{u}_0\|_{L_x^2 h^1}}  \epsilon,\\
& \left\|\vec{v}\right\|_{L_t^\infty H_x^4 h^5 } +
 \bigg\| \Big( \sum_{j \in \mathbb{Z} } \langle j \rangle^{10}  \big|  (\partial_x^4 v_j)(t,x)\big|^2 \Big)^\frac12   \bigg\|_{L_{t,x}^6(\mathbb{R} \times \mathbb{R} )}
+ \bigg\| \Big( \sum_{j \in \mathbb{Z} } \langle j \rangle^{10}  \big|   v_j(t,x)\big|^2 \Big)^\frac12   \bigg\|_{L_{t,x}^6(\mathbb{R} \times \mathbb{R} )}
 \lesssim_{\|\vec{u}_0\|_{L_x^2 h^1}} \|\vec{v}_0\|_{H^4_x h^5},
\end{align*}
and there exists $\vec{v}^\pm \in H^4_x h^5$ such that
\begin{align*}
\bigg\| \Big(\sum_{j\in \mathbb{Z} } \langle j\rangle^2 \big|(v_j(t)- e^{it\Delta_{\mathbb{R} }} v_j^\pm)(x)\big|^2\Big)^\frac12 \bigg\|_{L_x^2(\mathbb{R} )} \to 0, \text{ as } t\to \pm\infty.
\end{align*}
\end{corollary}
By Proposition \ref{lea1}, we see to show the global well-posedness and scattering of the solution $\vec{u}$ to \eqref{eq1.733}, it is enough to show the
$L_{t,x}^6 h^1$ norm of $\vec{u}$ is finite.
 We can further reduce the scattering norm to a weaker one. Before giving the weaker scattering criterion, we first give a stronger
nonlinear estimate (\ref{le4.5v39}) for the quintic resonant nonlinearity than \eqref{eq3.31v83} based on the following result:
\begin{lemma}\label{modif111}
For $\frac{3}{8}<\beta<1$, we obtain
\begin{equation}\label{eqmodif2}
\sup_{j\in \mathbb{Z} } \bigg\{ \langle j \rangle^2 \sum_{ \substack{ (p_1,p_2,p_3,p_4,p_5) \in \mathcal{R}(j),\\ |p_5|\sim \max\left( |j|,|p_2|,|p_4|\right) } } \langle p_1 \rangle^{-2\beta} \langle p_2 \rangle^{-2\beta} \langle p_3 \rangle^{-2\beta} \langle p_4 \rangle^{-2\beta} \langle p_5 \rangle^{-2}  \bigg\} \lesssim 1.
\end{equation}
\end{lemma}
\emph{Proof.} Without loss of generality, we assume that $|p_1|\leq |p_3|\leq |p_5|$, $|p_2| \leq |p_4|$ and $\max(|j|, |p_4|) \sim |p_5|$. Then, we have:
\begin{align*}
\left| p_3 - \frac{p_2 + p_4 + j - p_1}2\right|^2 = \frac{2(|p_2|^2 + |p_4|^2 + |j|^2 - |p_1|^2) -  |p_2 + p_4 +j - p_1|^2}4.
\end{align*}
which means $p_3$ is on the degenerate circle $\mathcal{C}$. Then:
\begin{align*}
\sum\limits_{\substack{ p_3 \in \mathcal{C},\\ |p_3| \ge \max(|p_1|, |p_2|, |p_4|)}} \langle p_3\rangle^{-2\beta} \lesssim \langle \max(|p_1|, |p_2|, |p_4| )\rangle^{-2\beta}.
\end{align*}
Thus, we have:
\begin{align*}
&   \sum\limits_{\substack{(p_1,p_2,p_3,p_4,p_5) \in \mathcal{R}(j),\\ |p_2| \le |p_4| \le |p_3|,\\
|p_1| \le |p_3| \le |p_5|}} \langle p_1 \rangle^{-2\beta} \langle p_2 \rangle^{-2\beta} \langle p_3 \rangle^{-2\beta} \langle p_4 \rangle^{-2\beta} {\langle p_5 \rangle^{-2} } {\langle j\rangle^2}  \\
\lesssim & \sum_{p_1,p_2,p_4 \in \mathbb{Z} } \langle p_1 \rangle^{- 2 \beta} \langle p_2 \rangle^{-2\beta} \langle p_4 \rangle^{-2\beta} \sum_{\substack{|p_3 | \ge \max(|p_1|, |p_2|, |p_4|),\\ (p_1,p_2, p_3,p_4, p_2 + p_4 + j - p_1 - p_3) \in \mathcal{R}(j)}} \langle p_3 \rangle^{-2\beta} \\
&\lesssim  \sum_{p_1,p_2,p_4\in \mathbb{Z} } \langle p_1\rangle^{-2\beta} \langle p_2 \rangle^{-2\beta} \langle p_4 \rangle^{-2\beta} \langle |p_1| + |p_2| + |p_4| \rangle^{-2\beta}
 \lesssim   1.
\end{align*}
The proof of Lemma \ref{modif111} is now complete.
\begin{lemma}[Nonlinear estimate]\label{le4.5v39}
For sequence $\left\{u_j\right\}_{j\in \mathbb{Z}} \in h^1(\mathbb{Z})$, then
\begin{align}\label{eq4.4v39}
\left\| \sum_{(j_1,j_2,j_3,j_4,j_5) \in \mathcal{R}(j)} u_{j_1} \bar{u}_{j_2} u_{j_3} \bar{u}_{j_4} u_{j_5} \right\|_{h^1} \lesssim \left\|\vec{u} \right\|_{h^1} \left\|\vec{u} \right\|_{h^\beta}^4,
\end{align}
where $\frac{3}{8}< \beta<1$.
\end{lemma}
\begin{proof}
We recall and use Lemma \ref{modif111}:
\begin{equation}\label{eq4.5.1}
\sup_{j\in \mathbb{Z} } \bigg\{ \langle j \rangle^2 \sum_{ \substack{ (j_1,j_2,j_3,j_4,j_5) \in \mathcal{R}(j),\\ |j_5|\sim \max\left( |j|,|j_2|,|j_4|\right)} } \langle j_1 \rangle^{- 2\beta} \langle j_2 \rangle^{- 2\beta} \langle j_3 \rangle^{- 2\beta} \langle j_4 \rangle^{- 2\beta} \langle j_5 \rangle^{-2}  \bigg\} \lesssim 1,
\end{equation}
we also refer to the proof of the same estimate in Lemma \ref{modif2} for the two-discrete-component quintic resonant nonlinear Schr\"odinger system case later.
\noindent We can obtain similar estimates for the cases when $|j_1|\sim \max \left(|j|,|j_2|,|j_4|\right)$ or $|j_3|\sim \max\left( |j|,|j_2|,|j_4|\right)$ as well. By (\ref{eq4.5.1}), we have
\begin{equation*}
\aligned
\left\|\vec{F}(\vec{u})\right\|^2_{h^1}&=\sum_{j \in \mathbb{Z} } \langle j \rangle^2 \Big|\sum_{(j_1,j_2, j_3,j_4,j_5) \in \mathcal{R}(j)}u_{j_1}\bar{u}_{j_2}u_{j_3}\bar{u}_{j_4}u_{j_5} \Big| \\
&\lesssim \sum_{j \in \mathbb{Z} } \langle j \rangle^2\bigg(\Big|\sum_{\substack{ (j_1,j_2,j_3,j_4,j_5) \in \mathcal{R}(j),\\j_1 \sim \max\left( |j|,|j_2|,|j_4|\right)} }  u_{j_1}\bar{u}_{j_2}u_{j_3}\bar{u}_{j_4}u_{j_5} \Big|+ \Big|\sum_{\substack{(j_1,j_2,j_3,j_4,j_5) \in \mathcal{R}(j),\\j_3 \sim \max\left( |j|,|j_2|,|j_4|\right)}}u_{j_1}\bar{u}_{j_2}u_{j_3}\bar{u}_{j_4}u_{j_5}\Big|\\
& \qquad + \Big|\sum_{\substack{ (j_1,j_2,j_3,j_4,j_5) \in \mathcal{R}(j),\\j_5 \sim  \max\left( |j|,|j_2|,|j_4|\right)} }u_{j_1}\bar{u}_{j_2}u_{j_3}\bar{u}_{j_4}u_{j_5}\Big|\bigg)\\
&\lesssim \left\|\vec{u}\right\|_{h^\beta}^8 \cdot \left\|\vec{u}\right\|_{h^1}^2.
\endaligned
\end{equation*}
\end{proof}
By using the nonlinear estimate in Lemma \ref{le4.5v39} together with a bootstrap argument (see also Theorem \ref{th2.946} for similar argument), we can obtain the following result:
\begin{lemma}[Scattering criterion] \label{le4.7v39}
If the solution $\vec{u}$ of the Cauchy problem \eqref{eq1.733} satisfies for $\frac{3}{8}< \beta <1$,
\begin{equation*}
\left\|\vec{u}\right\|_{L_{t,x}^6 h^\beta (\mathbb{R}\times \mathbb{R} \times \mathbb{Z}
)}<\infty.
\end{equation*}
Then we have
\begin{equation*}
\left\|\vec{u}\right\|_{L_{t,x}^6 h^1(\mathbb{R}\times \mathbb{R} \times \mathbb{Z}
)}<\infty,
\end{equation*}
and therefore the solution to \eqref{eq1.733} scatters in $L_x^2 h^1$.
\end{lemma}
\subsection{Existence of the almost-periodic solution}
To prove the one-discrete-component quintic resonant nonlinear Schr\"odinger system is globally well-posed and scattering, according to Lemma \ref{le4.7v39}, it suffices to prove that if $\vec{u}$ is a solution of the one-discrete-component quintic resonant nonlinear Schr\"odinger system \eqref{eq1.733}, then
\begin{align*}
\left\|\vec{u}\right\|_{L_{t,x}^6 h^\beta (\mathbb{R}\times \mathbb{R}\times \mathbb{Z})} < \infty
\end{align*}
for all $\vec{u}_0 \in L_x^2 h^1(\mathbb{R}\times \mathbb{Z})$,
where $\beta$ is a constant satisfying $\frac{3}{8}< \beta <1$.

For $\vec{u}$ solving the one-discrete-component quintic resonant nonlinear Schr\"odinger system with $\vec{u}_0 \in L_x^2 h^1(\mathbb{R}\times \mathbb{Z})$, we define
\begin{align*}
A(m) = \sup\left\{\left\|\vec{u}\right\|_{L_{t,x}^6 h^\beta(\mathbb{R} \times \mathbb{R} \times \mathbb{Z})}  \le m \right\},
\end{align*}
and
\begin{align*}
m_0 = \sup\left\{ m : A(m')< \infty, \forall\, m' < m \right\}.
\end{align*}
If we can show $m_0 = \infty$, the global well-posedness and scattering for Cauchy problem \eqref{eq1.733} are established. By a simple application of Theorem \ref{pr3.2v15} with $\alpha = 1$, $d=1$ and $\mathbb{D} = \mathbb{Z}$, we have
\begin{proposition}[Linear profile decomposition in $L_x^2 h^1(\mathbb{R} \times \mathbb{Z})$]
Let $\{\vec{u}_{n}\}_{n\ge 1}$ be a bounded sequence in $L_x^2 h^1(\mathbb{R}  \times \mathbb{Z})$. Then (after passing to a subsequence if necessary) there
exists $K^*\in  \{0,1, \cdots \} \cup \{\infty\}$, functions $\{\vec{\phi}^{k}\}_{k=1}^{K^*} \subseteq L_x^2 h^1$, group elements $\{g_n^k\}_{k=1}^{K^*} \subseteq G$,
and times $\{t_n^k\}_{k=1}^{K^*} \subseteq \mathbb{R}$ so that defining $\vec{w}_{n}^K$ by
\begin{align*}
\vec{u}_n(x) =    \sum_{k=1}^K g_n^k e^{it_n^k \Delta_{\mathbb{R} }} \vec{\phi}^k + \vec{w}_n^K(x)% \\
         :=   \sum_{k=1}^K \frac1{(\lambda_n^k)^\frac12 } e^{ix\xi_n^k} \left(e^{it_n^k \Delta_{\mathbb{R} } } \vec{\phi}^k\right)\left(\frac{x-x_n^k}{\lambda_n^k} \right) + \vec{w}_n^K(x ),
\end{align*}
we have the following properties:
\begin{align*}
\limsup_{n\to \infty} \left\|e^{it\Delta_{\mathbb{R} }} \vec{w}_n^K\right\|_{L_{t,x}^6 h^{1-\epsilon_0}(\mathbb{R} \times \mathbb{R}  \times \mathbb{Z} )} \to 0, \ \text{ as } K\to \infty, \\
e^{-it_n^k \Delta_{\mathbb{R} }} (g_n^k)^{-1} \vec{w}_n^K \rightharpoonup 0  \text{ in } L_x^2 h^1,  \text{ as  } n\to \infty,\text{ for each }  k\le K,\\
\sup_{K} \lim_{n\to \infty} \left(\left\|\vec{u}_n\right\|_{L_x^2 h^1}^2 - \sum_{k=1}^K \left\|\vec{\phi}^k\right\|_{L_x^2 h^1}^2 - \left\|\vec{w}_n^K\right\|_{L_x^2 h^1}^2\right) = 0,
\end{align*}
and lastly, for $k\ne k'$, and $n\to \infty$,
\begin{align*}
\frac{\lambda_n^k}{\lambda_n^{k'}} + \frac{\lambda_n^{k'}}{\lambda_n^k} + \lambda_n^k \lambda_n^{k'} |\xi_n^k - \xi_n^{k'}|^2 + \frac{|x_n^k-x_n^{k'}|^2 } {\lambda_n^k \lambda_n^{k'}}  + \frac{|(\lambda_n^k)^2 t_n^k -(\lambda_n^{k'})^2 t_n^{k'}|}{\lambda_n^k \lambda_n^{k'}} \to \infty.
\end{align*}
\end{proposition}
Then by using the argument in \cite{CMZ,TVZ1,TVZ2} and the above linear profile decomposition, we can obtain
\begin{theorem}[Existence of an almost periodic solution]\label{th4.9v51}
Assume $m_0 < \infty$, there exists a solution $\vec{u}\in C_t^0 L_x^2 h^1 \cap L_{t,x}^6 h^{\beta}(I \times \mathbb{R} \times \mathbb{Z})$ for $\frac{3}{8} < \beta < 1$ to the one-discrete-component quintic resonant nonlinear Schr\"odinger system \eqref{eq1.733} with $M(\vec{u}) = m_0$, which is almost periodic in the sense that there exists $(x(t), \xi(t), N(t)) \in \mathbb{R}\times \mathbb{R} \times \mathbb{R}^+$ such that for any $\eta > 0$, there exists $C(\eta) > 0$ such that for $t\in I$,
\begin{equation}\label{eq4.9.1}
 \int_{|x-x(t)|\ge \frac{C(\eta)}{N(t)}} \left\|\vec{u}(t,x)\right\|_{h^1}^2 \,\mathrm{d}x + \int_{|\xi- \xi(t)|\ge C(\eta) N(t)} \left\|\hat{\vec{u}}(t,\xi)\right\|_{h^1}^2 \,\mathrm{d}\xi
 %\right)
 < \eta,
\end{equation}
where $I$ the maximal lifespan interval. Moreover, we can take $N(0) = 1$, $x(0) = \xi(0) = 0$, $N(t) \le 1$ on $[0,\infty)$ with $[0,\infty) \subset I$, and
\begin{align*}
|N'(t) | + |\xi'(t)| \lesssim N(t)^3.
\end{align*}
\end{theorem}
We now fix three small constants
\begin{align}\label{eq4.15v51}
0<\epsilon_3 \ll \epsilon_2 \ll \epsilon_1<1
\end{align}
 satisfying
\begin{equation}\label{eq4.16v51}
|N'(t) |+|\xi'(t)|\leq 2^{-20}{\epsilon_1^{-\frac{1}{2}}} {N(t)^3},
\end{equation}
\begin{equation}\label{eq4.17v51}
\int_{|x-x(t)\geq \frac{2^{-20}\epsilon_3^{-\frac{1}{2}}}{N(t)} }\left\|\vec{u}(t,x)\right\|_{h^1}^2 \mathrm{d}x + \int_{|\xi-\xi(t)|\geq 2^{-20}\epsilon_3^{-\frac{1}{2}}N(t)} \left\|\hat{\vec{u}}(t,\xi)\right\|_{h^1}^2 \mathrm{d}\xi
 <\epsilon_2^2.
\end{equation}
By using the argument in \cite{D3,D1,D2,Killip-Visan1}, we have the following facts on the almost periodic solution in the above theorem:
\begin{lemma}
(1) There exists $\delta(\vec{u}) > 0$ such that for any $t_0 \in I$,
\begin{align*}
\left\|\vec{u}\right\|_{L_{t,x}^6 h^{\beta} ([t_0, t_0 + \frac{\delta}{N(t_0)^2}] \times \mathbb{R})} \sim \left\|\vec{u}\right\|_{L_{t,x}^6 h^{\beta}([t_0 - \frac\delta{N(t_0)^2}, t_0] \times \mathbb{R})} \sim 1.
\end{align*}
(2) If $J$ is an interval with $\left\|\vec{u}\right\|_{L_{t,x}^6 h^{\beta}(J\times \mathbb{R})} = 1$, then for $t_1,t_2 \in J$, $N(t_1) \sim_{m_0} N(t_2)$, and $|\xi(t_1) - \xi(t_2)| \lesssim N(J)$, where $N(J) := \sup\limits_{t\in J} N(t).$ In addition,
\begin{align*}
N(J) \sim \int_{J} N(t)^3 \,\mathrm{d}t \sim \inf\limits_{t\in J} N(t),
\end{align*}
\end{lemma}
\subsection{Long time Strichartz estimate}\label{subse4.2}
Now we are ready to establish the long time Strichartz estimate which is a strong tool for us to exclude the almost periodic solution. Long time Strichartz estimate is first developed by B. Dodson in solving mass critical NLS problems and has been proved as an important technique in the area of dispersive evolution equation. See \cite{BMMZ,CGZ,D4,KV2,MMZ,Mu,V,YZ} for the application of the long time Stricharz estimate for other dispersive problems. Similar as the mass-critical nonlinear Schr\"odinger equations, the long time Stricharz estimate mainly helps us in two aspects, i.e. obtaining the additional regularity of the solution for the rapid frequency cascade scenario and controlling the error term which appears in the frequency localized interaction Morawetz inequality. Here are several significant points we want to mention here. Firstly, different from the case when $d \geq 3$, we still need to use function spaces ($U^p_{\Delta},V^p_{\Delta}$) based on atomic spaces and variation spaces (see \cite{HTT,HTT1} for more information about these function spaces). Secondly, we will utilize long time Stricharz estimate to obtain additional regularity of the almost periodic solution and then exclude the rapid frequency cascade scenario (see Subsection \ref{subse4.4}), we refer to \cite{KTV} for the elaboration of these spaces. Thirdly, when we use the interaction Morawetz identity (see Subsection \ref{subse4.3}), the right hand side would involve $\dot{H}^1 h^\beta $ norm, but the regularity space is $L^2 h^\beta$, which reasonably indicates a low-frequency truncation. In this case, we use the long time Stricharz estimate to control the high-frequency error term. We refer to B. Dodson's work \cite{D3,D1,D2} for more information (motivation and details) about the long time Stricharz estimate.

First, we construct function spaces $U_{\Delta}^p(L^2h^\beta; \mathbb{R} )$ and $V_{\Delta}^p(L^2h^\beta; \mathbb{R})$ as in \cite{D1,D2} and we can get similar estimates with little modifications.

\begin{definition}[$U_{\Delta}^p(H; \mathbb{R} )$ spaces]
 Let $1\leq p < \infty$, and $H$ be a complex Hilbert space. A $U_{\Delta}^p(H; \mathbb{R} )$-atom is a piecewise defined function, $\vec{a}:\mathbb{R} \rightarrow H$, and
\begin{align*}
 \vec{a} = \sum_{k=1}^{K}\chi_{[t_{k-1},t_k)}e^{it\Delta} \vec{\phi}_{k-1},
\end{align*}
where $\left\{t_k\right\}_{k=0}^{K} \in \mathcal{Z}$ and $\left\{\vec{\phi}_k\right\}_{k=0}^{K-1} \subset H$ with $\sum\limits_{k=0}^{K} \left\|\vec{\phi}_k\right\|^p_H=1$. Here we let $\mathcal{Z}$ be the set of finite partitions $-\infty<t_0<t_1<...<t_K\leq \infty$ of the real line.

The atomic space $U_{\Delta}^p(H; \mathbb{R})$ consists of all functions $\vec{u}:\mathbb{R}\rightarrow H$ such that
$\vec{u} =\sum\limits_{j=1}^{\infty}\lambda_j \vec{a}_j$ for $U_{\Delta}^p$-atoms $\vec{a}_j$, $\left\{\lambda_j\right\}_{j\ge 1} \in l^1$,
with norm
\begin{align*}
\left\|\vec{u}\right\|_{U_{\Delta}^p(H; \mathbb{R})}:=\inf\left\{\sum^{\infty}_{j=1}|\lambda_j|: \vec{u} =\sum_{j=1}^{\infty}\lambda_j \vec{a}_j,\lambda_j\in \mathbb{C}, \vec{a}_j \text{ is } U_{\Delta}^p\textmd{-atom}\right\}.
 \end{align*}
\end{definition}
\begin{definition}[$V_{\Delta}^p(H; \mathbb{R})$ spaces]
Let $1\leq p < \infty$, and $H$ be a complex Hilbert space. We define $V_{\Delta}^p(H; \mathbb{R})$ as the space of all functions $\vec{v}:\mathbb{R} \rightarrow H$ such that
 \begin{align*}
 \|\vec{v}\|^p_{V_{\Delta}^p(H; \mathbb{R})}:=\left\|\vec{v}\right\|^p_{L_t^{\infty} H }+  \sup\limits_{\{t_k \}\nearrow }\sum_{k} \left\|e^{-it_k\Delta}v(t_k)-e^{-it_{k+1}\Delta }v(t_{k+1})\right\|^p_{H }.
\end{align*}
\end{definition}
We usually take $H=L^2h^\beta $ in our argument, where $\frac14 < \beta < 1$. We list some properties of these function spaces as follows: (since the proofs are similar to the mass-critical nonlinear Schr\"odinger equation case, so we omit the proofs.)
\begin{lemma}[\cite{HHK,KT}]
The function spaces $U^p_{\Delta} (L^2h^\beta) $ and $V^p_{\Delta} (L^2h^\beta)$ obey the embeddings
\begin{equation*}
U^p_{\Delta} (L^2h^\beta) \subseteq V^p_{\Delta}(L^2h^\beta)  \subseteq U^q_{\Delta}(L^2h^\beta) \subseteq L^{\infty}(L^2h^\beta), \textmd{ when } p<q.
\end{equation*}
\noindent We define $D U^p_{\Delta}(L^2h^\beta)$ be the space of functions
\begin{equation*}
D U^p_{\Delta}(L^2h^\beta) = \left\{(i\partial_t+\Delta)\vec{u}: \ \vec{u}\in U^p_{\Delta} (L^2h^\beta) \right\}.
\end{equation*}
\noindent According to Duhamel's formula,
\begin{equation*}
\left\|\vec{u}\right\|_{U^p_{\Delta} (L^2h^\beta) }   \lesssim \left\|\vec{u}(0)\right\|_{L^2h^\beta}+ \left\|(i\partial_t+\Delta_x)\vec{u}\right\|_{D U^p_{\Delta} (L^2h^\beta) }.
\end{equation*}
\noindent Moreover, we have the duality relation
\begin{equation*}
\left(DU^p_{\Delta} \left(L^2h^\beta \right) \right)^{\star}=V_{\Delta}^{p'} \left(L^2h^\beta\right).
\end{equation*}
\end{lemma}

\begin{lemma}\label{le4.16v46}
\noindent Suppose $J= \bigcup\limits_{j =1}^k J^j$, where $J^j=[a_j,b_j]$ are consecutive intervals, and $a_{j+1}=b_j$. Then for any $t_0\in J$,
\begin{equation*}
\left\|\int_{t_0}^t e^{i(t-\tau)\Delta}\vec{F}(\tau) \mathrm{d} \tau\right\|_{U^2_{\Delta}(L^2h^\beta; J)} \lesssim \sum_{j =1}^{k}\left\|\int_{J^j }e^{-i\tau\Delta} \vec{F}(\tau)  \mathrm{d} \tau \right\|_{L^2h^\beta}+ \left(\sum_{ j =1}^k \left\|\vec{F} \right\|^2_{DU^2_{\Delta}(L^2 h^\beta; J^j)} \right)^{\frac{1}{2}}.
\end{equation*}
 Moreover, we also have
\begin{equation*}
\left\|\vec{u}\right\|_{L^p_tL^q_xh^\beta(I)} \lesssim \left\|\vec{u}\right\|_{U^p_{\Delta}\left( L^2h^\beta; I\right)}.
\end{equation*}
\end{lemma}

\begin{theorem}[Bilinear Stricharz estimates] \label{le4.17v46}
\noindent If $\hat{u}_0$ is supported on $|\xi| \sim N$ and $\hat{v}_0$ is supported on $|\xi| \sim M$, $M \ll N$, we have
\begin{equation*}
\left \| \left\|e^{it\Delta}\vec{u}_0\right\|_{h^\beta} \left\|e^{it\Delta}\vec{v}_0\right\|_{h^\beta} \right \|_{L_{t,x}^2(\mathbb{R}\times \mathbb{R})}\lesssim N^{-\frac{1}{2}}\|\vec{u}_0\|_{L^2h^\beta}\|\vec{v}_0\|_{L^2h^\beta}.
\end{equation*}
\end{theorem}
For $k_0 \in \mathbb{Z}_+$, and Let $[0,T]$ be an interval such that $\left\|\vec{u}\right\|_{L^6_{x,t}h^{\beta}([0,T])}^6= 2^{k_0} $ and $\int_0^T N(t)^3 \mathrm{d}t=\epsilon_3 2^{k_0}$.
%M$.
We then partition $[0,T]=\bigcup\limits_{l=0}^{2^{k_0} -1} J_l$ with $\left\|\vec{u}\right\|_{L^6_{x,t}h^{\beta}(J_l)}=1$. As a convention, we call $J_l$ the small intervals.
\begin{definition}[$G_k^j$-intervals]
For an integer $0\leq j<k_0$, $0\leq k < 2^{k_0-j}$, let
\begin{align*}
G^j_k= \bigcup_{\alpha=k2^j}^{(k+1)2^j-1}J^{\alpha},
\end{align*}
where $J^{\alpha}$ satisfies $[0,T]=\bigcup\limits_{\alpha=0}^{2^{k_0} -1} J^{\alpha}$ with
\begin{equation*}
\int_{J^{\alpha}} \left(N(t)^3+\epsilon_3 \left\|\vec{u}(t) \right\|^6_{L^6_{x}h^{\beta}(\mathbb{R}\times \mathbb{Z})}\right) \mathrm{d}t=2\epsilon_3.
\end{equation*}
\noindent For $j\geq k_0$, let $G^j_k=[0,T]$. Now suppose $G^j_k=[t_0,t_1]$, we define $\xi(G_k^j) = \xi(t_0)$.
\end{definition}

\begin{lemma}
  If $J$ is a time interval with $\left\|\vec{u}\right\|_{L^6_{x,t}h^{\beta}(J)} \lesssim 1$, then
\begin{align*}
 \left\|\vec{u}\right\|_{U^2_{\Delta}(L^2h^\beta;J)} \lesssim 1 \textmd{ and } \Big\|P_{ \ge {2^{-4} \epsilon_3^{-\frac12} N(J)} } \vec{u} \Big\|_{U^2_{\Delta}(L^2h^\beta;J)} \lesssim \epsilon_2,
\end{align*}
where $N(J)=\sup\limits_{t\in J}N(t)$. Moreover, we have
\begin{align*}
\left\|\vec{u}\right\|_{L^p_tL^q_xh^{\beta}(J)} \lesssim 1 \textmd{ and }  \Big\|P_{ \ge {2^{-4} \epsilon_3^{-\frac12} N(J)} } \vec{u} \Big\|_{L^p_tL^q_xh^{\beta}(J)}\lesssim \epsilon_2, \text{ where $(p,q)$ is $L^2(\mathbb{R})-$adimissible.}
\end{align*}
\end{lemma}
We now introduce the following $\tilde{X}_{k_0} $ space, which is a refinement of the %$U_\Delta^2 h^\beta$
$U^2_{\Delta}(L^2h^\beta)$
 both in frequency and time interval, which also capture the some type of smoothness in high frequency due to the long time Strichartz estimate. The construction base on the refinement of the frequency first,
then we introduce a factor related to frequency $2^{i - j}$ to the low frequency $i \le j$, at last we refine the time interval to the small interval
for $i \le j$, which leads to the following $\tilde{X}_{k_0}$ space.
\begin{definition}[$\tilde{X}_{k_0}$ spaces]\label{de4.24v13}
\noindent For any $G^j_k \subset [0,T]$, let
\begin{align*}
\|\vec{u}\|_{X(G^j_k)}: &= \bigg(\sum_{0\leq i<j}2^{i-j} \sum_{G^i_{\alpha}\subset G^j_k}\|P_{\xi(G^i_{\alpha}),i-2 \leq \cdot \leq i+2}\vec{u}\|^2_{U^2_{\Delta}(L^2h^\beta; G^i_{\alpha} )}+\sum_{i\geq j}\|P_{\xi(G^i_{\alpha}),i-2\leq \cdot \leq i+2}\vec{u}\|^2_{U^2_{\Delta}(L^2h^\beta; G^j_k  )}\bigg)^\frac12,
\end{align*}
where $P_{\xi(t),i-2\leq \cdot \leq i+2}\vec{u}(t,x) =e^{ix\cdot \xi(t)}P_{i-2\leq \cdot \leq i+2}(e^{-ix\cdot \xi(t)}\vec{u}(t,x) )$.
\end{definition}
\noindent We then define
\begin{equation*}
\left\|\vec{u}\right\|_{\tilde{X}_{k_0}([0,T])}:=\sup\limits_{0\leq j\leq k_0}\sup\limits_{G^j_k \subset [0,T]} \left\|\vec{u}\right\|_{X(G^j_k)}.
\end{equation*}
\noindent For $0\leq k_{\star} \leq k_0$, we also let
\begin{equation*}
\left\|\vec{u}\right\|_{\tilde{X}_{k_{\star}}([0,T])}:=\sup\limits_{0\leq j\leq k_{\star}}\sup\limits_{G^j_k \subset [0,T]} \left\|\vec{u}\right\|_{X(G^j_k)}.
\end{equation*}
\begin{lemma}
For $i<j$, $(p,q)$ an admissible pair, we have
\begin{equation*}
\left\|P_{\xi(t),i}\vec{u}\right\|_{L^p_tL^q_xh^{\beta}(G^j_k\times \mathbb{R})} \lesssim 2^{\frac{j-i}{p}} \left\|\vec{u}\right\|_{\tilde{X}_j(G^j_k)},\
%\end{equation*}
%\begin{equation*}
 \left\|P_{\xi(t),\geq j}\vec{u}\right\|_{L^p_tL^q_xh^{\beta}(G^j_k\times \mathbb{R})} \lesssim \left\|\vec{u}\right\|_{X(G^j_k)}.
\end{equation*}
\end{lemma}
The following is the main theorem in this subsection, i.e. the long time Stricharz estimate:
\begin{theorem}[Long time Stricharz estimate]\label{thm4.1}
\noindent Suppose $\vec{u}$ is the almost periodic solution to \eqref{eq1.733} in theorem \ref{th4.9v51}. Then there exists a constant $C>0$ (depending only on the size of initial data), such that for $k_0 \in \mathbb{Z}_+, \epsilon_1,\epsilon_2,\epsilon_3$ satisfying \eqref{eq4.15v51}, \eqref{eq4.16v51}, \eqref{eq4.17v51}, $\left\|\vec{u}\right\|^6_{L^6_{x,t}h^{\beta}([0,T])}= 2^{k_0} $ and $\int_0^T N(t)^3 \mathrm{d}t=\epsilon_3 2^{k_0}$,
\begin{equation}\label{eq4.10v55}
\left\|\vec{u}\right\|_{\tilde{X}_{k_0}([0,T])} \leq C.
\end{equation}
\end{theorem}
\begin{remark}
Throughout this subsection, the implicit constant depends only on $u$ and not on $k_0$ or $\epsilon_1,\epsilon_2,\epsilon_3$.
\end{remark}
\noindent \emph{Proof.}
\noindent The main idea is: first, we reduce the estimate to the low frequency of the Duhamel term; and then we estimate the low frequency part of the Duhamel term based on a bootstrap argument follows as in B. Dodson's 1d mass critical NLS result (\cite{D2}). At last, we consider two situations according to the frequency decomposition and prove the bootstrap argument for those two cases. The estimate for the linear part and the high frequency part is
similar as in \cite{D2}.

According to definition \ref{de4.24v13}, it suffices to prove that there exists a constant $C>0$ such that for any $0\leq j \leq k_0$ and $G^j_k\subset [0,T]$,
\begin{equation} \label{eq4.22v43}
\sum_{0 \leq i <j}2^{i-j}\sum_{G^i_{\alpha}\subset G^j_k}\left\|P_{\xi(G^i_{\alpha}),i-2\leq \cdot \leq i+2}\vec{u}\right\|^2_{U^2_{\Delta}(L^2h^\beta;G^i_{\alpha} )}+\sum_{i\geq j} \left\|P_{\xi(G^j_{k}),i-2\leq \cdot \leq i+2}\vec{u}\right\|_{U^2_{\Delta}(L^2h^\beta; G^j_k )}^2 \leq C.
\end{equation}
By the Duhamel principle,
\begin{align*}
&\left\|P_{\xi(G^i_{\alpha}),i-2\leq \cdot \leq i+2}\vec{u} \right\|^2_{U^2_{\Delta}(L^2h^\beta; G^i_{\alpha} )} \\
&\lesssim \left\|P_{\xi(G^i_{\alpha}),i-2\leq \cdot \leq i+2}\vec{u}(t^i_{\alpha}) \right\|^2_{L^2h^\beta(\mathbb{R})} + \left\|\int_{t^i_{\alpha}}^t e^{i(t-\tau)\Delta} P_{\xi(G^i_{\alpha}),i-2\leq \cdot \leq i+2}\vec{F}(\vec{u}(\tau)) \mathrm{d}\tau \right\|^2_{U^2_{\Delta}(L^2h^\beta; G^i_{\alpha}  )}.
\end{align*}
\noindent For any $0\leq i<j$, $G^i_{\alpha}\subset G^j_k$, choose $t^i_{\alpha}$ such that
\begin{equation*}
\left\|P_{\xi(G^i_{\alpha}),i-2\leq \cdot \leq i+2}\vec{u}(t^i_{\alpha})\right\|^2_{L^2h^\beta(\mathbb{R})}=\inf\limits_{t\in G^i_{\alpha}} \left\|P_{\xi(G^i_{\alpha}),i-2\leq \cdot \leq i+2}\vec{u}(t)\right\|^2_{L^2h^\beta(\mathbb{R})}.
\end{equation*}
It is easy to estimate the linear part, so it suffices to control the following nonlinear part.
\begin{align*}
& \sum_{0 \leq i <j}2^{i-j}\sum_{G^i_{\alpha}\subset G^j_k} \bigg\|\int_{t^i_{\alpha}}^t e^{i(t-\tau)\Delta} P_{\xi(G^i_{\alpha}),i-2\leq \cdot \leq i+2}\vec{F}(\vec{u}(\tau)) \mathrm{d}\tau \bigg\|_{U^2_{\Delta}(L^2h^\beta; G^i_\alpha  )}^2 \\
& +\sum_{i\geq j} \bigg\|\int_{t^i_{\alpha}}^t e^{i(t-\tau)\Delta} P_{\xi(G^i_{\alpha}),i-2\leq \cdot \leq i+2}\vec{F}(\vec{u}(\tau)) \mathrm{d}\tau \bigg\|_{U^2_{\Delta}(L^2h^\beta; G^j_k )}^2. %\nonumber
\end{align*}
Furthermore, the high frequency part of the nonlinear term is easier to control. Therefore, it suffices to estimate the low frequency part and the proof is based on a bootstrap argument. The bootstrap argument can be proved by using the following theorem.
\begin{theorem}\label{th4.23v40}
\begin{equation}\label{eq4.23.1}
\aligned
& \sum_{ \substack{ i\geq j, \\ N(G^j_k)\leq 2^{i-5}\epsilon_3^{\frac{1}{2}}} }\bigg\|\int_{t_{\alpha}^i}^{t}e^{i(t-\tau)\Delta}P_{\xi(G^j_k),i-2\leq \cdot \leq i+2} \vec{F}(\vec{u}(\tau)) \mathrm{d}\tau\bigg\|^2_{U^2_{\Delta}(L^2h^\beta; G^j_k )} \\
& + \sum_{0\leq i \leq j}2^{i-j}   \sum_{ \substack{G^i_{\alpha} \subset G^j_k, \\ N(G^i_\alpha)\leq 2^{i-5}\epsilon_3^{\frac{1}{2}}} }\bigg\|\int_{t_{\alpha}^i}^{t}e^{i(t-\tau)\Delta}P_{\xi(G^j_k),i-2\leq \cdot \leq i+2} \vec{F}(\vec{u}(\tau)) \mathrm{d}\tau \bigg\|^2_{U^2_{\Delta}(L^2h^\beta; G^i_{\alpha} )}
\lesssim \epsilon_2^{\frac{1}{2}}\left(1+\|\vec{u}\|_{\tilde{X}_j([0,T])}\right)^4.
\endaligned
\end{equation}
\end{theorem}
\noindent \emph{Proof of Theorem \ref{th4.23v40}:} As in \cite{D1,D2}, an important idea to estimate the left hand is to split the terms by frequencies into two parts. And we just need to deal with the two parts respectively. First, we start with a bilinear estimate. By the Sobolev embedding, for $l_0>i-5$,
\begin{align*}
\left\|\left\|P_{\xi(G^i_{\alpha}),l_0}\vec{u} \right\|_{h^\beta}\cdot \right\|P_{\xi(G^i_{\alpha}),\leq i-10}\vec{u}\|_{h^\beta}^2\|^2_{L^2_{t,x}(G^i_{\alpha}\times \mathbb{R})} \lesssim 2^{i-l_0}\|P_{\xi(G^i_{\alpha}),l_0}\vec{u}\|^2_{U^2_{\Delta}(L^2h^\beta; G^i_{\alpha} )} \|\vec{u}\|^2_{X(G^i_{\alpha} )}.
\end{align*}
\noindent Noticing that
\begin{equation}\label{eq4.25v39}
P_{\xi(G^i_{\alpha}),i-2\leq \cdot \leq i+2} \left(\sum\limits_{(j_1,j_2,j_3,j_4,j_5) \in \mathcal{R}(j)}u_{j_1}\bar{u}_{j_2}u_{j_3}\bar{u}_{j_4}u_{j_5} \right)=(a)+(b),
\end{equation}
where $(a)$ (two high frequency part) is
\begin{equation*}
(a)=P_{\xi(G^i_{\alpha}),i-2\leq \cdot \leq i+2}  O\left(\sum\limits_{(j_1,j_2,j_3,j_4,j_5) \in \mathcal{R}(j)}P_{\xi(G^i_{\alpha}),\geq i-5} u_{j_1}P_{\xi(G^i_{\alpha}),\geq i-5} \bar{u}_{j_2}u_{j_3}\bar{u}_{j_4}u_{j_5} \right)
\end{equation*}
and $(b)$ (four low frequency part) is
\begin{equation}\label{eq4.27v39}
(b)=P_{\xi(G^i_{\alpha}),i-2\leq \cdot \leq i+2} O\left(\sum\limits_{(j_1,j_2,j_3,j_4,j_5) \in \mathcal{R}(j)}  P_{\xi(G^i_{\alpha}),i-5\leq \cdot \leq i+5} u_{j_1}P_{\xi(G^i_{\alpha}),\cdot \leq i-5} \bar{u}_{j_2}P_{\xi(G^i_{\alpha}),\cdot \leq i-5}u_{j_3}P_{\xi(G^i_{\alpha}),\cdot \leq i-5}\bar{u}_{j_4}P_{\xi(G^i_{\alpha}),\cdot \leq i-5}u_{j_5} \right).
\end{equation}
\noindent Now we split it into two parts and estimate separately,
\begin{align*}
&\quad \Big\|\int_{t_{\alpha}^i}^{t}e^{i(t-\tau)\Delta}P_{\xi(G^i_{\alpha}),i-2\leq \cdot \leq i+2} \vec{F}(\vec{u}(\tau)) \mathrm{ d} \tau\Big\|_{U^2_{\Delta}(L^2h^\beta; G^i_{\alpha} )}  \\
&\lesssim \Big\|\int_{t_{\alpha}^i}^{t}e^{i(t-\tau)\Delta}P_{\xi(G^i_{\alpha}),i-2\leq \cdot \leq i+2}  O\Big(\sum\limits_{\mathcal{R}(j)}P_{\xi(G^i_{\alpha}),\geq i-5} u_{j_1}P_{\xi(G^i_{\alpha}),\geq i-5} \bar{u}_{j_2}u_{j_3}\bar{u}_{j_4}u_{j_5}\Big)(\tau) \mathrm{d} \tau\Big\|_{U^2_{\Delta}(L^2h^\beta; G^i_{\alpha}  )}  \\
& \quad + \Big\|\int_{t_{\alpha}^i}^{t}e^{i(t-\tau)\Delta} P_{\xi(G^i_{\alpha}),i-2\leq \cdot \leq i+2} O\Big(\sum\limits_{\mathcal{R}(j)}P_{\xi(G^i_{\alpha}),i-5\leq \cdot \leq i+5} u_{j_1}P_{\xi(G^i_{\alpha}),\cdot \leq i-5} \bar{u}_{j_2} P_{\xi(G^i_{\alpha}),\cdot \leq i-5}u_{j_3}\\
& \qquad \cdot P_{\xi(G^i_{\alpha}),\cdot \leq i-5}\bar{u}_{j_4}P_{\xi(G^i_{\alpha}),\cdot \leq i-5}u_{j_5}\Big)(\tau)   \mathrm{d} \tau\Big\|_{U^2_{\Delta}(L^2h^\beta; G^i_{\alpha} )}.
\end{align*}
\noindent We take the first term (two high frequency part) by using the following theorem.
\begin{theorem}\label{th4.24v40}
\noindent For a fixed $G^j_k\subset [0,T]$,
\begin{align*}
& \sum_{0 \leq i< j}2^{i-j} \sum_{G^i_{\alpha}\subset G^j_k}\Big\|\int_{t_{\alpha}^i}^{t}e^{i(t-\tau)\Delta}P_{\xi(G^i_{\alpha}),i-2\leq \cdot \leq i+2}  \Big(\sum\limits_{\mathcal{R}(j)}P_{\xi(G^i_{\alpha}),\geq i-5} u_{j_1}P_{\xi(G^i_{\alpha}),\geq i-5} \bar{u}_{j_2}u_{j_3}\bar{u}_{j_4}u_{j_5}\Big)(\tau) \mathrm{d}\tau\Big\|^2_{U^2_{\Delta}(L^2h^\beta; G^i_{\alpha} )} \\
&+ \sum_{i\geq j}\Big\|\int_{t_{\alpha}^i}^{t}e^{i(t-\tau)\Delta}P_{\xi(G^i_{\alpha}),i-2\leq \cdot \leq i+2} \Big(\sum\limits_{ \mathcal{R}(j)}P_{\xi(G^i_{\alpha}),i-5\leq \cdot \leq i+5} u_{j_1}P_{\xi(G^i_{\alpha}),\cdot \leq i-5} \bar{u}_{j_2} P_{\xi(G^i_{\alpha}),\cdot \leq i-5}u_{j_3}\\
& \qquad \cdot P_{\xi(G^i_{\alpha}),\cdot \leq i-5}\bar{u}_{j_4}P_{\xi(G^i_{\alpha}),\cdot \leq i-5}u_{j_5}\Big)(\tau) \mathrm{d}\tau \Big\|^2_{U^2_{\Delta}(L^2h^\beta; G^i_{\alpha} )}\\
&\lesssim \epsilon_2\|\vec{u}\|_{\tilde{X}_j([0,T])}^6+\epsilon_2^2\|\vec{u}\|_{\tilde{X}_j([0,T])}^8.
\end{align*}
\end{theorem}
\noindent In order to prove Theorem \ref{th4.24v40}, we need to do some preparations as follows. Since $N(G^i_{\alpha})\leq 2^{i-5}\epsilon_3^{\frac{1}{2}}$,
\begin{equation} \label{eq4.28v39}
\aligned
&\bigg\|\sum\limits_{(j_1,j_2,j_3,j_4,j_5)\in \mathcal{R}(j)}P_{\xi(G^i_{\alpha}),\geq i-5} u_{j_1}P_{\xi(G^i_{\alpha}),\geq i-5} \bar{u}_{j_2}u_{j_3}\bar{u}_{j_4}u_{j_5} \bigg\|_{L_t^{\frac{4}{3}}L_x^1 h^\beta(G^i_{\alpha}\times \mathbb{R})} \\
&\lesssim \left\|P_{\xi(G^i_{\alpha}),\geq i-5}\vec{u} \right\|^{\frac{1}{2}}_{L^{\infty}_t L^2_x h^\beta} \left\|\left\|P_{\xi(G^i_{\alpha}),\geq i-5}\vec{u}\right\|_{h^\beta}\cdot \left\|P_{\xi(G^i_{\alpha}),\leq i-5} \vec{u}\right\|_{h^\beta}^2\right\|^{\frac{3}{2}}_{L^2_{t,x}}\\
&\lesssim \epsilon_2^{\frac{1}{2}} \left\|\vec{u}\right\|^{\frac{3}{2}}_{X(G^i_{\alpha} )}\left(\sum_{l_0\geq i-5}2^{\frac{i-l_0}{2}} \left\|P_{\xi(G^i_{\alpha}),l_0}\vec{u}\right\|_{U^2_{\Delta}(L^2h^\beta; G^i_{\alpha} )}\right)^{\frac{3}{2}}.
\endaligned
\end{equation}
\noindent Now take $\vec{v}$ satisfying it is supported on $\left|\xi-\xi(G^i_{\alpha})\right| \sim 2^i$ and $\left\|\vec{v}\right\|_{V^2_{\Delta}(L^2h^\beta; G^i_{\alpha}  )}=1$. Noticing
\begin{align}\label{eq4.29v39}
\left\|P_{\xi(t),\geq i-10}\vec{u} \right\|_{L_t^{\frac{9}{2}}L_x^{18}h^\beta(G^i_{\alpha}\times \mathbb{R})} \lesssim \left\|\vec{u}\right\|_{X(G^i_{\alpha} )},
\end{align}
\begin{equation} \label{eq4.30v39}
\aligned
& \int_{G^i_{\alpha}}\Big\langle \vec{v}, \sum\limits_{(j_1,j_2,j_3,j_4,j_5) \in \mathcal{R}(j)}P_{\xi(G^i_{\alpha}),\geq i-5} u_{j_1}P_{\xi(G^i_{\alpha}),\geq i-5} \bar{u}_{j_2}u_{j_3}\bar{u}_{j_4}u_{j_5} \Big\rangle \mathrm{d}t  \\
&\lesssim \left\|P_{\xi(t) ,\geq i-5}\vec{u}\right\|_{L_t^{\infty}L_x^2h^\beta} \left\|P_{\xi(t),\geq i-10} \vec{u}\right\|^3_{L^{\frac{9}{2}}_{t}L_x^{18}h^\beta(G^i_{\alpha}\times \mathbb{R})} \left\|\left\|\vec{v}\right\|_{h^\beta}\cdot \left\|P_{\xi(t),\leq i-10} \vec{u}\right\|_{h^\beta}  \right\|_{L^3_{t,x}( G^i_{\alpha} \times \mathbb{R}   )} \\
&\lesssim \epsilon_2 \left\|\vec{u}\right\|^3_{X(G^i_{\alpha} )}\sum_{l_0 \geq i-5}2^{\frac{i-l_0}{4}}\left\|P_{\xi(G^i_{\alpha}),l_0} \vec{u}\right\|_{U^2_{\Delta}(L^2h^\beta; G^i_{\alpha}  )}.
\endaligned
\end{equation}
\noindent \emph{Proof of Theorem \ref{th4.24v40}:} For any $0\leq l \leq j$, $G^j_k$ overlaps $2^{j-l}$ intervals $G^l_{\gamma}$ and for $0\leq i \leq l$, each $G^l_{\gamma}$ overlaps $2^{l-i}$ intervals $G^i_{\alpha}$. Moreover, each $G^i_{\alpha}$ is the subset of one $G^l_{\gamma}$. By \eqref{eq4.28v39}, \eqref{eq4.30v39} and interpolation, the first term of the left hand side in the Theorem \ref{th4.24v40} is controlled by
\begin{align*}
%&
\lesssim \left(\epsilon_2\left\|\vec{u}\right\|^6_{X(G^i_{\alpha}
 )}+\epsilon^2 \left\|\vec{u}\right\|^8_{X(G^i_{\alpha}
 )}\right)
\sum_{0\leq i \leq j}2^{i-j}\sum_{G^i_{\alpha}\subset G^j_k}\sum_{l\geq i-10}2^{\frac{i-l}{4}} \left\|P_{\xi(G),l} \vec{u}\right\|^2_{U^2_{\Delta}(L^2h^\beta; G^i_{\alpha} )}.
\end{align*}
\noindent And
\begin{align*}
& \sum_{0\leq i \leq j}2^{i-j}\sum_{G^i_{\alpha}\subset G^j_k}
\sum_{l\geq i-10}2^{\frac{i-l}{4}} \left\|P_{\xi(G),l} \vec{u}\right\|^2_{U^2_{\Delta}(L^2h^\beta; G^i_{\alpha} )} \\
& \lesssim \sum_{0\leq i \leq j}2^{i-j}\sum_{G^l_{\gamma}\subset G^j_k}\sum_{i-10 \leq l \leq i}2^{\frac{i-l}{4}} \left\|P_{\xi(G),l} \vec{u}\right\|^2_{U^2_{\Delta}(L^2h^\beta;  G^l_{\gamma} )}
 + \sum_{0\leq i \leq j}2^{l-j}\sum_{G^l_{\gamma}\subset G^j_k} \left\|P_{\xi(G^l_{\gamma}),l} \vec{u}\right\|^2_{U^2_{\Delta}(L^2h^\beta; G^l_{\gamma} )}\left(\sum_{0\leq i \leq l}2^{\frac{i-l}{4}} \right)\\
& \quad + \sum_{0\leq i\leq j}\sum_{l\geq j} \left\|P_{\xi(G^j_k),l} \vec{u}\right\|^2_{U^2_{\Delta}(L^2h^\beta; G^i_{\alpha} )} \\
& \lesssim \left\|\vec{u}\right\|^2_{\tilde{X}_j([0,T]  )}.
\end{align*}
\noindent Now we take the second term (four low frequency part) by using Theorem \ref{th4.25v40} as follows:
\begin{theorem}\label{th4.25v40}
\begin{align*}
& \sum_{0 \leq i< j}2^{i-j}\sum_{0 \le l_5 \leq l_4 \leq l_3 \leq l_2 \leq i-10} \sum_{G^i_{\alpha}\subset G^j_k}\Big\|\int_{t_{\alpha}^i}^{t}e^{i(t-\tau)\Delta}P_{\xi(G^i_{\alpha}),i-2\leq \cdot \leq i+2}  \Big(\sum\limits_{(j_1,j_2,j_3,j_4,j_5) \in \mathcal{R}(j)}P_{\xi(G^i_{\alpha}), i-5 \leq \cdot \leq i+5} u_{j_1}P_{\xi(\tau),l_2} \bar{u}_{j_2}P_{\xi(\tau),l_3}u_{j_3}\\
& \qquad \cdot P_{\xi(\tau),l_4}\bar{u}_{j_4}P_{\xi(\tau),l_5} u_{j_5} \Big)(\tau) \mathrm{d} \tau\Big\|^2_{U^2_{\Delta}(L^2h^\beta;G^i_{\alpha} )} \\
&+ \sum_{i\geq j}\sum_{0 \le l_5 \leq l_4 \leq l_3 \leq l_2  \leq i-10}\Big\|\int_{t_{\alpha}^i}^{t}e^{i(t-\tau)\Delta}P_{\xi(G^i_{\alpha}),i-2\leq \cdot \leq i+2}\Big(\sum\limits_{ (j_1,j_2,j_3,j_4,j_5) \in \mathcal{R}(j)}P_{\xi(G^i_{\alpha}), i-5 \leq \cdot \leq i+5} u_{j_1}P_{\xi(\tau),l_2} \bar{u}_{j_2} P_{\xi(\tau),l_3} u_{j_3} P_{\xi(\tau),l_4}\bar{u}_{j_4}\\
& \qquad  P_{\xi(\tau),l_5}u_{j_5} \Big)(\tau)  \mathrm{d}\tau \Big\|^2_{U^2_{\Delta}(L^2h^\beta;G^i_{\alpha} )}
\lesssim \epsilon_2^2\|u\|_{\tilde{X}_j([0,T])}^6.
\end{align*}
\end{theorem}
\noindent Similar as dealing with Theorem \ref{th4.24v40}, before proving Theorem \ref{th4.25v40}, we need to do some preparations as follows:

\noindent For a given $G^i_{\alpha}$, there are at most two small intervals $J_1$, $J_2$ that overlap $G^i_{\alpha}$ but are not contained in $G^i_{\alpha}$. Let $\title{G^i_{\alpha}}=G^i_{\alpha}\setminus (J_1 \cap J_2)$.
By the Sobolev embedding theorem,
\begin{equation*}
\left\|P_{\xi(t),l_4} \vec{u}\right\|_{L^{\infty}_{t,x}h^\beta(J_l\times \mathbb{R})} \lesssim 2^{\frac{l_4}{2}}\epsilon_2 + 2^{\frac{l_4}{4}}N(J_l)^{\frac{1}{4}} \epsilon_3^{-\frac{1}{8}}.
\end{equation*}
\noindent By using this, as in \cite{D2}, we can obtain
\begin{equation*}
\left\|\left\|P_{\xi(G^i_{\alpha}),i-5\leq \cdot \leq i+5}\vec{u} \right\|_{h^\beta}\cdot \left\|P_{\xi(G^i_{\alpha}),\leq i-10}\vec{u}\right\|_{h^\beta}^2 \right\|_{L^2_{t,x}(G^i_{\alpha}\times \mathbb{R})}\lesssim \epsilon_2 \left\|P_{\xi(G^i_{\alpha}),i-5\leq \cdot \leq i+5}\vec{u} \right\|_{U^2_{\Delta}(L^2h^\beta; G^i_{\alpha} )}\left( \left\|\vec{u}\right\|_{X(G^i_{\alpha} )}+1\right).
\end{equation*}
\begin{proof}
Using Lemma \ref{le4.16v46}, we partition $G^i_\alpha$ into sets $G^{l_3}_{\gamma}$. Let $\vec{v}_{\gamma}^{l_3}$ satisfies $\left\|\vec{v}_{\gamma}^{l_3}\right\|_{L^2h^\beta}=1$ and its Fourier transform supported in $|\xi-\xi(G^i_{\alpha})| \sim 2^i$. Then by using Lemma \ref{le4.17v46} and Definition \ref{de4.24v13},
\begin{equation*}
\sum_{G_{\gamma}^{l_3}\subset G^i_{\alpha}}\left\| \left\|e^{it\Delta}\vec{v}^{l_3}_{\gamma}\right\|_{h^\beta} \cdot \left\|P_{\xi(t),l_3}\vec{u}\right\|_{h^\beta}\right\|_{L_{t,x}^2( G_\gamma^{l_3}\times \mathbb{R})}
\lesssim 2^{\frac{i}{2}} \left\|P_{\xi(t),l_3}\vec{u}\right\|_{U^2_{\Delta}(L^2h^\beta;  G^i_{\alpha} )}.
\end{equation*}
 We use it to obtain
\begin{equation*}
\aligned
&\sum_{0\leq l_5 \le  l_4 \leq l_3 \leq l_2 \leq i-10}\sum_{G^{l_3}_{\gamma}} \bigg\|\sum_{ \mathcal{R}(j)} \int_{G^{l_3}_{\gamma}}e^{it\Delta}(P_{\xi(G^i_{\alpha}),i-5\leq \cdot \leq i+5} u_{j_1} P_{\xi(\tau),l_2} \bar{u}_{j_2} P_{\xi(\tau),l_3} u_{j_2}P_{\xi(\tau),l_4} \bar{u}_{j_4} P_{\xi(\tau),l_5}u_{j_5})(\tau)  \mathrm{d}\tau \bigg\|_{L^2h^\beta} \\
&\lesssim \epsilon_2^2 \left\|P_{\xi(G^i_{\xi}),i-5\leq \cdot \leq i+5} \vec{u}\right\|_{U^2_{\Delta}(L^2h^\beta; G^i_{\alpha} )} \cdot \left(1+ \left\|\vec{u}\right\|_{X(G^i_{\alpha} )}\right)^2.
\endaligned
\end{equation*}
This takes care of the first term in Lemma \ref{le4.16v46}. Now we turn to the second term, for each $G^{l_3}_{\gamma}$, we consider $\vec{v}^{l_3}_{\gamma}$ satisfying $\left\|\vec{v}^{l_3}_{\gamma}\right\|_{V^2_{\Delta}(L^2h^\beta, G^{l_3}_{\gamma} )}=1$ and $\hat{\vec{v}}^{l_3}_{\gamma}$ is supported on $|\xi-\xi(G^i_{\alpha})| \sim 2^i$. The rest of the proof follows as Theorem 4.4 in \cite{D2}. This completes the proof of Theorem \ref{th4.25v40}.
\end{proof}
We can now finish the proof of Theorem \ref{thm4.1} by collecting the above theorems.

According to Definition \ref{de4.24v13}, we have
\begin{equation}\label{4base}
\left\|\vec{u}\right\|_{\tilde{X}_{{0}}([0,T]   )} \leq C(\vec{u}),
\end{equation}
\begin{equation*}
\left\|\vec{u}\right\|_{\tilde{X}_{k_{\star}+ 1 }([0,T] )}  \leq 2 \left\|\vec{u}\right\|_{\tilde{X}_{k_{\star}}([0,T] )}.
\end{equation*}
Suppose
\begin{equation*}
\left\|\vec{u}\right\|_{\tilde{X}_{k_{\star} + 1}([0,T] )} \leq 2C_0,
\end{equation*}
then by \eqref{eq4.23.1},
\begin{equation*}
\left\|\vec{u}\right\|_{\tilde{X}_{k_{\star } + 1}([0,T] )} \leq C(\vec{u})\left(1+\epsilon_2^{\frac{1}{2}} C_0^4\right).
\end{equation*}
Taking $C_0=2C(u),\epsilon_2>0$ sufficiently small closes the bootstrap and implies
\begin{equation*}
\left\|\vec{u}\right\|_{\tilde{X}_{k_{\star} + 1}([0,T] )} \leq C_0.
\end{equation*}
Thus theorem \ref{thm4.1} follows from the base case (\ref{4base}) and the induction on $k_{\star}$.

\subsection{Frequency localized interaction Morawetz estimate}\label{subse4.3}
Before giving the frequency localized interaction Morawetz estimate, we show the following lemma.
\begin{lemma}\label{le5.30v9}
\noindent For $\vec{F}_j(\vec{u}) =\sum\limits_{(j_1,j_2,j_3,j_4,j_5)\in \mathcal{R}(j)}u_{j_1}\bar{u}_{j_2}u_{j_3}\bar{u}_{j_4}u_{j_5}$, we have:
\begin{equation*}
\sum_{j\in \mathbb{Z}} \left\{\vec{F}_j(\vec{u}) ,u_j\right\}_p=-\frac{1}{3}\sum_{j\in \mathbb{Z} }\nabla\left(\bar{u}_j \vec{F}_j\left(\vec{u}\right)\right),
\end{equation*}
\noindent where $\{f,g\}_p := \Re(f\nabla \bar{g}-g \nabla \bar{f})$ is the momentum bracket.
\end{lemma}
\begin{proof}
The idea to prove the above lemma is using the symmetric properties of the resonant index set, which can be justified straightforwardly, so we omit the proof.
\end{proof}
\begin{lemma}\label{le4.27v40}
For any $L^2h^1$ sequence $\{u_j\}_{j\in \mathbb{Z} }$, we have
\begin{equation*}
 \sum\limits_{\substack{j_1,j_2,j_3,j_4,j_5,j\in \mathbb{Z}, \\ j+j_2+j_4=j_1+j_3+j_5, \\ |j|^2+|j_2|^2+|j_4|^2=|j_1|^2+|j_3|^2+|j_5|^2} }  \bar{u}_j u_{j_1}\bar{u}_{j_2}u_{j_3}\bar{u}_{j_4}u_{j_5} \geq 0 .
 \end{equation*}
\end{lemma}
\noindent {\it Sketch of the proof.}
 The idea to prove the above lemma is minimization. First, we reduce the problem to the finite case and we can regard the expression on left hand side to be a multi-variable function; Second, we show the function obtains zero at its critical point by noting the homogeneous property; Moreover, we prove for each single variable, the corresponding function is convex, and so the minimum can only be obtained at the critical points instead of the boundary. Thus, the multi-variable function gains minimum ${0}$ at the original point, which implies it is non-negative.
\begin{remark}\label{rm4.29}
We consider the $\mathbb{R}\times \mathbb{Z}$ case for the above two lemmas. The analogues of Lemmas \ref{le5.30v9} and \ref{le4.27v40} for the $\mathbb{R}\times \mathbb{Z}^2$ case also hold with very similar proofs.
\end{remark}
\begin{remark}\label{rm4.30}
The reason that we need to check and use Lemmas \ref{le5.30v9} and \ref{le4.27v40} is to ensure the term related to the momentum bracket in the proof of the interaction Morawetz estimate \eqref{eq4.20v66} is positive definite. The analogues of the above two lemmas for the mass-critical nonlinear Schr\"odinger equation case are trivially true.
\end{remark}
\noindent The following theorem is the main theorem of this subsection and we rely on it heavily to exclude the quasi-soliton scenario (see Theorem \ref{4main}).
\begin{theorem}[Low-frequency localized interaction Morawetz estimate]\label{th4.28v43}
 Suppose $\vec{u}(t,x)$ is the almost periodic solution to \eqref{eq1.733} on $[0,T]$ with $\int_0^T N(t)^3 \mathrm{d}t=\epsilon_3 K$. Then
\begin{equation*}
\left\|\sum_{j\in \mathbb{Z}}\partial_x  \left( \left|P_{\le K}  u_j(t,x)\right|^2 \right) \right\|^2_{L^2_{t,x}([0,T]\times \mathbb{R})}\lesssim \sup\limits_{t \in [0,T]}M_{I}(t)+o(K),
\end{equation*}
\noindent where $o(K)$ is a quality such that $\frac{o(K)}{K} \rightarrow 0$ as $K \rightarrow \infty$, and
\begin{equation}\label{eq4.53v55}
M_I(t) :=\sum_{j,\, j' \in \mathbb{Z}} \iint_{\mathbb{R}^2} a(x-y)|P_{\le CK} u_j(t,y)|^2 \Im \left( \overline{ P_{\le  K} u_{j'}} \partial_{x} P_{\le  K} u_{j'}\right)(t,x) \mathrm{d}x\mathrm{d}y
\end{equation}
is the frequency localized interaction Morawetz action. For this problem, we take $a(x-y)=\frac{x-y}{|x-y|}$. More generally, we can obtain the Morawetz estimate for general $a(x-y)$, where $a(x-y)$ is odd in $x-y$ and satisfying $ \partial_x a(x-y)$ is bounded in $L^1_x$,
\end{theorem}
\begin{proof}
We can define the interaction Morawetz action to be
\begin{equation*}
M(t)=\sum_{j,\,j' \in \mathbb{Z}} \iint_{\mathbb{R}^2}
a(x-y)|u_j(t,y)|^2 \Im \left( \overline{u_{j'}(t,x)} \partial_{x}u_{j'}(t,x) \right) \mathrm{d}x\mathrm{d}y.
\end{equation*}
\noindent Integration by parts, together with Lemmas \ref{le5.30v9} and \ref{le4.27v40}, then as in \cite{PV}, we can prove
\begin{equation}\label{eq4.20v66}
\sum_{j \in \mathbb{Z}} \int_0^T \int_{\mathbb{R}} |\partial_{x} \left( |u_j(t,x)|^2 \right) |^2 \mathrm{d}x\mathrm{d}t \lesssim \int_0^T  M'(t) \,\mathrm{d}t  \lesssim \sup\limits_{t\in [0,T]}|M(t)|.
\end{equation}
\noindent Similar as in \cite{D1,D2,D3}, since $\dot{H}^1 h^\beta$ can not be controlled by $L^2 h^\beta$ norm, so we turn to the low-frequency localized interaction Morawetz action $M_I(t)$ defined in \eqref{eq4.53v55}.
Because of the low-frequency truncation, an error term will appear inevitably and our task is to control the error term, which can be controlled by
the long time Stricharz estimate (\ref{eq4.10v55}). We have
\begin{equation*}
\sum_{j \in \mathbb{Z}} \int_0^T \int_{\mathbb{R}} |\partial_{x} \left( | P_{\le  K} u_j(t,x)|^2 \right) |^2 \mathrm{d}x\mathrm{d}t \lesssim \int_0^T M_I'(t) \mathrm{d}t +\varepsilon \lesssim \sup\limits_{t\in [0,T]}|M_I(t)| +  \varepsilon,
\end{equation*}
\noindent where
\begin{align*}
\varepsilon = & 2\sum_{j,\,j' \in \mathbb{Z}}\int_0^T \iint_{\mathbb{R}^2}
a(x-y)\Im \left(\overline{P_{\le  K} u_j} P_{\le  K}
\vec{F}_j(\vec{u}) \right)(t,y) \Im \left( \overline{ P_{\le  K} u_{j'}}\partial_{x} P_{\le  K} u_{j'} \right) (t,x) \mathrm{d}x\mathrm{d}y\mathrm{d}t \\
 & +\sum_{j,\,j' \in \mathbb{Z}}\int_0^T \iint_{\mathbb{R}^2}
  a(x-y)| P_{\le  K} u_j(t,y)|^2 \Re \left(\left( \overline{P_{\le  K}
  \vec{F}_{j'}(\vec{u} )}-
  \overline{ \vec{F}_{j'} (  P_{\le  K} \vec{u} )} \right)\partial_{x} P_{\le  K} u_{j'} \right)(t,x) \mathrm{d}x\mathrm{d}y\mathrm{d}t  \\
 & +\sum_{j,\,j' \in \mathbb{Z}}\int_0^T \iint_{\mathbb{R}^2}
 a(x-y)|P_{\le K} u_j(t,y)|^2 \Re \left( \overline{P_{\le  K} u_{j'}}\partial_{x}\left( \vec{F}_{j'}(P_{\le  K} \vec{u})- P_{\le K} \vec{F}_{j'}(\vec{u})\right) \right)(t,x) \mathrm{d}x\mathrm{d}y\mathrm{d}t.
\end{align*}
\noindent It suffices to prove that $\varepsilon \leq o(K)$. Since $a(\cdot )$ is odd, we have
\begin{equation}\label{4.3terms}
\aligned
\varepsilon = & 2\sum_{j,\,j'\in \mathbb{Z}}\int_0^T \iint_{\mathbb{R}^2}
a(x-y)\Im \left(\overline{ P_{\le  K} u_j} P_{ \le  K} \vec{F}_j(\vec{u}) \right)(t,y) \Im \left(\overline{P_{\le  K} u_{j'}}(\partial_{x}-i\xi(t))P_{\le  K} u_{j'} \right)(t,x) \mathrm{d}x\mathrm{d}y\mathrm{d}t   \\
& + \sum_{j,\,j' \in \mathbb{Z}}\int_0^T \iint_{\mathbb{R}^2}
a(x-y)|P_{\le  K} u_j(t,y)|^2 \Re \left( \left( \overline{ P_{\le  K} \vec{F}_{j'} (\vec{u} ) } - \overline{ \vec{F}_{j'}( P_{\le  K} \vec{u} )}\right)(\partial_{x}-i\xi(t)) P_{\le  K} u_{j'}\right)(t,x) \mathrm{d}x\mathrm{d}y\mathrm{d}t  \\
& + \sum_{j,\,j' \in \mathbb{Z}}\int_0^T \iint_{\mathbb{R}^2}
 a(x-y)|P_{\le  K} u_j(t,y)|^2 \Re \left( \overline{P_{\le  K} u_{j'}}(\partial_{x}-i\xi(t))\left( \vec{F}_{j'} (P_{\le K} \vec{u})-P_{\le K} \vec{F}_{j'} (\vec{u}  )\right) \right)(t,x) \mathrm{d}x\mathrm{d}y\mathrm{d}t\\
   := &  I + II + III.
\endaligned
\end{equation}
\noindent Now it suffices to control the three high frequency error terms in \eqref{4.3terms} respectively.
For $\lambda = \frac{2^{k_0}}{K}$, let $\vec{u}_\lambda(t,x) = \lambda^\frac12 \vec{u}(\lambda^2 t, \lambda x)$, we have
\begin{equation}\label{eq4.41v40}
\aligned
\left\|\left(\partial_x-i\xi(t)\right)P_{\le  K} \vec{u}\right\|_{L^4_tL^{\infty}_xh^\beta\left(\left[0,T\right]\times \mathbb{R}\right)}
&=2^{-k_0}K \left\|\left(\partial_x-i \lambda \xi(t)\right)P_{\le \lambda  K}\vec{u}_{\lambda}\right\|_{L^4_tL^{\infty}_xh^\beta\left(\left[0,\frac{T}{\lambda^2}\right]\times \mathbb{R}\right)} \\
&\lesssim 2^{-k_0}K\sum_{j=0}^{k_0+2}2^j \left\|P_{\lambda\xi(t),j}\vec{u}_{\lambda}\right\|_{L^4_tL^{\infty}_xh^\beta\left(\left[0,\frac{T}{\lambda^2}\right]\times \mathbb{R}\right)}   \lesssim K.
\endaligned
\end{equation}
As in \cite{D3}, we use a useful trick of frequency decomposition, i.e. considering $u= P_{\leq \frac{K}{32}}u+P_{\geq \frac{K}{32}}u$ and the following basic fact:
\begin{lemma}
For $\varphi$ in \eqref{eq1.6v77}, we have for any $\xi_1, \xi_2\in \mathbb{R}$, 
\begin{equation}
\left|\varphi\left(\frac{ \xi_2+\xi_1}{K} \right)- \varphi\left(\frac{ \xi_1}{K} \right)\right| \lesssim \frac{|\xi_2|}{K}. 
\end{equation}
\end{lemma}
Similar as in the proof of Theorem 6.2 in \cite{D3}, we can obtain an estimate as follows:
\begin{align}\label{eq4.42v40}
\Bigg\| &  P_{\le K} \left( \sum\limits_{(j_1,j_2,j_3,j_4,j_5)\in \mathcal{R}(j)} \mathcal{O}\left( P_{\leq \frac{K}{32}} u_{j_1} \overline{P_{\leq \frac{K}{32}} u_{j_2}}  P_{\leq \frac{K}{32}} u_{j_3}  \overline{P_{\leq \frac{K}{32}} u_{j_4}} P_{\ge \frac{K}{32}} u_{j_5}\right)
\right)   \\
 & -   \sum\limits_{(j_1,j_2,j_3,j_4,j_5)\in \mathcal{R}(j)} \mathcal{O}\left( P_{\leq \frac{K}{32}} u_{j_1} \overline{P_{\leq \frac{K}{32}} u_{j_2}}
  P_{\leq \frac{K}{32}} u_{j_3} \overline{ P_{\leq \frac{K}{32}} u_{j_4}}  P_{\le K}  P_{\ge \frac{K}{32}} u_{j_5}\right)
 \Bigg\|_{L_t^{\frac{4}{3}} L^{1}_xh^\beta([0,T] \times \mathbb{R})} \lesssim o(1). \notag
\end{align}
Also, we can prove
\begin{align}\label{eq4.43v40}
\Bigg\| &  P_{\le K} \left( \sum\limits_{(j_1,j_2,j_3,j_4,j_5)\in \mathcal{R}(j)} \mathcal{O}\left( P_{\leq \frac{K}{32}} u_{j_1} \overline{P_{\leq \frac{K}{32}} u_{j_2}}  P_{\leq \frac{K}{32}} u_{j_3}  {P_{\leq \frac{K}{32}} u_{j_4}} \overline{P_{\ge \frac{K}{32}} u_{j_5}}\right)
\right)   \\
 & -   \sum\limits_{(j_1,j_2,j_3,j_4,j_5)\in \mathcal{R}(j)} \mathcal{O}\left( P_{\leq \frac{K}{32}} u_{j_1} \overline{P_{\leq \frac{K}{32}} u_{j_2}}
  P_{\leq \frac{K}{32}} u_{j_3}  P_{\leq \frac{K}{32}} u_{j_4}  P_{\le K}  \overline{P_{\ge \frac{K}{32}} u_{j_5}}\right)
 \Bigg\|_{L_t^{\frac{4}{3}} L^{1}_x h^\beta([0,T] \times \mathbb{R})} \lesssim o(1). \notag
\end{align}
And
\begin{align}\label{eq4.44v40}
& \left\|  \sum\limits_{(j_1,j_2,j_3,j_4,j_5)\in \mathcal{R}(j)} \mathcal{O}\left( P_{\ge \frac{K}{32}} u_{j_1} \overline{P_{\ge \frac{K}{32}} u_{j_2}}  P_{\leq \frac{K}{32}} u_{j_3} \overline{ P_{\leq \frac{K}{32}} u_{j_4}} P_{\le \frac{K}{32}} u_{j_5}\right)
  \right\|_{L_t^{\frac{4}{3}} L^{1}_x h^\beta([0,T] \times \mathbb{R})} \\
& \  +  \left\|  \sum\limits_{(j_1,j_2,j_3,j_4,j_5)\in \mathcal{R}(j)} \mathcal{O}\left( P_{\ge \frac{K}{32}} u_{j_1} \overline{ P_{\ge \frac{K}{32}} u_{j_2} } P_{\ge \frac{K}{32}} u_{j_3} \overline{ P_{\ge \frac{K}{32}} u_{j_4}}  P_{\ge \frac{K}{32}} u_{j_5}\right)
  \right\|_{L_t^{\frac{4}{3}} L^{1}_x h^\beta([0,T] \times \mathbb{R})}
  \lesssim o(1). \notag
\end{align}
By using \eqref{eq4.41v40}, \eqref{eq4.42v40}, \eqref{eq4.43v40}, \eqref{eq4.44v40} and the conservation of the mass, we see
\begin{align*}
 II &  = \sum_{j,\,j' \in \mathbb{Z}}\int_0^T \iint_{\mathbb{R}^2}
a(x-y)|P_{\le  K} u_j(t,y)|^2 \Re \left( \left( \overline{ P_{\le  K} \vec{F}_{j'} (\vec{u} ) } - \overline{ \vec{F}_{j'}( P_{\le  K} \vec{u} )}\right)(\partial_{x}-i\xi(t)) P_{\le  K} u_{j'}\right)(t,x) \mathrm{d}x\mathrm{d}y\mathrm{d}t\\
& \lesssim  \left\|P_{\le K} u \right\|_{L_t^\infty L_y^2 l^2}^2 \left\| P_{\le  K} \vec{F}  (\vec{u} )  -  \vec{F} ( P_{\le  K} \vec{u} ) \right\|_{L_t^\frac43 L_x^1 l^2} \left\|\left(\partial_{x}-i\xi(t)\right) P_{\le  K} \vec{u} \right\|_{L_t^4 L_x^\infty l^2}
   \lesssim o(K).
\end{align*}
For the third term $III$, by using integration by parts, we have
\begin{align} \label{eq4.45v40}
III &  =- II  \\
& \quad  -\sum_{j,\,j' \in \mathbb{Z}} \int_0^T\iint_{\mathbb{R}^2}
\partial_x a(x-y)|P_{\le K} u_{j}(t,y)|^2 \Re\left( \overline{P_{\le  K} u_{j'}(t,x)}
\left(\vec{F}_{j'} \left( P_{\le  K} \vec{u} \right)- P_{\le  K} \vec{F}_{j'} \left(\vec{u} \right)\right)(t,x) \right)  \mathrm{d}x\mathrm{d}y\mathrm{d}t\notag\\
 & : = - II + IV. \nonumber
\end{align}
It suffices to estimate the difference term $IV$, which can be handled as follows. Noticing $\partial_x a(x-y)$ is an $L^1$ function so by using Young's inequality and then using \eqref{eq4.42v40}, \eqref{eq4.43v40} and conservation law,
\begin{equation*}
IV
\lesssim \left\|P_{\le  K} \vec{u}\right\|^3_{L_t^{12}L_x^{3}h^\beta([0,T])} \left\|\vec{F}(P_{\le  K} \vec{u}  )- P_{\le  K} \vec{F}( \vec{u}) \right\|_{L_t^{\frac{4}{3}}L_x^1 h^\beta([0,T])} \lesssim o(K).
\end{equation*}
Finally, we turn to the term $I$.
Using the symmetric property of the resonance nonlinearity, we see
\begin{equation*}
\sum_{j \in \mathbb{Z}}  \Im \left(\overline{ P_{\le  K} u_j} \vec{F}_j \left( P_{\le  K} \vec{u} \right) \right)=0.
\end{equation*}
Thus, we can write
\begin{equation*}
\sum_{j \in \mathbb{Z}}  \Im\left(\overline{ P_{\le  K} u_j} P_{\le  K} \vec{F}_j(\vec{u}) \right)=\sum_{ j \in \mathbb{Z}}  \Im \left( \overline{P_{\le  K} u_j}
\left( P_{\le CK} \vec{F}_j\left(\vec{u}\right)- \vec{F}_j\left( P_{\le  K} \vec{ u} \right)\right) \right).
\end{equation*}
By using \eqref{eq4.41v40}, \eqref{eq4.42v40}, \eqref{eq4.43v40} and \eqref{eq4.44v40},
\begin{equation}\label{eq6.43v69}
\left\| \sum\limits_{j' \in \mathbb{Z}} \overline{P_{\le  K} P_{\ge \frac{K}{32}} {u}_{j'} }\left(P_{\le  K} \vec{F}_{j'}
 \left(\vec{u}\right)- \vec{F}_{j'} \left(P_{\le  K} \vec{u}\right) \right)\right\|_{L^1_{t,x} ([0,T] \times \mathbb{R} )} \lesssim o(1).
\end{equation}
Similarly, we can obtain
\begin{align}\label{eq6.44v69}
& \left\| \sum\limits_{j\in\mathbb{Z}}
\overline{ P_{\le \frac{K}{32}}  u_j}  \cdot  \sum\limits_{(j_1,j_2,j_3,j_4,j_5)\in \mathcal{R}(j)} \mathcal{O}\left( P_{\ge \frac{K}{32}} u_{j_1} \overline{P_{\ge \frac{K}{32}} u_{j_2}}  P_{\leq \frac{K}{32}} u_{j_3} \overline{ P_{\leq \frac{K}{32}} u_{j_4}} P_{\le \frac{K}{32}} u_{j_5}\right)
  \right\|_{L_{t,x}^{1} ([0,T] \times \mathbb{R})} \\
& \  +  \left\| \sum\limits_{j \in \mathbb{Z}} \overline{ P_{\le \frac{K}{32}}  u_j } \cdot    \sum\limits_{(j_1,j_2,j_3,j_4,j_5)\in \mathcal{R}(j)} \mathcal{O}\left( P_{\ge \frac{K}{32}} u_{j_1} \overline{ P_{\ge \frac{K}{32}} u_{j_2} } P_{\ge \frac{K}{32}} u_{j_3} \overline{ P_{\ge \frac{K}{32}} u_{j_4}}  P_{\ge \frac{K}{32}} u_{j_5}\right)
  \right\|_{L_{t,x}^{1}  ([0,T] \times \mathbb{R})}
  \lesssim 1. \notag
\end{align}
\noindent Moreover, the Fourier transform of
\begin{align*}
&  P_{\le K} \left( \sum\limits_{(j_1,j_2,j_3,j_4,j_5)\in \mathcal{R}(j)} \mathcal{O}\left( P_{\leq \frac{K}{32}} u_{j_1} \overline{P_{\leq \frac{K}{32}} u_{j_2}}  P_{\leq \frac{K}{32}} u_{j_3}  \overline{P_{\leq \frac{K}{32}} u_{j_4}} P_{\ge \frac{K}{32}} u_{j_5}\right)
\right)   \\
 & \quad -   \sum\limits_{(j_1,j_2,j_3,j_4,j_5)\in \mathcal{R}(j)} \mathcal{O}\left( P_{\leq \frac{K}{32}} u_{j_1} \overline{P_{\leq \frac{K}{32}} u_{j_2}}
  P_{\leq \frac{K}{32}} u_{j_3} \overline{ P_{\leq \frac{K}{32}} u_{j_4}}  P_{\le K}  P_{\ge \frac{K}{32}} u_{j_5}\right)\\
= &  P_{\le K} \left( \sum\limits_{(j_1,j_2,j_3,j_4,j_5)\in \mathcal{R}(j)} \mathcal{O}\left( P_{\leq \frac{K}{32}} u_{j_1} \overline{P_{\leq \frac{K}{32}} u_{j_2}}  P_{\leq \frac{K}{32}} u_{j_3}  \overline{P_{\leq \frac{K}{32}} u_{j_4}} P_{\ge \frac{K}{ 2}} u_{j_5}\right)
\right)   \\
 & \quad -   \sum\limits_{(j_1,j_2,j_3,j_4,j_5)\in \mathcal{R}(j)} \mathcal{O}\left( P_{\leq \frac{K}{32}} u_{j_1} \overline{P_{\leq \frac{K}{32}} u_{j_2}}
  P_{\leq \frac{K}{32}} u_{j_3} \overline{ P_{\leq \frac{K}{32}} u_{j_4}}    P_{\ge \frac{K}{ 2}} u_{j_5}\right)
\end{align*}
is supported on $|\xi|\geq \frac{K}{2}$. And therefore the Fourier transform of
\begin{align*}
& P_{\le \frac{K}{32}} u_j \Bigg( P_{\le K} \left( \sum\limits_{(j_1,j_2,j_3,j_4,j_5)\in \mathcal{R}(j)} \mathcal{O}\left( P_{\leq \frac{K}{32}} u_{j_1} \overline{P_{\leq \frac{K}{32}} u_{j_2}}  P_{\leq \frac{K}{32}} u_{j_3}  \overline{P_{\leq \frac{K}{32}} u_{j_4}} P_{\ge \frac{K}{32}} u_{j_5}\right)
\right)   \\
 & \quad -   \sum\limits_{(j_1,j_2,j_3,j_4,j_5)\in \mathcal{R}(j)} \mathcal{O}\left( P_{\leq \frac{K}{32}} u_{j_1} \overline{P_{\leq \frac{K}{32}} u_{j_2}}
  P_{\leq \frac{K}{32}} u_{j_3} \overline{ P_{\leq \frac{K}{32}} u_{j_4}}  P_{\le K}  P_{\ge \frac{K}{32}} u_{j_5}\right)\Bigg)
 \end{align*}
is supported on $|\xi|\geq \frac{K}{4}$.

The Sobolev embedding implies,
\begin{equation}\label{eq6.48v69}
\frac{1}{K}\left\|P_{\le K} \vec{u}\right\|^4_{L^8_tL^{\infty}_xh^\beta([0,T])} \lesssim 1.
\end{equation}
Since $N(t)\leq 1$, the almost periodicity \eqref{eq4.17v51} implies that
\begin{equation}\label{eq6.49v66}
\left\|\left(\partial_x-i\xi(t)\right)P_{\le  K} \vec{u}\right\|_{L^2_xh^\beta} \lesssim o(K).
\end{equation}
By using integration by parts, \eqref{eq6.48v69}, and \eqref{eq6.49v66}, we have
\begin{align}\label{eq6.47v69}
&\quad \sum_{j,\, j' \in \mathbb{Z} } \int_0^T \iint_{\mathbb{R}^2}  a(x-y)
P_{>\frac{K}{2}}\left( P_{\ge \frac{K}{32}}  u_j \cdot    \sum\limits_{(j_1,j_2,j_3,j_4,j_5)\in \mathcal{R}(j)} \mathcal{O}\left( P_{\le \frac{K}{32}} u_{j_1} \overline{ P_{\le \frac{K}{32}} u_{j_2} } P_{\le \frac{K}{32}} u_{j_3} \overline{ P_{\le \frac{K}{32}} u_{j_4}}  P_{\le \frac{K}{32}} u_{j_5}\right)
\right)(t,y)\\
& \qquad  \cdot \Im\left(\overline{P_{\le  K}  {u}_{j'} }\left(\partial_x-i\xi(t)\right)P_{\le K}  {u}_{j'} \right)(t,x)\mathrm{d}x\mathrm{d}y\mathrm{d}t \notag \\
&=\sum_{j, \, j' \in \mathbb{Z} } \int_0^T\iint_{\mathbb{R}^2}   \partial_x a(x-y)\frac{\partial_y}{\partial_y^2}
P_{>\frac{K}{2}}\left( P_{\ge \frac{K}{32}}  u_j \cdot    \sum\limits_{(j_1,j_2,j_3,j_4,j_5)\in \mathcal{R}(j)} \mathcal{O}\left( P_{\le \frac{K}{32}} u_{j_1} \overline{ P_{\le \frac{K}{32}} u_{j_2} } P_{\le \frac{K}{32}} u_{j_3} \overline{ P_{\le \frac{K}{32}} u_{j_4}}  P_{\le \frac{K}{32}} u_{j_5}\right)
\right)(t,y)
 \notag \\
& \qquad \cdot \Im \left(\overline{P_{\le  K} {u}_{j'} }
\left(\partial_x-i\xi(t)\right)P_{\le  K}  {u}_{j'}  \right)(t,x)\mathrm{d}x\mathrm{d}y\mathrm{d}t  \notag \\
&\lesssim \frac{1}{K} \left\|\left(\partial_x-i\xi(t)\right) P_{\le K} \vec{u} \right\|_{L_t^\infty L^2_xh^\beta([0,T]) }  \left\| P_{\le  K} \vec{u}\right\|^4_{L^8_tL^{\infty}_xh^\beta([0,T])} \lesssim o(K). \notag
\end{align}
By \eqref{eq6.43v69}, \eqref{eq6.44v69}, and \eqref{eq6.47v69}, we have
\begin{align*}
I = 2\sum_{j,\,j' \in \mathbb{Z}}\int_0^T \iint_{\mathbb{R}^2} a(x-y)\Im \left(\overline{P_{\le  K} u_j} P_{\le  K}
\vec{F}_j(\vec{u}) \right)(t,y) \Im \left( \overline{ P_{\le  K} u_{j'}}\partial_{x} P_{\le  K} u_{j'} \right) (t,x) \mathrm{d}x\mathrm{d}y\mathrm{d}t
\lesssim o(K).
\end{align*}
The proof of theorem \ref{th4.28v43} is now complete.
\end{proof}

\subsection{Exclusion of the Rapid frequency cascade scenario and Quasi-soliton scenario}\label{subse4.4}
\begin{theorem}\label{4main}
\noindent The almost periodic solution in Theorem \ref{thm4.1} does not exist.
\end{theorem}
\begin{proof}
We rule out the critical element for the following two scenarios respectively.

\noindent \textmd{\bf Case 1. Rapid frequency cascade ($\int_0^{\infty}N(t)^3 \,\mathrm{d}t =K <  \infty$)}.
Additional regularity is crucial for us to exclude this case. We can obtain additional regularity $\left\|\vec{u}(t,x)\right\|_{L^{\infty}_t\dot{H}^5_xh^{\beta}} \lesssim   \epsilon_3^{-5} K^5$
 by using the long time Stricharz estimate similar as in the \cite{D1,D2,D3}.

Let $\xi_\infty = \lim\limits_{t\to \infty} \xi(t)$, we have $|\xi_\infty| \le 2^{-20}  \epsilon_1^{-\frac12} K$ by \eqref{eq4.16v51}, so after making a Galilean transformation that shifts $ \xi_\infty$ to the origin, we still have
\begin{align}
\|\vec{u} \|_{L_t^\infty \dot{H}^5 h^\beta} \lesssim \epsilon_3^{-5} K^5,
\end{align}
By \eqref{eq4.16v51} and $\int_0^{\infty}N(t)^3 \,\mathrm{d}t =K <  \infty$, we have $\xi(t) \to 0$, and $N(t) \to 0$, as $t\to \infty$. Together with the definition of almost periodic solution yields
\begin{align*}
\left\|\vec{u}(t) \right\|_{\dot{H}_x^1 h^m} & \le  \left\|P_{\xi(t), \ge C(\eta) N(t)} \vec{u}(t) \right\|_{\dot{H}_x^1 h^\beta} + \left\|P_{\xi(t), \le C(\eta) N(t)} \vec{u}(t) \right\|_{\dot{H}_x^1 h^\beta}\\
& \lesssim \left\|P_{\xi(t), \ge C(\eta) N(t)} \vec{u}(t) \right\|_{L^2_x h^\beta}^\frac45 \left\|P_{\xi(t), \ge C(\eta) N(t)} \vec{u}(t) \right\|_{\dot{H}^5_x h^\beta}^\frac15 + (C(\eta) N(t) + |\xi(t)|) \left\| \vec{u}(t)\right\|_{L^2_x h^\beta}\\
& \lesssim \left\|P_{\xi(t), \ge C(\eta) N(t)} \vec{u}(t) \right\|_{L^2_x h^\beta}^\frac45 \left\|\vec{u}(t) \right\|_{\dot{H}^5_x h^\beta}^\frac15 + \left(C(\eta) N(t) + |\xi(t)|\right) \left\| \vec{u}(t) \right\|_{L^2_x h^\beta }\\
& \lesssim \eta^\frac45 \epsilon_3^{-1} K +  C(\eta) N(t) + |\xi(t)|,
\end{align*}
for any $\eta >0$.
Therefore, we have
\begin{align*}
\|\vec{u}(t) \|_{\dot{H}_x^1 h^\beta} \to 0,  \text{ as $t \to \infty$},
\end{align*}
By the Gagliardo-Nirenberg inequality and nonlinear estimates, we can show
\begin{align*}
\mathcal{E}(\vec{u}(t)) \lesssim \|\vec{u}(t)\|_{\dot{H}_x^1 h^\beta}^2 + \|\vec{u}(t) \|_{L_x^2 h^\beta}^4 \|\vec{u}(t)\|_{\dot{H}_x^1 h^\beta}^2 \to 0, \text{ as } t\to \infty,
\end{align*}
which implies $\mathcal{E}(\vec{u}(t)) = 0$, and thus $u \equiv 0$.
 This rules out the rapid frequency cascade case.

\noindent \textmd{\bf Case 2. Quasi-soliton ($\int_0^{\infty}N(t)^3 \,\mathrm{d}t = \infty$)}. In this case, by H{\"o}lder, Gagliardo-Nirenberg, interpolation, Sobolev, and also the definition of the H\"older continuity in \cite{D2}, we have
\begin{align*}
&\quad \sum_{j\in \mathbb{Z}}
\int_{|x-x(t)|\leq \frac{C\Big(\frac{ \|\vec{u}\|_{L^2_x l^2}^2}{100}\Big)}{N(t)}}|P_{\le  K} u_j(t,x)|^2 \mathrm{d}x \\
& \lesssim \sum_{j\in \mathbb{Z}}
\int_{|x-x(t)|\leq \frac{C\Big(\frac{ \|\vec{u}\|_{L^2_x l^2}^2}{100}\Big)}{N(t)}} |x - x(t)|^\frac12 \frac{ |P_{\le  K} u_j(t,x) |^2 - |P_{\le  K} u_j(t, x(t))|^2}{ | x - x(t)|^\frac12} \,\mathrm{d}x \\
&  \quad
 + \sum_{j\in \mathbb{Z}}
\int_{|x-x(t)|\leq \frac{C\Big(\frac{ \|\vec{u}\|_{L^2_x l^2}^2}{100}\Big)}{N(t)}} |P_{\le  K} u_j(t,x(t))|^2 \,\mathrm{d}x
 \\
& \lesssim \sum_{j\in \mathbb{Z}}
\int_{|x-x(t)|\leq \frac{C\Big(\frac{ \|\vec{u}\|_{L^2_x l^2}^2}{100}\Big)}{N(t)}} |x - x(t)|^\frac12 \frac{ |P_{\le  K} u_j(t,x) |^2 - |P_{\le  K} u_j(t, x(t))|^2}{ | x - x(t)|^\frac12} \,\mathrm{d}x \\
 & \quad  + \frac1{100} \left\| \left(\sum\limits_{j\in \mathbb{Z}} |P_{\le  K} u_j(t,x)|^2 \right)^\frac12 \right\|_{L_x^2}^2
 \\
&\lesssim \Bigg(\frac{C \Big(\frac{\|\vec{u}\|_{L^2_x l^2}^2}{100}\Big)}{N(t)}\Bigg)^{\frac{3}{2}}\bigg\|\sum\limits_{j\in \mathbb{Z}}|P_{\le  K} u_j(t,x)|^2\bigg\|_{\dot{C}_x^{\frac{1}{2}}(\mathbb{R})} +  \frac{\|\vec{ u}\|_{L^2 l^2}^2}{100}
\\
& \lesssim \Bigg(\frac{C\Big(\frac{ \|\vec{u}\|_{L^2 l^2}^2}{100}\Big)}{N(t)} \Bigg)^{\frac{3}{2}}\left\|\sum\limits_{j\in \mathbb{Z}}\partial_x \left( |P_{\le K} u_j(t,x)|^2 \right) \right\|_{L^2_x(\mathbb{R})} +  \frac{\|\vec{ u}\|_{L^2 l^2}^2}{100},
\end{align*}
where $\dot{C}_x^\frac12(\mathbb{R})$ is the homogeneous H\"older norm.

By Theorem \ref{th4.28v43}
and \eqref{eq6.49v66}, we have
\begin{equation*}
\Big\|\sum_{j\in \mathbb{Z}}\partial_x \left( |P_{\le K} u_j(t,x)|^2  \right) \Big\|^2_{L^2_{t,x}([0,T]\times \mathbb{R})} \lesssim o(K),
\end{equation*}
therefore, for $K \ge C\Big(\frac{ \left\|\vec{u} \right\|_{L^2_x l^2}^2}{100}\Big)$, we have
\[\frac{\left\|\vec{u}\right\|_{L^2 l^2}^2}{2} \le
\int_{|x-x(t)|\leq \frac{C\Big(\frac{ \|\vec{u} \|_{L^2 l^2}^2}{100}\Big)}{N(t)}} \left\|P_{\le  K} \vec{u}
 (t,x)\right\|_{l^2}^2 \mathrm{d}x. \]
Thus
\begin{align*}
\|\vec{u}\|^4_{L^2 l^2}K   \sim  \|\vec{u} \|^4_{L^2 l^2}  \int_0^T N(t)^3 \mathrm{d}t
&\lesssim \int_0^T N(t)^3 \bigg(\int_{|x-x(t)|\leq \frac{C\Big(\frac{\|\vec{u}\|_{L^2_x l^2 }^2}{100}\Big)}{N(t)}}\left\|P_{\le  K}
\vec{u}(t,x)\right\|_{l^2}^2\mathrm{d}x \bigg )^2 \mathrm{d}t \\
& \lesssim \bigg\|\sum_{j\in \mathbb{Z}}\partial_x \left( |P_{\le  K} u_j(t,x)|^2 \right) \bigg\|^2_{L^2_{t,x}([0,T]\times \mathbb{R})} \lesssim o(K).
\end{align*}
When $K$ is sufficiently large, we can get a contradiction to $\vec{u} \ne 0$. This completes the proof of theorem \ref{th1.2}.
\end{proof}

\subsection{The resolution of the conjecture of Z. Hani and B. Pausader}
In this subsection, we prove the global well-posedness and scattering for the two-discrete-component quintic nonlinear Schr\"odinger system arising in \cite{HP}. We consider the following Cauchy problem:
\begin{equation}\label{eq4.57v40}
\begin{cases}
i\partial_t u_j + \Delta_{\mathbb{R} } u_j = \sum\limits_{(j_1,j_2,j_3,j_4,j_5) \in \mathcal{R}(j) } u_{j_1} \bar{u}_{j_2} u_{j_3}\bar{u}_{j_4}u_{j_5},\\
u_j(0) = u_{0,j},\ j\in \mathbb{Z}^2,
\end{cases}
\end{equation}
 where $\mathcal{R}(j) =
 \left\{ j_1,j_2,j_3,j_4,j_5 \in \mathbb{Z}^2: j_1-j_2+j_3-j_4+j_5= j, \, |j_1|^2 - |j_2|^2 + |j_3|^2 - |j_4|^2 + |j_5|^2 = |j|^2 \right\}.$

\begin{remark}\label{2d.1}
The difference of this problem from \eqref{eq1.733} is that the dimension of the discrete component is $2$ instead of $1$. In \cite{HP}, Z. Hani and B. Pausader considered the defocusing quintic NLS on $\mathbb{R}\times \mathbb{T}^2$ and they prove scattering for the Cauchy problem with the $H^1$ initial data assuming scattering for the two-discrete-component quintic resonant nonlinear Schr\"odinger system.
\end{remark}

\begin{remark}\label{2d.2}
Generally speaking, with the increase of the discrete dimension of the resonant system, the resonant relation will become more complicated, and the nonlinear estimates (see Lemma \ref{le4.5v39} and Lemma \ref{le4.7v39}) will no longer hold and the scattering result are not expected when the discrete dimension is greater than 2. This observation also coincides with the waveguide case.
\end{remark}
\begin{remark}\label{2d.3}
Since both of the two quintic resonant systems (when discrete dimension is less than or equal to 2 )
correspond to the 1d mass critical NLS problem, the main idea of solving these two problems is using the techniques in \cite{D2}. The difference between these two systems is the discrete dimension. The two-discrete-dimensional quintic resonant system has more complicated resonant relation, which complicates the
 nonlinearity, thus the nonlinearity need to be handled delicately.
\end{remark}
The following theorem is the scattering theorem for the two-discrete-component quintic resonant nonlinear Schr\"odinger system:
\begin{theorem}[Scattering for the two-discrete-component quintic resonant nonlinear Schr\"odinger system]\label{th4.30v46}
Let $E>0$, for any initial data $\vec{u}_0$ satisfying
\begin{equation*}
\|\vec{u}_0\|_{L_x^2 h^1(\mathbb{R} \times \mathbb{Z}^2) } :=   \left\|\Big(\sum\limits_{j\in \mathbb{Z}^2} \langle j\rangle^2 |u_{0,j}( x)|^2\Big)^\frac12 \right\|_{L^2(\mathbb{R} )} \le E,
\end{equation*}
there exists a global solution $\vec{u} = \left\{u_j\right\}_{j\in \mathbb{Z}^2}$ to \eqref{eq4.57v40} satisfying
\begin{equation}\label{eq5.19new}
\|\vec{u}\|_{L_{t,x}^6 h^1(\mathbb{R}\times \mathbb{R}  \times \mathbb{Z}^2 )}  := \bigg\| \Big( \sum\limits_{j\in \mathbb{Z}^2 } \langle j\rangle^2 |u_j(t,x)|^2 \Big)^\frac12   \bigg\|_{L_{t,x}^6(\mathbb{R}\times \mathbb{R} )}  \le C,
\end{equation}
for some constant $C $ depends only on $\|\vec{u}_0\|_{L_x^2 h^1}$, and the solution scatters in $L_x^2 h^1$ in the sense that there exists $\left\{u_j^{\pm }\right\}_{j \in \mathbb{Z}^2} \in  L^2_x h^1$ such that
\begin{equation*}
  \bigg \| \Big( \sum\limits_{j\in \mathbb{Z}^2} \langle j\rangle^2  |  u_j(t) - e^{it\Delta_{\mathbb{R} }} u_j^{\pm  }|^2 \Big)^\frac12  \bigg\|_{L^2(\mathbb{R} )} \to 0, \text{ as } t\to \pm \infty.
\end{equation*}
\end{theorem}
Similar as the one-discrete-component quintic resonant nonlinear Schr\"odinger system, \eqref{eq4.57v40} also has the following conserved quantities:
\begin{align*}
\text{mass: }    &\quad   \mathcal{M}_{a,b,c}(\vec{u}(t))  = \int_{\mathbb{R}   } \sum_{j\in \mathbb{Z}^2} (a+ bj +  c |j|^2) |u_j(t,x )|^2\,\mathrm{d}x, \text{ where } a,b,c \in \mathbb{R},\\
\text{   energy:  }     &  \quad \mathcal{E}(\vec{u}(t))  =    \int_{\mathbb{R}   } \sum_{j\in \mathbb{Z}^2}\frac12 |\nabla u_j(t,x)|^2  + \frac16  \sum_{\substack{j\in \mathbb{Z}^2, \\n\in \mathbb{N}}} \Big|\sum_{ \substack{
j_1-j_2 + j_3  =j, \\
 |j_1|^2 - |j_2|^2 + |j_3|^2 = n} } (  u_{j_1} \bar{u}_{j_2} u_{j_3})(t,x) \Big|^2 \,\mathrm{d}x.
\end{align*}
Before giving the local well-posedness,
we give the following improvement of the elementary result in the Appendix (Lemma A.2 and Lemma A.3 in Page 1537-1539) of \cite{HP}.
\begin{lemma}\label{modif1}
For any $P \in \mathbb{R}^2$, $R > 0$, $\frac{7}{8}<\tilde{\beta}<1$ and $A > 1$,
\begin{align}\label{eqmodif1}
\sum_{\substack{|p|\ge A, \ p \in \mathbb{Z}^2 \cap C(P,R)}} \langle p \rangle^{-2\tilde{\beta} } \lesssim A^{1-2\tilde{\beta} },
\end{align}
where $C(P,R)$ denotes the circle of radius $R$ centered at $P$.
\end{lemma}
\emph{Proof of Lemma \ref{modif1}:} By using the following result in Lemma A.3 in \cite{HP}: %\ref{lea.3v42}
%as follows:
for any $k\ge 0$,
\begin{align*}
|\{ p : 2^k A \le |p|\le 2^{k+1} A, p \in C(P,R)\}| \lesssim 2^k A.
\end{align*}
We can obtain (\ref{eqmodif1}).
\begin{lemma}\label{modif2}
For $\frac{7}{8}<\tilde{\beta} <1$, we modified Lemma A.2 in \cite{HP} as follows:
\begin{equation}\label{eqmodif2}
\sup_{j\in \mathbb{Z}^2 } \bigg\{ \langle j \rangle^2 \sum_{ \substack{ (p_1,p_2,p_3,p_4,p_5) \in \mathcal{R}(j),\\ |p_5|\sim \max\left( |j|,|p_2|,|p_4|\right) } }
 \langle p_1 \rangle^{-2\tilde{\beta}} \langle p_2 \rangle^{-2\tilde{\beta}} \langle p_3 \rangle^{-2\tilde{\beta} } \langle p_4 \rangle^{-2\tilde{\beta} } \langle p_5 \rangle^{-2}  \bigg\} \lesssim 1.
\end{equation}
\end{lemma}
\begin{proof}
Without loss of generality, we assume that $|p_1|\leq |p_3|\leq |p_5|$, $|p_2| \leq |p_4|$ and $\max(|j|, |p_4|) \sim |p_5|$. Similar as showed in Lemma A.2 in \cite{HP}, we have:
\begin{align*}
\left| p_3 - \frac{p_2 + p_4 + j - p_1}2\right|^2 = \frac{2\left(|p_2|^2 + |p_4|^2 + |j|^2 - |p_1|^2\right) -  |p_2 + p_4 +j - p_1|^2}4.
\end{align*}
Let $\mathcal{C}$ be the circle of radius $\frac{\sqrt{2\left(|p_2|^2 + |p_4|^2 + |j|^2 - |p_1|^2\right) -  |p_2 + p_4 +j - p_1|^2}}2$ centered at $\frac{p_2 + p_4 + j - p_1}2 $.
We use Lemma \ref{modif1} to obtain:
\begin{align*}
\sum\limits_{\substack{ p_3 \in \mathcal{C},\\ |p_3| \ge \max\left(|p_1|, |p_2|, |p_4|\right)}} \langle p_3\rangle^{-2\tilde{\beta}} \lesssim \langle \max\left(|p_1|, |p_2|, |p_4| \right)\rangle^{1-2\tilde{\beta}}.
\end{align*}
Thus, we have:
\begin{align*}
&   \sum\limits_{\substack{(p_1,p_2,p_3,p_4,p_5) \in \mathcal{R}(j),\\ |p_2| \le |p_4| \le |p_3|,\\
|p_1| \le |p_3| \le |p_5|}} \langle p_1 \rangle^{-2\tilde{\beta}} \langle p_2 \rangle^{-2\tilde{\beta}} \langle p_3 \rangle^{-2\tilde{\beta}} \langle p_4 \rangle^{-2\tilde{\beta}} {\langle p_5 \rangle^{-2}}  {\langle j\rangle^2} \\
\lesssim & \sum_{p_1,p_2,p_4 \in \mathbb{Z}^2 } \langle p_1 \rangle^{-2 \tilde{\beta}} \langle p_2 \rangle^{-2\tilde{\beta} } \langle p_4 \rangle^{-2\tilde{\beta}} \sum_{\substack{|p_3 | \ge \max(|p_1|, |p_2|, |p_4|),\\ (p_1,p_2, p_3,p_4, p_2 + p_4 + j - p_1 - p_3) \in \mathcal{R}(j)}} \langle p_3 \rangle^{-2\tilde{\beta}} \\
\lesssim &   \sum_{p_1,p_2,p_4 \in \mathbb{Z}^2 } \langle p_1\rangle^{-2\tilde{\beta}} \langle p_2 \rangle^{-2\tilde{\beta} } \langle p_4 \rangle^{-2\tilde{\beta}} \langle |p_1| + |p_2| + |p_4| \rangle^{1-2\tilde{\beta}}
 \lesssim  1.
\end{align*}
The proof of Lemma \ref{modif2} is now complete.
\end{proof}
\begin{remark}
For this two-discrete-component case, the range for $\tilde{\beta}$ is $\frac{7}{8}<\tilde{\beta}<1$ , which is different from the $\frac38 < \beta< 1$ in the one-discrete-component case.
\end{remark}
Similarly to the one-discrete-component case, we list the following results regarding this problem without proof. The following proposition lists the local wellposedness and small-data scattering for the two-discrete-component quintic resonant nonlinear Schr\"odinger system \eqref{eq4.57v40}.
\begin{proposition}[Local wellposedness and small data scattering]\label{pr4.33v46}
Let $\vec{u}(0)= \left\{u_j(0)\right\}_{j \in\mathbb{Z}^2} \in  L_x^2 h^1$ satisfy $\left\|\vec{u}_0\right\|_{ L^2_x h^1} \le E$, then

(1) There exists an open interval $ 0 \in I$ and a unique solution $\vec{u}(t)$ of \eqref{eq4.57v40} in $C_t^0 L_x^2 h^1(I\times \mathbb{R} \times \mathbb{Z}^2 ) \cap S(I)$, where we define the norm $ {S}(I)$ by
\begin{align*}
\left\|\vec{u}\right\|_{S(I)} =  \bigg \|\Big(\sum_{j \in \mathbb{Z}^2 } \langle j \rangle^2 | u_j(t,x)|^2 \Big)^\frac12   \bigg\|_{L_{t,x}^6(I \times \mathbb{R} )},
\end{align*}
which is the scattering norm of the solution.

(2) There exists $E_0$ small enough such that if $\mathcal{E}(\vec{u})  \le E_0$, $\vec{u}(t)$ is global and scatters in positive and negative infinite time.
\end{proposition}
The following lemma is a nonlinear estimate for the quintic resonant nonlinearity which is a crucial step for Lemma \ref{le4.7v39v57}. Lemma \ref{le4.5v39v57} and Lemma \ref{le4.7v39v57} are the analogues of Lemma \ref{le4.5v39} and Lemma \ref{le4.7v39} for two-discrete-dimension case, and the proof of Lemma \ref{le4.5v39v57} utilizes Lemma \ref{modif2}.
\begin{lemma}\label{le4.5v39v57}
For sequence $\{u_j\}_{j\in \mathbb{Z}^2} \in h^1(\mathbb{Z}^2)$, then
\begin{align}\label{eq4.4v39}
\left\| \sum_{(j_1,j_2,j_3,j_4,j_5) \in \mathcal{R}(j)} u_{j_1} \bar{u}_{j_2} u_{j_3} \bar{u}_{j_4} u_{j_5} \right\|_{h^1} \lesssim \|\vec{u} \|_{h^1} \|\vec{u} \|_{h^{\tilde{\beta}}}^4,
\end{align}
where $\frac{7}{8}< \tilde{\beta} <1$.
\end{lemma}
\begin{proof}
By Lemma \ref{modif2}, 
\begin{equation}\label{eq4.5.1v57}
\sup_{j\in \mathbb{Z}^2}\bigg\{ \langle j \rangle^2 \sum_{ \substack{ (j_1,j_2,j_3,j_4,j_5) \in \mathcal{R}(j),\\ |j_5|\sim \textmd{max}\{ |j|,|j_2|,|j_4|\} } } \langle j_1 \rangle^{-2\tilde{\beta}} \langle j_2 \rangle^{-2\tilde{\beta}} \langle j_3 \rangle^{-2\tilde{\beta}} \langle j_4 \rangle^{-2\tilde{\beta} } \langle j_5 \rangle^{-2}  \bigg\} \lesssim 1.
\end{equation}
\noindent We can obtain similar estimates for the cases when $|j_1|\sim \textmd{max}\{j,|j_2|,|j_4|\}$ or $|j_3|\sim \textmd{max}\{j,|j_2|,|j_4|\}$ as well. Now we can use (\ref{eq4.5.1v57}) to prove Lemma \ref{le4.5v39v57} as follows:
\begin{equation*}
\aligned
\|\vec{F}(\vec{u})\|^2_{h^1}&=\sum_{j} \langle j \rangle^2 \Big|\sum_{(j_1,j_2, j_3,j_4,j_5) \in \mathcal{R}(j)}u_{j_1}\bar{u}_{j_2}u_{j_3}\bar{u}_{j_4}u_{j_5} \Big| \\
&\lesssim \sum_{j} \langle j \rangle^2\bigg(\Big|\sum_{\substack{ (j_1,j_2,j_3,j_4,j_5) \in \mathcal{R}(j),\\j_1 \sim \textmd{max}\{|j|,|j_2|,|j_4|\}} }u_{j_1}\bar{u}_{j_2}u_{j_3}\bar{u}_{j_4}u_{j_5} \Big|+ \Big|\sum_{\substack{(j_1,j_2,j_3,j_4,j_5) \in \mathcal{R}(j),\\j_3 \sim \textmd{max}\{ |j|,|j_2|,|j_4|\}}}u_{j_1}\bar{u}_{j_2}u_{j_3}\bar{u}_{j_4}u_{j_5}\Big|\\
& \qquad + \Big|\sum_{\substack{ (j_1,j_2,j_3,j_4,j_5) \in \mathcal{R}(j),\\j_5 \sim \textmd{max}\{ |j|,|j_2|,|j_4|\}} }u_{j_1}\bar{u}_{j_2}u_{j_3}\bar{u}_{j_4}u_{j_5}\Big|\bigg)\\
&\lesssim \left\|\vec{u}\right\|_{h^{\tilde{\beta}}}^8 \cdot \left\|\vec{u}\right\|_{h^1}^2.
\endaligned
\end{equation*}
\end{proof}
Similar as the mass-critical nonlinear Schr\"odinger equations, by using standard arguments, we can see the scattering norm for the quintic resonant system is $\|\vec{u}\|_{L_{t,x}^6 h^1(\mathbb{R}\times \mathbb{R} \times \mathbb{Z}^2)}$. Our current goal is to show the scattering norm of the solution is finite. The following lemma explains that we can reduce the scattering norm to a weaker one. Arguing as Lemma \ref{le4.7v39}, we can obtain the the following result as a consequence of Lemma \ref{le4.5v39v57}:
\begin{lemma} \label{le4.7v39v57}
If the solution $\vec{u}$ of the Cauchy problem (\ref{eq4.57v40}) with $L^2h^1$ initial data satisfies for some $\frac{7}{8}< \tilde{\beta} <1$,
\begin{equation*}
\|\vec{u}\|_{L_{t,x}^6 h^{\tilde{\beta}} (\mathbb{R}\times \mathbb{R} \times \mathbb{Z}^2 )}<\infty,
\end{equation*}
then we have
\begin{equation*}
\|\vec{u}\|_{L_{t,x}^6 h^1(\mathbb{R}\times \mathbb{R} \times \mathbb{Z}^2 )}<\infty
\end{equation*}
\end{lemma}
\subsubsection{Existence of the almost-periodic solution}
For the solution $\vec{u}$ of the two-discrete-component quintic resonant nonlinear Schr\"odinger system \eqref{eq4.57v40} with $\vec{u}_0 \in L_x^2 h^1(\mathbb{R}\times \mathbb{Z}^2)$, we define
\begin{align*}
A(m) = \sup\left\{\left\|\vec{u}\right\|_{L_{t,x}^6 h^{\tilde{\beta}}(\mathbb{R}\times \mathbb{R}  \times \mathbb{Z}^2)}  \le m \right\},
\end{align*}
and
\begin{align*}
m_0 = \sup\{ m : A(m')< \infty, \forall\, m' < m \},
\end{align*}
and our aim is to show $m_0 = \infty$, which reveals \eqref{eq4.57v40} is globally well-posed and scatters in $L_x^2 h^1(\mathbb{R}\times \mathbb{Z}^2)$.

Similar to the proof of Theorem \ref{pr3.2v15}, we have the following linear profile decomposition: 
\begin{proposition}[Linear profile decomposition in $L_x^2 h^1(\mathbb{R} \times \mathbb{Z}^2)$]\label{profile2}
Let $\left\{\vec{u}_{n}\right\}_{n\ge 1}$ be a bounded sequence in $L_x^2 h^1(\mathbb{R}  \times \mathbb{Z}^2)$. Then (after passing to a subsequence if necessary) there
exists $K^*\in  \{0,1, \cdots \} \cup \{\infty\}$, functions $\vec{\phi}^{k} \subseteq L_x^2 h^1$, $(\lambda_n^k, t_n^k, x_n^k , \xi_n^k)_{n\ge 1} \subseteq (0, \infty) \times \mathbb{R} \times \mathbb{R} \times \mathbb{R}$, for $ 1 \le k \le K^*$,
 so that defining $\vec{w}_{n}^K$ by
\begin{align*}
\vec{u}_n(x)  =   \sum_{k=1}^K \frac1{(\lambda_n^k)^\frac12 } e^{ix\xi_n^k} (e^{it_n^k \Delta_{\mathbb{R} } } \vec{\phi}^k)\left(\frac{x-x_n^k}{\lambda_n^k} \right) + \vec{w}_n^K(x ),
\end{align*}
we have the following properties:
\begin{align*}
& \limsup_{n\to \infty} \left\|e^{it\Delta_{\mathbb{R} }} \vec{w}_n^K \right\|_{L_{t,x}^6 h^{1-\epsilon_0}(\mathbb{R} \times \mathbb{R}  \times \mathbb{Z}^2)} \to 0, \ \text{ as } K\to \infty,\\
& \sup_{K} \lim_{n\to \infty} \left( \left\|\vec{u}_n \right\|_{L_x^2 h^1}^2 - \sum_{k=1}^K \left\|\vec{\phi}^k \right\|_{L_x^2 h^1}^2 - \left\|\vec{w}_n^K \right\|_{L_x^2 h^1}^2\right) = 0, \\
& (\lambda_n^k)^\frac12 e^{-it_n^k \Delta_{\mathbb{R} }}\left(  e^{-i(\lambda_n^k x + x_n^k) \xi_n^k}
 \vec{w}_n^K(\lambda_n^k x + x_n^k) \right)  \rightharpoonup 0  \text{ in } L_x^2 h^1,  \text{ as  } n\to \infty,\text{ for each }  k\le K,
\end{align*}
and lastly, for $k\ne k'$, and $n\to \infty$,
\begin{align*}
\frac{\lambda_n^k}{\lambda_n^{k'}} + \frac{\lambda_n^{k'}}{\lambda_n^k} + \lambda_n^k \lambda_n^{k'} |\xi_n^k - \xi_n^{k'}|^2
+ \frac{|x_n^k-x_n^{k'}- 2t_n^{k} (\lambda_n^k)^2 (\xi_n^k - \xi_n^{k'}) |^2 } {\lambda_n^k \lambda_n^{k'}}
+ \frac{|(\lambda_n^k)^2 t_n^k -(\lambda_n^{k'})^2 t_n^{k'}|}{\lambda_n^k \lambda_n^{k'}} \to \infty.
\end{align*}
\end{proposition}
Then by using the argument in \cite{CMZ,TVZ2} and the above linear profile decomposition, we can obtain
\begin{theorem}[Existence of an almost periodicity solution]
Assume $m_0 < \infty$, there exists a solution $\vec{u}\in C_t^0 L_x^2 h^1 \cap L_{t,x}^6 h^{\tilde{\beta}}(I \times \mathbb{R} \times \mathbb{Z}^2)$ to the two-discrete-component quintic resonant nonlinear Schr\"odinger system \eqref{eq4.57v40} with $I$ the maximal lifespan interval such that

(1) $M(\vec{u}) = m_0$, and $\vec{u}$ blows up at both directions in time.

(2) $\vec{u}$ is almost periodic in the sense that there exist $(x(t), \xi(t), N(t)) \in \mathbb{R}\times \mathbb{R} \times \mathbb{R}^+$ such
that for any $\eta > 0$, there exists $C(\eta) > 0$ such that for $t\in I$,
\begin{align*}
 \int_{|x-x(t)|\ge \frac{C(\eta)}{N(t)}} \left\|\vec{u}(t,x)\right\|_{h^1}^2 \,\mathrm{d}x + \int_{| \xi- \xi(t)|\ge C(\eta) N(t)} \left\|\hat{u}(t,\xi)\right\|_{h^1}^2 \,\mathrm{d}\xi < \eta.
\end{align*}
\end{theorem}
From now on, since the dimension of the discrete direction does not matter, the method to exclude the almost periodic solution is the same as the one-discrete-component case. Thus we omit the rest part of the proof, and refer to the proofs for the one-discrete-component case in Subsections \ref{subse4.2}, \ref{subse4.3} and \ref{subse4.4}.

\section{Existence and exclusion of the critical element}\label{se5}
\subsection{Existence of a critical element}\label{sse5.1v15}
By Theorem \ref{th2.946}, to prove the scattering of the solution of \eqref{eq1.1},
we only need to show the finiteness of the space-time norm $\|\cdot \|_{L_{t,x}^6 H_y^{1-\epsilon_0}}$ of the solution $u$ of \eqref{eq1.1}.
Define
\begin{equation*}
\Lambda(L) = \sup  \, \|  u\|_{L_t^6 L_x^{6} H_y^{1-\epsilon_0}  (\mathbb{R} \times \mathbb{R} \times \mathbb{T})},
\end{equation*}
where the supremum is taken over all global solutions $u \in C_t^0 H_{x,y}^{ 1}$ of \eqref{eq1.1} obeying $S(u(t)) \le L$, where $S(u(t)) := E(u(t)) + M(u(t))$.

By the local wellposedness theory, $\Lambda(L)< \infty$ for $L$ sufficiently small.
In addition, define $L_{max} = \sup \left\{ L: \Lambda(L) < \infty \right\}$.
Our goal is to prove $L_{max}  = \infty$. Suppose to the contradiction $L_{max} < \infty$, we will show a Palais-Smale type theorem.
\begin{proposition}[Palais-Smale condition modulo symmetries in $H_{x,y}^{ 1}(\mathbb{R}  \times \mathbb{T})$] \label{pr7.1}
Assume that $L_{max} < \infty$. Let $\{t_n\}_{n\ge1}$ be arbitrary sequence of real numbers and $\{u_n\}_{n\ge 1}$ be a sequence of solutions in $C_t^0 H_{x,y}^{1}(\mathbb{R} \times \mathbb{R}  \times \mathbb{T})$ to \eqref{eq1.1} satisfying
\begin{align}
 & S(u_n)  \to L_{max},\\
 & \|u_n\|_{L_{t,x}^6 H_y^{1-\epsilon_0} ((-\infty,t_n) \times \mathbb{R} \times \mathbb{T})} \to \infty,
\|u_n\|_{ L_{t,x}^6 H_y^{1-\epsilon_0}((t_n,\infty) \times \mathbb{R} \times \mathbb{T})} \to \infty, \text{ as } n \to \infty. \label{eq3.13n0}
\end{align}
Then, after passing to a subsequence, there exists a sequence
$x_n \in \mathbb{R} $ and $w\in H^{ 1}(\mathbb{R} \times \mathbb{T})$ such that
\begin{equation*}
  u_n  (x +  x_n, y, t_n)    \to w(x,y) \text{ in }   H_{x,y}^{ 1}(\mathbb{R}  \times \mathbb{T}), \text{ as } n\to \infty.
\end{equation*}
\end{proposition}
\begin{proof}
By replacing $u_n(t)$ with $u_n(t+t_n)$, we may assume $t_n = 0$. Applying Proposition \ref{pro3.9v23} to $\{u_n(0)\}_{n \ge 1}$, after passing to a subsequence, we have
\begin{align*}
u_n(0,x,y) &  = \sum\limits_{k=1}^K  \frac1{(\lambda_n^k)^\frac12 }   e^{ix\xi_n^k}   \left(e^{it_n^k \Delta_{\mathbb{R}  }}   P_n^k \phi^k\right)
\left(\frac{x-x_n^k}{\lambda_n^k},y\right)   + w_n^K(x,y).
\end{align*}
The remainder has asymptotically trivial linear evolution
\begin{equation}\label{eq6.1}
\limsup\limits_{n\to \infty} \left\|e^{it\Delta_{\mathbb{R}  \times \mathbb{T}  }}w_n^K \right\|_{ L_t^6 L_x^6 H_y^{1-\epsilon_0}} \to 0, \text{ as } K \to \infty,
\end{equation}
and we also have asymptotic decoupling of the mass and energy:
\begin{align}
& \lim\limits_{n\to \infty} \left(  M(u_n(0)) - \sum\limits_{k=1}^K M\left(  \frac1{(\lambda_n^k)^\frac{1}2} e^{ix\xi_n^k} \left( e^{it_n^k \Delta_{\mathbb{R}}} P_n^k \phi^k\right) \left( \frac{x-x_n^k}{\lambda^k_n}, y\right) \right) - M(w_n^K)  \right) = 0, \label{eq6.2} \\
& \lim\limits_{n\to \infty} \left(  E(u_n(0)) - \sum\limits_{k=1}^K E\left(  \frac1{(\lambda_n^k)^\frac{1}2} e^{ix\xi_n^k} \left( e^{it_n^k \Delta_{\mathbb{R}}} P_n^k \phi^k\right) \left( \frac{x-x_n^k}{\lambda^k_n}, y\right) \right) - E(w_n^K)  \right) = 0,\  \forall\, K, \label{eq5.5v30}
\end{align}
There are two possibilities:

{\bf Case 1.}  $\sup\limits_{k} \limsup\limits_{n\to \infty}
%E
S\left(  \frac1{(\lambda_n^k)^\frac{1}2} e^{ix\xi_n^k} \left( e^{it_n^k \Delta_{\mathbb{R}}} P_n^k \phi^k\right) \left( \frac{x-x_n^k}{\lambda^k_n}, y\right)\right) = L_{max}$.
Combining \eqref{eq6.2}, \eqref{eq5.5v30} with the fact that $\phi^k$ are nontrivial in $L_x^2 H_y^{1}$, we deduce that
\begin{align*}
u_n(0,x,y) =  \frac1{\lambda_n^\frac12} e^{ix\xi_n}    \left(e^{it_n\Delta_{  \mathbb{R}  }}  P_n \phi\right)\left(\frac{x-x_n}{\lambda_n},y\right) + w_n(x,y), \end{align*}
 with $\lim\limits_{n\to \infty} \|w_n\|_{H_{x,y}^{ 1}} = 0$. We will show that $\lambda_n \equiv 1$, otherwise $\lambda_n \to  \infty$.

By Theorem \ref{pr5.9}, there exists a unique global solution $u_n$ for $n$ large enough with
\begin{align*}
u_n(0,x,y) = \frac1{\lambda_n^\frac12}  e^{ix\xi_n}   (e^{it_n \Delta_{\mathbb{R}  }}   P_n \phi)\left(\frac{x-x_n}{\lambda_n}, y\right)
\end{align*}
and
\begin{align*}
\limsup\limits_{n\to \infty} \|u_n\|_{ L_t^6 L_x^6 H_y^{1-\epsilon_0} (\mathbb{R} \times \mathbb{R}  \times \mathbb{T})} \le C(L_{max}),
\end{align*}
which is a contradiction with \eqref{eq3.13n0}.

Therefore, $\lambda_n \equiv 1$, and $u_n(0,x,y) =   e^{ix\xi_n}     \left(e^{it_n \Delta_{\mathbb{R}  } }  P_n \phi\right)(x-x_n,y)   + w_n(x,y)$.
If $t_n \equiv 0$, by the fact $\xi_n $ is bounded, this is precisely the conclusion.
If $t_n \to -\infty$, by the Galilean transform
\begin{align*}
e^{it_0 \Delta_{\mathbb{R} }} e^{ix\xi_0} \tilde{\phi} (x) = e^{-it_0 |\xi_0|^2} e^{ix\xi_0} (e^{it_0 \Delta_{\mathbb{R} }} \tilde{ \phi} )(x-2t_0 \xi_0),
\end{align*}
we observe
\begin{align*}
&\left\|e^{it\Delta_{\mathbb{R} \times \mathbb{T}}}   \left( e^{ix\xi_n}   ( e^{it_n\Delta_{\mathbb{R}   }} P_n \phi)(x-x_n,y) \right)  \right\|_{ L_{t,x}^6 H_y^{1-\epsilon_0} ((-\infty,0) \times \mathbb{R} \times \mathbb{T})}\\
= & \left\|e^{-it|\xi_n|^2} e^{ix\xi_n} (e^{it\Delta_{\mathbb{R} }} (e^{it_n\Delta_{\mathbb{R} }} P_n \phi)(\cdot -x_n,y))(x-2t\xi_n) \right\|_{ L_{t,x}^6 H_y^{1-\epsilon_0} ((-\infty,0) \times \mathbb{R} \times \mathbb{T})}\\
=   & \ \|e^{i(t+t_n) \Delta_{\mathbb{R}  }} P_n \phi\|_{L_{t,x}^6 H_y^{1-\epsilon_0}((-\infty,0)\times \mathbb{R}  \times \mathbb{T})}
=   \ \|e^{it\Delta_{\mathbb{R}   }} P_n \phi\|_{ L_{t,x}^6 H_y^{1-\epsilon_0}((-\infty,t_n)\times \mathbb{R}  \times \mathbb{T})}
\to 0, \text{ as } n\to \infty.
 \end{align*}
As a consequence of Theorem \ref{th2.3}, we see for $n$ large enough,
\begin{align*}
\|u_n\|_{ L_{t,x}^6 H_y^{1-\epsilon_0}((-\infty, 0) \times \mathbb{R}  \times \mathbb{T})} \le 2 \delta_0 < \infty,
\end{align*}
which contradicts \eqref{eq3.13n0}.
The case $t_n \to \infty$ is similar.

{ \bf Case 2.}  $\sup\limits_{k}  \limsup\limits_{n\to \infty}  S\left(  \frac1{(\lambda_n^k)^\frac{1}2} e^{ix\xi_n^k} \left( e^{it_n^k \Delta_{\mathbb{R}}} P_n^k \phi^k\right) \left( \frac{x-x_n^k}{\lambda^k_n}, y\right)\right) \le L_{max} - 2\delta$ for some $\delta > 0$.

We observe that in this case, for each finite $K\le K^*$, we have
\begin{align*}
S\left(  \frac1{(\lambda_n^k)^\frac{1}2} e^{ix\xi_n^k} \left( e^{it_n^k \Delta_{\mathbb{R}}} P_n^k \phi^k\right) \left( \frac{x-x_n^k}{\lambda^k_n}, y\right)\right)\le L_{max}- \delta,
\end{align*}
for all $1 \le k \le K$ and $n$ sufficiently large, by the definition of $L_{max}$, there exist global solution $v_n^k$ to
\begin{equation*}
\begin{cases}
i\partial_t v_n^k + \Delta_{\mathbb{R} \times \mathbb{T}}  v_n^k = |v_n^k|^4 v_n^k,\\
v_n^k(0,x,y) =  \frac1{(\lambda_n^k)^\frac12}   e^{ix\xi_n^k}   \left(e^{it_n^k \Delta_{\mathbb{R}   }}   P_n^k \phi^k\right)
\left(\frac{x-x_n^k}{\lambda_n^j},y\right),
\end{cases}
\end{equation*}
satisfying
$ \|v_n^k \|_{L_{t,x}^6 H_y^{1-\epsilon_0}} \lesssim \Lambda(L_{max} - \delta) < \infty$.
We can use
\begin{align*}
\| v_n^k\|_{L_{t,x}^6 H_y^{1-\epsilon_0} }^2 &  \lesssim_{L_{max},\delta} S(v_n^k(0), \text{ for $S(v_n^k(0)) \le \eta_0$, }
\end{align*}
where $\eta_0$ denotes the small data threshold in the small data scattering theorem, together with our bounds on the space-time norms of $v_n^k$ and the finiteness of $L_{max}$ to deduce
\begin{align}\label{eq6.3}
\| v_n^k\|_{L_{t,x}^6 H_y^{1-\epsilon_0} }^2 &  \lesssim_{L_{max},\delta}
S\left(  \frac1{(\lambda_n^k)^\frac{1}2} e^{ix\xi_n^k} \left( e^{it_n^k \Delta_{\mathbb{R}}} P_n^k \phi^k\right) \left( \frac{x-x_n^k}{\lambda^k_n}, y\right)\right) \lesssim_{L_{max}, \delta} 1.
\end{align}
Let
\begin{align*}
u_n^K = \sum\limits_{j=1}^K v_n^k + e^{it\Delta_{\mathbb{R} \times \mathbb{T}}} w_n^K.
\end{align*}
Then we have $u_n^K(0) = u_n(0)$. We claim that for sufficiently large $K$ and $n$, $u_n^K$ is an approximate solution to $u_n$ in the sense of the Theorem \ref{le2.6}. Then we have the finiteness of the $L_{t,x}^6 H_y^{1-\epsilon_0} $ norm of $u_n$, which contradicts with \eqref{eq3.13n0}.

To verify the claim, we only need to check that $u_n^K$ satisfies the following properties:

$(i)$ $\limsup\limits_{n\to \infty} \left\|u_n^K \right\|_{ L_t^6 L_x^6 H_y^{1-\epsilon_0} } \lesssim_{L_{max},\delta} 1$, uniformly in $K$;

$(ii)$ $  \limsup\limits_{n\to \infty} \|e_n^K \|_{ L_{t,x}^\frac65 H_y^{1-\epsilon_0} } \to 0$, as $ K \to K^* $, where $e_n^K = (i\partial_t + \Delta_{\mathbb{R} \times \mathbb{T}})u_n^K - |u_n^K|^4  u_n^K$.

The verification of $(i)$ relies on the asymptotic decoupling of the nonlinear profiles $v_n^j$, which we record in the following lemma.
Similarly to the proof in \cite{HP} to deal with the quintic nonlinear Schr\"odinger equation on $\mathbb{R}\times \mathbb{T}^2$,
we can obtain the following lemma from Theorem \ref{pr5.9}. We also refer to \cite{CMZ} for similar argument.
\begin{lemma}[Decoupling of nonlinear profiles]\label{le6.3}
Let $v_n^j$ be the nonlinear solutions defined above, then for $j\ne k$,
\begin{align}
&  \left\|\langle \nabla_y  \rangle^{1-\epsilon_0}  v_n^j \cdot  \langle \nabla_y
 \rangle^{1-\epsilon_0} v_n^k   \right\|_{L_{t,x}^3 L_{y}^1} \to 0, \label{eq7.29} \\
& \left \|v_n^k \cdot \langle \nabla_y  \rangle^{1-\epsilon_0} v_n^j\right\|_{L_{t,x}^3 L_{y}^2} \to 0, \label{eq5.8v49}
 \text{ as } n\to \infty.
\end{align}
\end{lemma}
\begin{proof}
We only prove the asymptotically in \eqref{eq7.29}, as \eqref{eq5.8v49} can be proved similarly.
By Theorem \ref{pr5.9}, we only need to show
\begin{align*}
& \left\|\langle \nabla_y \rangle^{1- \epsilon_0} w_n^k \cdot \langle \nabla_y \rangle^{1- \epsilon_0} w_n^{k'} \right\|_{L_{t,x}^3 L_y^1} \to 0, \text{ as } n\to \infty.
\end{align*}
where $w_n^k$ and $w_n^{k'}$ are the approximate solution of $v_n^k$ and $v_n^{k'}$ in Theorem \ref{pr5.9}, respectively.
\begin{align*}
& \left\|\langle \nabla_y \rangle^{1- \epsilon_0} w_n^k \cdot \langle \nabla_y \rangle^{1- \epsilon_0} w_n^{k'} \right\|_{L_{t,x}^3 L_y^1} \\
= & \bigg\|e^{-i(t-t_n^k) |\xi_n^k|^2} e^{ix \xi_n^k} \langle \nabla_y \rangle^{1- \epsilon_0} \sum\limits_{j \in \mathbb{Z}} \frac1{(\lambda_n^k)^\frac12 } e^{-it|j|^2} e^{iy j} v_j\left( \frac{t}{(\lambda_n^k)^2 } + t_n^k, \frac{x- x_n^k - 2 \xi_n^k(t-t_n^k)}{\lambda^k_n}\right) \\
& \quad \cdot
e^{-i(t-t_n^{k'})|\xi_n^{k'}|^2} e^{ix\xi_n^{k'}} \langle \nabla_y \rangle^{1- \epsilon_0} \sum\limits_{j \in \mathbb{Z}} \frac1{(\lambda_n^{k'})^\frac12} e^{-it|j|^2} e^{iyj} v_j\left(\frac{t}{(\lambda_n^{k'})^2} + t_n^{k'}, \frac{x - x_n^{k'} - 2 \xi_n^{k'} (t-t_n^{k'})}{\lambda_n^{k'}} \right) \bigg\|_{L_{t,x}^3 L_y^1}\\
= & \bigg\| \langle \nabla_y\rangle^{1- \epsilon_0} \sum\limits_{j \in \mathbb{Z}} \frac1{(\lambda_n^k)^\frac12} e^{-it|j|^2} e^{iyj} v_j\left(\frac{t}{(\lambda_n^k)^2 } + t_n^k, \frac{x - x_n^k - 2 \xi_n^k(t-t_n^k)}{\lambda_n^k}\right) \\
& \quad \cdot \langle \nabla_y \rangle^{1- \epsilon_0} \sum\limits_{j\in \mathbb{Z}} \frac1{(\lambda_n^{k'})^\frac12} e^{-it|j|^2} e^{iy j} v_j \left( \frac{t}{(\lambda_n^{k'})^2 } + t_n^{k'}, \frac{x- x_n^{k'}  - 2 \xi_n^{k'} (t - t_n^{k'})}{\lambda_n^{k'}} \right) \bigg\|_{L_{t,x}^3 L_y^1}\\
= & \frac1{(\lambda_n^k \lambda_n^{k'})^\frac12} \bigg\| \langle \nabla_y \rangle^{1- \epsilon_0} \sum\limits_{j\in \mathbb{Z}} e^{-it|j|^2} e^{iyj} v_j\left( \frac{t}{(\lambda_n^k)^2} + t_n^k, \frac{x - x_n^k - 2 \xi_n^k(t-t_n^k)}{\lambda_n^k} \right)\\
& \quad \cdot \langle \nabla_y \rangle^{1- \epsilon_0} \sum\limits_{j \in \mathbb{Z}} e^{-it|j|^2} e^{iy j} v_j \left( \frac{t}{(\lambda_n^{k'})^2} + t_n^{k'}, \frac{x - x_n^{k'} - 2\xi_n^{k'}(t- t_n^{k'})}{\lambda_n^{k'}}\right) \bigg\|_{L_{t,x}^3 L_y^1}.
\end{align*}
We can assume $v_j \in C_0^\infty(\mathbb{R} \times \mathbb{R})$, then
\begin{align*}
supp v_j \left( \frac{t}{(\lambda_n^k)^2 } + t_n^k, \frac{x - x_n^k - 2\xi_n^k( t-t_n^k)}{\lambda_n^k}\right) & \subseteq \left\{ (t,x): \left|\frac{t}{(\lambda_n^k)^2} + t_n^k \right| \le T, \left|\frac{x - x_n^k - 2\xi_n^k(t -t_n^k)}{\lambda_n^k} \right| \le R\right\} \\
& = \left\{ |t+ (\lambda_n^k)^2 t_n^k | \le (\lambda_n^k)^2 T, | x - x_n^k - 2\xi_n^k(t- t_n^k)| \le \lambda_n^k R\right\},
\end{align*}
where $supp v_j \subset (- T, T) \times (- R, R)$.
We see
\begin{align*}
& \frac1{(\lambda_n^k \lambda_n^{k'})^\frac12} \bigg\| \langle\nabla_y \rangle^{1- \epsilon_0} \sum\limits_{j \in \mathbb{Z}} e^{-it|j|^2} e^{iyj} v_j \left( \frac{t}{(\lambda_n^k)^2} + t_n^k, \frac{x - x_n^k - 2\xi_n^k(t -t_n^k)}{\lambda_n^k}\right) \\
& \quad \cdot \langle \nabla_y \rangle^{1- \epsilon_0} \sum\limits_{j \in \mathbb{Z}} e^{-it |j|^2} e^{iy j} v_j \left( \frac{t}{(\lambda_n^{k'})^2 } + t_n^{k'}, \frac{x - x_n^{k'} - 2\xi_n^{k'} ( t- t_n^{k'})}{\lambda_n^{k'}}\right) \bigg\|_{L_{t,x}^3 L_y^1\left( \substack{ |t  +(\lambda_n^k)^2 t_n^k| \le (\lambda_n^k)^2 T,\ |t + (\lambda_n^{k'})^2 t_n^{k'} | \le (\lambda_n^{k'})^2 T,\\ | x - x_n^k - 2 \xi_n^k (t-t_n^k) | \le \lambda_n^k R,\ | x- x_n^{k'} - 2\xi_n^{k'}(t -t_n^{k'})| \le \lambda_n^{k'} R}\right)}\\
& \lesssim \frac1{(\lambda_n^k \lambda_n^{k'})^\frac12 } \bigg\| \langle \nabla_y \rangle^{1- \epsilon_0} \sum \limits_{j \in \mathbb{Z}} e^{-it|j|^2} e^{iy j} v_j\left(\frac{t}{(\lambda_n^k)^2} +t_n^k, \frac{x - x_n^k - 2\xi_n^k(t- t_n^k)}{\lambda_n^k}\right) \bigg\|_{L_{t,x}^6 L_y^2} \\
& \quad \cdot \left\| \langle \nabla_y \rangle^{1  - \epsilon_0} \sum\limits_{j \in \mathbb{Z}} e^{-it |j|^2} e^{iyj} v_j \left( \frac{t}{(\lambda_n^{k'})^2} + t_n^{k'} , \frac{x - x_n^{k'} - 2\xi_n^{k'}(t -t_n^{k'})}{ \lambda_n^{k'}} \right)\right\|_{L_{t,x}^6 L_y^2\left( \substack{ |t  +(\lambda_n^k)^2 t_n^k| \le (\lambda_n^k)^2 T,\ |t + (\lambda_n^{k'})^2 t_n^{k'} | \le (\lambda_n^{k'})^2 T,\\ | x - x_n^k - 2 \xi_n^k (t-t_n^k) | \le \lambda_n^k R,\ | x- x_n^{k'} - 2\xi_n^{k'}(t -t_n^{k'})| \le \lambda_n^{k'} R}\right)}.
\end{align*}
Let
\begin{align*}
\Lambda_n^k = \left\{(t,x) : \left|( \lambda_n^k)^{-2} t + t_n^k  \right| + \left| \frac{x - x_n^k - 2\xi_n^k(t -t_n^k)}{\lambda_n^k} \right| \le R \right\},\\
\Lambda_n^{k'} = \left\{(t,x) : \left|( \lambda_n^{k'})^{-2} t + t_n^{k'} \right| + \left| \frac{x - x_n^{k'}  - 2\xi_n^{k'}(t -t_n^{k'})}{\lambda_n^{k'}} \right| \le R \right\},
\end{align*}
we see for
\begin{align*}
\tilde{w}_n^k (t,x,y) = \frac1{(\lambda_n^k)^\frac12} \sum\limits_{j \in \mathbb{Z}} e^{-it |j|^2} e^{iyj} v_j\left( \frac{t}{(\lambda_n^k)^2} + t_n^k, \frac{x - x_n^k - 2\xi_n^k(t -t_n^k)}{\lambda_n^k} \right),
\end{align*}
$supp \tilde{w}_n^k \subset \Lambda_n^k, \ supp \tilde{w}_n^{k'} \subset \Lambda_n^{k'}$, then
\begin{align*}
& \left\| \langle \nabla_y \rangle^{1- \epsilon_0} \tilde{w}_n^k \cdot \langle \nabla_y \rangle^{1- \epsilon_0} \tilde{w}_n^{k'} \right\|_{L_{t,x}^3 L_y^1(\mathbb{R} \times \mathbb{R} \times \mathbb{T})}\\
\lesssim & \left\| \langle \nabla_y \rangle^{1- \epsilon_0} \tilde{w}_n^k \cdot \langle \nabla_y \rangle^{1- \epsilon_0} \tilde{w}_n^{k'} \right\|_{L_{t,x}^3 L_y^1(\mathbb{R} \times \mathbb{R} \times \mathbb{T} \setminus \Lambda_n^k \times \mathbb{T})}
+ \left\| \langle \nabla_y \rangle^{1- \epsilon_0} \tilde{w}_n^k \cdot \langle \nabla_y \rangle^{1- \epsilon_0} \tilde{w}_n^{k'} \right\|_{L_{t,x}^3 L_y^1(\mathbb{R} \times \mathbb{R} \times \mathbb{T} \setminus \Lambda_n^{k'} \times \mathbb{T})}\\
& +  \left\| \langle \nabla_y \rangle^{1- \epsilon_0} \tilde{w}_n^k \cdot \langle \nabla_y \rangle^{1- \epsilon_0} \tilde{w}_n^{k'} \right\|_{L_{t,x}^3 L_y^1( ( \Lambda_n^k \cap \Lambda_n^{k'})  \times \mathbb{T})} \\
\lesssim & \left\| \langle \nabla_y \rangle^{1- \epsilon_0} \tilde{w}_n^k \right\|_{L_{t,x}^6 L_y^2( \{ \mathbb{R}^2 \setminus \Lambda_n^k \} \times \mathbb{T})}
\left\| \langle \nabla_y \rangle^{1- \epsilon_0} \tilde{w}_n^{k'} \right\|_{L_{t,x}^6 L_y^2}
+ \left\| \langle \nabla_y \rangle^{1- \epsilon_0} \tilde{w}_n^k  \right\|_{L_{t,x}^6 L_y^2} \left\|\langle \nabla_y \rangle^{1- \epsilon_0} \tilde{w}_n^{k'} \right\|_{L_{t,x}^6 L_y^2 (\{ \mathbb{R}^2 \setminus \Lambda_n^{k'} \} \times \mathbb{T})} \\
& + \frac1{(\lambda_n^k \lambda_n^{k'})^\frac12}  area( \Lambda_n^{k} \cap \Lambda_n^{k'})^\frac13.
\end{align*}
We see
\begin{align*}
 \left\|\langle \nabla_y \rangle^{1- \epsilon_0} \tilde{w}_n^k \right\|_{L_{t,x}^6 L_y^2} %\\
= &  \left\| \langle \nabla_y \rangle^{1- \epsilon_0} \frac1{(\lambda_n^k)^\frac12} \sum\limits_{j \in \mathbb{Z}} e^{-it|j|^2} e^{iyj} v_j \left( \frac{t}{(\lambda_n^k)^2 } + t_n^k, \frac{x - x_n^k - 2 \xi_n^k (t - t_n^k)}{\lambda_n^k}\right) \right\|_{L_{t,x}^6 L_y^2}\\
= & \left\| \langle j \rangle^{1- \epsilon_0} \frac1{(\lambda_n^k)^\frac12} e^{-it |j|^2} v_j \left( \frac{t}{(\lambda_n^k)^2 } + t_n^k, \frac{x - x_n^k - 2 \xi_n^k( t- t_n^k)}{\lambda_n^k}\right) \right\|_{L_{t,x}^6 l_j^2}\\
= & \frac1{( \lambda_n^k)^\frac12} \left\| \langle j \rangle^{1- \epsilon_0} v_j \left( \frac{t}{(\lambda_n^k)^2} + t_n^k, \frac{x - x_n^k - 2 \xi_n^k (t - t_n^k)}{\lambda_n^k}\right) \right\|_{L_{t,x}^6 l_j^2}\\
 = & \frac1{(\lambda_n^k)^\frac12} ( \lambda_n^k)^\frac26 ( \lambda_n^k)^\frac16 \|\vec{v} \|_{L_{t,x}^6 h_j^{1- \epsilon_0}} = \|\vec{v} \|_{L_{t,x}^6 h_j^{1 - \epsilon_0}},
\end{align*}
we also note
\begin{align*}
& \left\| \langle \nabla_y \rangle^{1- \epsilon_0} \tilde{w}_n^k \right\|_{L_{t,x}^6 L_y^2( \{ \mathbb{R}^2 \setminus \Lambda_n^k \} \times \mathbb{T})} \\
= & \left\| \langle \nabla_y \rangle^{1- \epsilon_0} \frac1{(\lambda_n^k)^\frac12 } \sum\limits_{j \in \mathbb{Z}} e^{-it|j|^2} e^{iyj} v_j\left( \frac{t}{(\lambda_n^k)^2} + t_n^k, \frac{x - x_n^k - 2 \xi_n^k(t - t_n^k)}{\lambda_n^k}\right) \right\|_{L_{t,x}^6 L_y^2\left( \left\{ \left|\frac{t}{(\lambda_n^k)^2} +t_n^k \right| + \left| \frac{ x - x_n^k - 2 \xi_n^k (t - t_n^k)}{\lambda_n^k} \right| > R\right\} \times \mathbb{T}\right)}\\
= & \left\| \left\| \frac1{(\lambda_n^k)^\frac12} v_j\left( \frac{t}{(\lambda_n^k)^2} + t_n^k, \frac{x - x_n^k - 2\xi_n^k(t - t_n^k)}{\lambda_n^k}\right) \right\|_{h_j^{1- \epsilon_0}} \right\|_{L_{t,x}^6 \left( \left\{ \left|\frac{t}{(\lambda_n^k)^2} +t_n^k \right| + \left| \frac{ x - x_n^k - 2 \xi_n^k (t - t_n^k)}{\lambda_n^k} \right| > R\right\} \times \mathbb{T}\right)}\\
= & \| \vec{v} \|_{L_{t,x}^6 h^{1- \epsilon_0}\left( \{ |t| + |x| \ge R\} \times \mathbb{Z}\right)} \to 0, \text{ as } R \to \infty.
\end{align*}
Hence, we are reduced to prove that
\begin{align*}
\frac1{(\lambda_n^k \lambda_n^{k'})^\frac12} area( \Lambda_n^k \cap \Lambda_n^{k'})^\frac13 \to 0, \text{ as } n \to \infty.
\end{align*}
Note
\begin{align*}
& \quad area ( \Lambda_n^k \cap \Lambda_n^{k'})\\
& = area \left( \left\{ \left|\frac{t}{(\lambda_n^k)^2} + t_n^k \right| + \left| \frac{x - x_n^k - 2 \xi_n^k( t- t_n^k) }{ \lambda_n^k} \right| < R\right\} \cap \left\{ |\frac{t}{(\lambda_n^{k'})^2} + t_n^{k'} | + | \frac{x - x_n^{k'}  - 2 \xi_n^{k'}( t- t_n^{k'} ) }{ \lambda_n^{k'}} | < R\right\}\right)\\
& \le C_R \min\left( (\lambda_n^k)^3, ( \lambda_n^{k'})^3\right),
\end{align*}
hence if $\lim\limits_{n\to \infty} \frac{\lambda_n^k}{\lambda_n^{k'}} $ is either 0 or $\infty$, we see
\begin{align*}
\frac1{(\lambda_n^k \lambda_n^{k'})^\frac12} area( \Lambda_n^{k} \cap \Lambda_n^{k'})^\frac13 \to 0, \text{ as } n \to \infty.
\end{align*}
The same thing happens if
\begin{align*}
\lim\limits_{n\to \infty} \frac{ | ( \lambda_n^k)^2 t_n^k - ( \lambda_n^{k'})^2 t_n^{k'}|}{\lambda_n^k \lambda_n^{k'}}  = \infty.
\end{align*}
Assume then
\begin{align*}
\lim\limits_{n\to \infty} \frac{ | ( \lambda_n^k)^2 t_n^k - ( \lambda_n^{k'})^2 t_n^{k'}|}{\lambda_n^k \lambda_n^{k'}}  < \infty,
\end{align*}
we have
\begin{align*}
\lambda_n^k \lambda_n^{k'} | \xi_n^k - \xi_n^{k'}|^2 + \frac{ | x_n^{k'} - x_n^k - 2t_n^{k'} (\lambda_n^{k'})^2 ( \xi_n^{k'} - \xi_n^k)|^2}{\lambda_n^k \lambda_n^{k'}} \to \infty, \text{ as } n \to \infty.
\end{align*}
We see
\begin{align*}
& area( \Lambda_n^k \cap \Lambda_n^{k'}) \\
= & area \left( \left\{ \left|\frac{t}{(\lambda_n^k)^2} + t_n^k \right| + \left| \frac{x - x_n^k - 2\xi_n^k( t- t_n^k)}{\lambda_n^k} \right| < R\right\} \bigcap \left\{ \left| \frac{t}{(\lambda_n^{k'})^2} + t_n^{k'}\right| + \left| \frac{ x- x_n^{k'} - 2\xi_n^{k'} (t - t_n^{k'})}{\lambda_n^{k'}} \right| < R\right\}\right).
\end{align*}
Let
\begin{align*}
v = x - x_n^k - 2\xi_n^k( t- t_n^k), \quad w = x - x_n^{k'} - 2\xi_n^{k'}(t - t_n^{k'}),
\end{align*}
we have
\begin{align*}
area( \Lambda_n^k \cap \Lambda_n^{k'})
= \int_{(t,x) \in \Lambda_n^k \cap \Lambda_n^{k'}} \,\mathrm{d}t \mathrm{d}x
\lesssim \int_{| v| \le \lambda_n^k R, |w| \le \lambda_n^{k'} R} \,\mathrm{d}t \mathrm{d}x,
\end{align*}
since
\begin{align*}
\frac{\partial(v,w)}{\partial(t,x)} =
\left(
\begin{array}{ll}
- 2\xi_n^k \ 1\\
- 2\xi_n^{k'} \ 1
\end{array}
\right),
\end{align*}
we then have
\begin{align*}
area( \Lambda_n^k \cap \Lambda_n^{k'})
\lesssim \int_{|v| \le \lambda_n^k R, |w| \le \lambda_n^{k'} R} \left| \frac{\partial(t,x)}{\partial(v,w)} \right| \,\mathrm{d}v \mathrm{d}w
\lesssim \int_{|v| \le \lambda_n^k R, |w| \le \lambda_n^{k'} R} \frac1{|2(\xi_n^k - \xi_n^{k'})|} \,\mathrm{d}v \mathrm{d}w
\lesssim \frac{\lambda_n^k \lambda_n^{k'} R^2}{| \xi_n^k - \xi_n^{k'}|}.
\end{align*}
We see that by $\lambda_n^k \sim \lambda_n^{k'}$, and $\lambda_n^k \lambda_n^{k'} |\xi_n^k - \xi_n^{k'}|^2 \to \infty$,
\begin{align*}
\frac1{(\lambda_n^k \lambda_n^{k'})^\frac12} area( \Lambda_n^k \cap \Lambda_n^{k'})^\frac13
\lesssim \frac1{\lambda_n^k} \left( \frac{(\lambda_n^k)^2 R^2}{|\xi_n^k - \xi_n^{k'}|}\right)^\frac13 \sim \frac{R^\frac23}{ (\lambda_n^k)^\frac13 |\xi_n^k - \xi_n^{k'}|^\frac13} \to 0, \text{ as } n\to \infty.
\end{align*}
We are now only need to consider the case
\begin{align*}
\frac{ | x_n^{k'} - x_n^k - 2 t_n^{k'} ( \lambda_n^{k'})^2 ( \xi_n^{k'} - \xi_n^k)|^2}{\lambda_n^k \lambda_n^{k'}} \to \infty, \text{ as } n \to \infty.
\end{align*}
For $(t,x) \in \Lambda_n^k \cap \Lambda_n^{k'}$,
\begin{align*}
\left|\frac{t}{(\lambda_n^k)^2} + t_n^k \right| + \left|\frac{x - x_n^k - 2\xi_n^k t}{\lambda_n^k}\right| \le R,\quad
\left|\frac{t}{(\lambda_n^{k'})^2} + t_n^{k'} \right| + \left|\frac{x - x_n^{k'} - 2\xi_n^{k'} t}{ \lambda_n^{k'} }  \right| \le R,
\end{align*}
we see
\begin{align*}
R \ge \frac{ |x - x_n^k - 2\xi_n^k t|}{\lambda_n^k}
& = \frac{ |x_n^{k'} - x_n^k + x - x_n^{k'} - 2\xi_n^k t|}{\lambda_n^k}\\
& \ge \frac{|x_n^{k'} - x_n^k - 2t_n^{k'}(\lambda_n^{k'})^2( \xi_n^{k'} - \xi_n^k)}{ \lambda_n^k}
- \frac{ | 2 t_n^{k'}(\lambda_n^{k'})^2 ( \xi_n^{k'} - \xi_n^k ) + x - x_n^{k'} - 2\xi_n^k t|}{\lambda_n^k}.
\end{align*}
We note
\begin{align*}
\frac{ | x- x_n^{k'} - 2\xi_n^k t + 2 ( \xi_n^{k'} -  \xi_n^k) t_n^{k'} (\lambda_n^{k'})^2 |}{\lambda_n^k}
& =  \frac{ | x- x_n^{k'} - 2\xi_n^{k'} t + 2 ( \xi_n^{k'} - \xi_n^k)  t + 2 ( \xi_n^{k'} - \xi_n^k) t_n^{k'} (\lambda_n^{k'})^2|}{ \lambda_n^k}\\
& \le \frac{ | x- x_n^{k'} - 2\xi_n^{k'} t |}{\lambda_n^k} + \frac{ 2 |( \xi_n^{k'} - \xi_n^k)t + ( \xi_n^{k'} - \xi_n^{k}) t_n^{k'} ( \lambda_n^{k'})^2 |}{ \lambda_n^k} \\
& \le \frac{ \lambda_n^{k'}}{\lambda_n^k} R + \frac{ 2 |( \xi_n^{k'} - \xi_n^k)( t+ (\lambda_n^{k'})^2 t_n^{k'} )}{ \lambda_n^k} \\
& \le \frac{ \lambda_n^{k'}}{\lambda_n^k} R + \frac{ 2 |( \xi_n^{k'} - \xi_n^k) ( \lambda_n^{k'})^2 |}{ \lambda_n^k} R < \infty.
\end{align*}
Therefore, $\Lambda_n^k \cap \Lambda_n^{k'}  = \varnothing$, when $n$ large enough.

\end{proof}
Let us verify claim $(i)$ above.
By \eqref{eq6.3} and \eqref{eq7.29}, we have
\begin{align*}
     \left\|\sum\limits_{k=1}^K v_n^k \right\|_{  L_{t,x}^6 H_y^{1-\epsilon_0}}^6
 \lesssim
  & \left(\sum\limits_{k =1 }^K \left\|    v_n^k     \right\|_{L_{t,x}^6 H_y^{1-\epsilon_0}}^2 +
 \sum\limits_{j\ne k} \left\|    \langle \nabla_y  \rangle^{1-\epsilon_0} v_n^j \cdot  \langle \nabla_y
 \rangle^{1-\epsilon_0} v_n^k     \right\|_{L_{t,x}^3 L_{y}^1}\right)^3\\
 \lesssim &  \left(\sum\limits_{k=1}^K
 S\left(  \frac1{(\lambda_n^k)^\frac{1}2} e^{ix\xi_n^k} \left( e^{it_n^k \Delta_{\mathbb{R}}} P_n^k \phi^k\right) \left( \frac{x-x_n^k}{\lambda^k_n}, y\right)\right)  + o_K(1)\right)^3,
\end{align*}
for $K$ large enough.
The mass and energy decoupling implies
\begin{align*}
 \sum\limits_{k=1}^K
 S \left(  \frac1{(\lambda_n^k)^\frac{1}2} e^{ix\xi_n^k} \left( e^{it_n^k \Delta_{\mathbb{R}}} P_n^k \phi^k\right) \left( \frac{x-x_n^k}{\lambda^k_n}, y\right)\right)
\le L_{max},
\end{align*}
 together with \eqref{eq5.3}, we obtain
\begin{align}\label{eq6.5}
\lim\limits_{K \to K^*} \limsup\limits_{n\to \infty} \|  u_n^K \|_{L_t^6 L_x^{6} H_y^{1-\epsilon_0}} \lesssim_{L_{max}, \delta} 1.
\end{align}
It remains to check property $(ii)$ above, by the definition of $u_n^K$, we decompose
\begin{align*}
e_n^K & = (i\partial_t + \Delta_{\mathbb{R} \times \mathbb{T}} )u_n^K - |u_n^K|^4 u_n^K \\
      & = \sum\limits_{k=1}^K  |v_n^k|^4 v_n^k - \left|\sum\limits_{k =1}^K v_n^k\right|^4 \sum\limits_{k=1}^K  v_n^k
      + |u_n^K- e^{it\Delta_{\mathbb{R} \times \mathbb{T}}} w_n^K |^4 (u_n^K - e^{it\Delta_{\mathbb{R} \times \mathbb{T}}} w_n^K) - |u_n^K |^4 u_n^K.
\end{align*}
First consider
\begin{align*}
\sum\limits_{k=1}^K  |v_n^k|^4 v_n^k - \left|\sum\limits_{k=1}^K v_n^k \right|^4 \sum\limits_{k=1}^K v_n^k.
\end{align*}
Thus by the fractional chain rule, Minkowski, H\"older, Sobolev, \eqref{eq5.8v49} and \eqref{eq6.3},
\begin{align}\label{eq4.832}
&\quad  \left\| \sum\limits_{k=1}^K |v_n^k|^4 v_n^k - \left|\sum\limits_{k=1}^K  v_n^k \right|^4 \sum\limits_{k =1}^K v_n^k \right\|_{
L_{t,x}^\frac65 H_y^{1-\epsilon_0}}\\
& \lesssim \sum_{k \ne k'} \big \|\langle \nabla_y \rangle^{1-\epsilon_0} (v_n^k (v_n^{k'})^4) \big\|_{L_{t,x}^\frac65 L_y^2}           \notag \\
& \lesssim \sum_{k \ne k'} \big( \|v_n^{k'} \langle \nabla_y \rangle^{1-\epsilon_0} v_n^k\|_{L_{t,x}^3 L_y^2} \|v_n^{k'}\|_{L_{t,x}^6 L_y^\infty}^3
+ \|v_n^k \langle \nabla_y \rangle^{1-\epsilon_0} v_n^{k'} \|_{L_{t,x}^3 L_y^2} \|v_n^{k'}\|_{L_{t,x}^6 L_y^\infty}^3\big)       \notag \\
& \sim \sum_{k \ne k'} \|v_{k'} \langle \nabla_y \rangle v_n^k \|_{L_{t,x}^3 L_y^2} \big(\|v_n^{k'} \|_{L_{t,x}^6 L_y^\infty}^3 + \|v_n^k \|_{L_{t,x}^6 L_y^\infty}^3\big)\notag \\
& \lesssim  \sum_{k \ne k'} \left\|v_n^j \langle \nabla_y \rangle^{1-\epsilon_0} v_n^k \right\|_{L_{t,x}^3 L_y^2} \big(\|v_n^{k'} \|_{L_{t,x}^6 H_y^{1-\epsilon_0}}^3
+ \|v_n^k \|_{L_{t,x}^6 H_y^{1-\epsilon_0}}^3\big)
\lesssim o_K(1), \text{ as } n\to \infty.\notag
\end{align}
We now estimate $\left|u_n^K   - e^{it\Delta_{\mathbb{R}  \times \mathbb{T}}} w_n^K \right|^4 (u_n^K   - e^{it\Delta_{\mathbb{R}  \times \mathbb{T}}} w_n^K ) - |u_n^K |^4 u_n^K $.
By the fractional chain rule, H\"older, Sobolev, we have
\begin{align*}
& \left\|\left|u_n^K - e^{it\Delta_{\mathbb{R}  \times \mathbb{T}}} w_n^K \right|^4 \left(u_n^J  - e^{it\Delta_{\mathbb{R}  \times \mathbb{T}}} w_n^K \right) - \left|u_n^K \right|^4 u_n^K \right\|_{  L_{t,x}^\frac65    H_y^{1-\epsilon_0} }\\
\lesssim & \left\|u_n^K \right\|_{L_{t,x}^6 H_y^{1-\epsilon_0}}^4 \left\|e^{it\Delta_{\mathbb{R}\times \mathbb{T}}} w_n^K \right\|_{L_{t,x}^6 H_y^{1-\epsilon_0}}
+ \left\|u_n^K \right\|_{L_{t,x}^6 H_y^{1-\epsilon_0}}^3 \left\|e^{it\Delta_{\mathbb{R}\times \mathbb{T}}} w_n^K \right\|_{L_{t,x}^6 H_y^{1-\epsilon_0}}^2
+ \left\|u_n^K \right\|_{L_{t,x}^6 H_y^{1-\epsilon_0}}^2 \left\|e^{it\Delta_{\mathbb{R}\times \mathbb{T}}} w_n^K \right\|_{L_{t,x}^6 H_y^{1-\epsilon_0}}^3\\
& \ +
\left\|u_n^K \right\|_{L_{t,x}^6 H_y^{1-\epsilon_0}}  \left\|e^{it\Delta_{\mathbb{R}\times \mathbb{T}}} w_n^K  \right\|_{L_{t,x}^6 H_y^{1-\epsilon_0}}^4
+  \left\|e^{it\Delta_{\mathbb{R}\times \mathbb{T}}} w_n^K \right\|_{L_{t,x}^6 H_y^{1-\epsilon_0}}^5\\
 \lesssim & \left( \left\|u_n^K \right\|_{L_{t,x}^6 H_y^{1-\epsilon_0}}^4 + \left\|e^{it\Delta_{\mathbb{R} \times \mathbb{T}}} w_n^K  \right\|_{L_{t,x}^6 H_y^{1-\epsilon_0}}^4 \right)
\left\|e^{it\Delta_{\mathbb{R} \times \mathbb{T}}} w_n^K \right\|_{L_{t,x}^6 H_y^{1-\epsilon_0}  }.
 \end{align*}
Using \eqref{eq6.5}, and the decay property \eqref{eq6.1}, we get
\begin{align*}
\limsup\limits_{n\to \infty} \left\|\left|u_n^K - e^{it\Delta_{\mathbb{R}  \times \mathbb{T}}} w_n^K \right|^4 \left( u_n^K - e^{it\Delta_{\mathbb{R}  \times \mathbb{T}}} w_n^K  \right) - \left|u_n^K \right|^4 u_n^K  \right\|_{  L_{t,x}^6 H_y^{1-\epsilon_0}} \to 0, \text{ as }  K\to K^*.
\end{align*}
\end{proof}
Arguing as in \cite{CMZ}, the proof of Proposition \ref{pr7.1} implies the following result:
\begin{theorem}[Existence of the almost-periodic solution] \label{co4.727}
Assume that $L_{max} < \infty$, then there exists $u_c\in C_{t }^0 H_{x,y}^1(\mathbb{R}\times \mathbb{R}  \times \mathbb{T})$ solving \eqref{eq1.1} satisfying
\begin{align*}
S(u_c(t))= L_{max},\
\|u_c\|_{L_{t,x}^6 H_y^{1-\epsilon_0}(\mathbb{R}\times \mathbb{R}  \times \mathbb{T})} = \infty.
\end{align*}
Furthermore, $u_c$ is almost periodic in the sense that $\forall \eta > 0$, there is $C(\eta) > 0$ such that
\begin{align}\label{eq4.1735}
\int_{|x+x(t)| \ge C(\eta)} \|   u_c(t,x,y)\|_{H_y^1(\mathbb{T})}^2 \mathrm{d}x   < \eta
\end{align}
for all $t\in \mathbb{R}$, where $ {x}: \mathbb{R}\to \mathbb{R} $ is a Lipschitz function with $\sup\limits_{t\in \mathbb{R}} | {x'}(t)| \lesssim 1$.
\end{theorem}

\subsection{Extinction of the critical element}
In this section, we will exclude the almost-periodic solution in Theorem \ref{co4.727} by using the bilinear virial action as in \cite{PV}.
\begin{proposition}[Non-existence of the almost-periodic solution]\label{pr7.3}
The almost-periodic solution $u_c$ in Theorem \ref{co4.727} does not exist.
\end{proposition}
\begin{proof}
We define the bilinear virial
\begin{align*}
I(t) = \iint_{x > \tilde{x}} \iint_{\mathbb{T}^2} (x - \tilde{x}) |u_c(t,x,y)|^2 |u_c(t,\tilde{x}, \tilde{y})|^2 \,\mathrm{d}x \mathrm{d}y \mathrm{d}\tilde{x} \mathrm{d}\tilde{y}.
\end{align*}
By direct computation, we have
\begin{align*}
I'(t) = & 4  \iint_{x > \tilde{x}} \iint_{\mathbb{T}^2} |u_c(t, \tilde{x}, \tilde{y})|^2 \Im(\bar{u}_c \partial_x u_c)(t,x,y) \,\mathrm{d}x \mathrm{d}y \mathrm{d}\tilde{x} \mathrm{d}\tilde{y}.
\end{align*}
We see
\begin{align}\label{eq5.032}
|I'(t)|  & \lesssim \Big| \iint_{\mathbb{T}^2} \iint_{x> \tilde{x}} |u_c(t, \tilde{x}, \tilde{y})|^2 \Im( \bar{u}_c \partial_x u_c) (t,x,y) \,\mathrm{d}x \mathrm{d}y \mathrm{d}\tilde{x} \mathrm{d} \tilde{y} \Big|
\lesssim \|u_c\|_{L_{x,y}^2}^3 \|\nabla_x u_c \|_{L_{x,y}^2} \lesssim 1.
\end{align}
On the other hand,
\begin{align*}
\frac12 I''(t) & \ge \frac43 \int_{\mathbb{R}} \iint_{\mathbb{T}^2} |u_c(t, \tilde{x}, \tilde{y})|^2 |u_c(t, \tilde{x}, y)|^6 \,\mathrm{d}\tilde{x} \mathrm{d} y \mathrm{d} \tilde{y}
+ \int_{\mathbb{R}}  \left( \int_{\mathbb{T}} \partial_{\tilde{x}} (|u_c(t,\tilde{x},y)|^2 ) \,\mathrm{d}y \right)^2 \,\mathrm{d}\tilde{x}\\
& = \frac43 \int_{\mathbb{R}} \iint_{\mathbb{T}^2} |u_c(t, \tilde{x}, \tilde{y})|^2 |u_c(t, \tilde{x}, y)|^6 \,\mathrm{d}\tilde{x} \mathrm{d} y \mathrm{d} \tilde{y}
+ \int_{\mathbb{R}} \left( \partial_{\tilde{x}} \left( \|u_c(t,\tilde{x}, y)\|_{L_y^2}^2 \right) \right)^2 \,\mathrm{d}\tilde{x}.
\end{align*}
Therefore,
\begin{align*}
\int_{-T_0}^{T_0} I''(t) \,\mathrm{d}t \gtrsim \int_{-T_0}^{T_0} \int_{\mathbb{R}} \left( \partial_x \left( \|u_c(t,x,y)\|_{L_y^2}^2 \right) \right)^2 \,\mathrm{d}x \mathrm{d}t,
\end{align*}
we note
\begin{align*}
\int_{\mathbb{R}} \left( \partial_x \left( \|u_c(t,x,y)\|_{L_y^2}^2 \right) \right)^2 \,\mathrm{d}x \cdot C\left(\frac{m_0}{1000}\right)^\frac32
& \gtrsim \left\| \left \|u_c(t,x,y) \right\|_{L_y^2}^2 \right\|_{\dot{C}_x^\frac12} \cdot C\left(\frac{m_0}{1000}\right)^\frac32 \\
& \gtrsim \int_{ |x + x(t)| \le C\left( \frac{m_0}{1000}\right) } \|u_c(t,x,y)\|_{L_y^2}^2 \,\mathrm{d}x,
\end{align*}
where $m_0  = \mathcal{M}(u_c)$.
Thus
\begin{align}\label{eq5.132}
\int_{| x+ x(t)| \le C\left( \frac{m_0}{1000}\right) } \|u_c(t,x,y) \|_{L_y^2}^2 \,\mathrm{d}x \lesssim C\left( \frac{m_0}{1000}\right)^\frac32 \left\| \partial_x \left( \|u_c(t,x,y) \|_{L_y^2}^2 \right) \right\|_{L_x^2}^2.
\end{align}
By \eqref{eq4.1735} and conservation of mass, we have
\begin{align}\label{eq5.232}
\frac{m_0 }2 \le \int_{|x + x(t)|\le C\big(\frac{m_0 }{100}\big)} \|  u_c(t,x,y)\|_{L_y^2}^2 \,\mathrm{d}x.
\end{align}
Therefore, $\forall \ T_0 > 0$, by \eqref{eq5.232}, \eqref{eq5.132}, \eqref{eq5.032},
\begin{align*}
m_0^2 T_0 =  \int_{-T_0}^{T_0}  m^2_0 \,\mathrm{d}t & \lesssim \int_{-T_0}^{T_0} \left(\int_{|x + x(t)|\le C\big(\frac{m_0 }{100}\big)} \|u_c(t,x,y)\|_{L_y^2}^2
 \,\mathrm{d}x \right)^2 \,\mathrm{d}t    \lesssim C\Big(\frac{m_0 }{100}\Big).
\end{align*}
 Let $T_0 \to \infty$, we obtain a contradiction unless $u_c \equiv 0$, which is impossible due to
$\|u_c\|_{L_{t,x}^6 H_y^{1-\epsilon_0}(\mathbb{R}\times \mathbb{R} \times \mathbb{T} )}   = \infty$.
\end{proof}

\textbf{Acknowledgments.} The authors appreciate Professor Benjamin Dodson for his kind help and useful discussions. Zehua Zhao thanks Professor Z. Hani and Professor B. Pausader for their useful suggestions and comments.
Xing Cheng also thanks Jiqiang Zheng for explaining the idea of B. Dodson's work on the mass-critical nonlinear Schr\"odinger equation.
The authors would also thank David Wood and Michael Gill for polishing the article. This work has been reported by Xing Cheng at the 12th AIMS conference on dynamical systems, differential equations and applications in Taipei.

\end{document}